\documentclass[12pt]{article}
\usepackage[margin=2cm, top=2cm, bottom=2cm]{geometry}
\usepackage[utf8]{inputenc} 
\usepackage[T1]{fontenc}  
\usepackage{url}          
\usepackage{booktabs}      
\usepackage{amsfonts}     
\usepackage{nicefrac}      
\usepackage{microtype}    
\usepackage{mathrsfs}  

\usepackage{amsmath,amssymb,amsthm,xcolor}
\numberwithin{equation}{section}
\usepackage{array}
\usepackage{float}
\newtheorem{theorem}{Theorem}[section]

\newtheorem{remark}[theorem]{Remark}
\newtheorem{lemma}[theorem]{Lemma}
\newtheorem{corollary}[theorem]{Corollary}

\newtheorem{proposition}[theorem]{Proposition}
\newtheorem{definition}[theorem]{Definition}
\newtheorem{fact}[theorem]{Fact}
\usepackage{graphicx}
\usepackage{subcaption}
\usepackage{mdframed}
\usepackage{lipsum}
\usepackage{wrapfig}
\usepackage[hidelinks,hypertexnames=false]{hyperref}
\usepackage{tikz}
\usepackage{enumitem}
\allowdisplaybreaks
\usepackage{epstopdf}
\usepackage{algorithmic}
\usepackage{algorithm}
\usepackage{booktabs}
\usepackage{changepage}
\usepackage{enumitem}
\usepackage{indentfirst}
\usepackage{stmaryrd}
\usepackage{soul}
\usetikzlibrary{calc,arrows.meta} 
\title{On the diameter of subgradient sequences in o-minimal structures}
\author{\large  Lexiao Lai\thanks{\url{lai.lexiao@hku.hk}, Department of Mathematics, The University of Hong Kong, Hong Kong.} \and Mingzhi Song\thanks{\url{songmingzhi123@gmail.com}, Department of Mathematics, The University of Hong Kong, Hong Kong.}}
\date{}

\begin{document}

\maketitle
\vspace*{-5mm}
\begin{center}
    \textbf{Abstract}
    \end{center}
    \vspace*{-4mm}
 \begin{adjustwidth}{0.2in}{0.2in}
~~~~ We study subgradient sequences of locally Lipschitz functions definable in a polynomially bounded o-minimal structure. We show that the diameter of any subgradient sequence is related to the variation in function values, with error terms dominated by a double summation of step sizes. Consequently, we prove that bounded subgradient sequences converge if the step sizes are of order $1/k$. The proof uses Lipschitz $L$-regular stratifications in o-minimal structures to analyze subgradient sequences via their projections onto different strata.

\end{adjustwidth}
\section{Introduction}
Let $f:\mathbb{R}^n\to \mathbb{R}$ be locally Lipschitz. We study \emph{subgradient sequences} $(x_k)_{k\in \mathbb{N}}$ defined by
\begin{equation*}
    x_{k+1} \in x_k - \alpha_k \partial f(x_k)
\end{equation*}
for $k\in \mathbb{N}$, where $x_0\in \mathbb{R}^n$ is arbitrary, $(\alpha_k)_{k\in \mathbb{N}}$ is a sequence of positive scalars (called \emph{step sizes}) that is not summable, and $\partial f:\mathbb{R}^n\rightrightarrows \mathbb{R}^n$ denotes the Clarke subdifferential \cite{clarke1975,clarke1990} of $f$. Subgradient sequences are discretizations of continuous-time subgradient trajectories \cite{bolte2007lojasiewicz}, which are solutions to the differential inclusion $x'\in -\partial f(x)$. Subgradient trajectories can be viewed as generalizations of classical gradient trajectories of smooth functions \cite{absil2006stable,santambrogio2017euclidean}. In the optimization literature, subgradient sequences are realizations of the subgradient method \cite{shor1962application}, which generalizes Cauchy's steepest descent method \cite{cauchy1847methode} to minimize locally Lipschitz functions. The subgradient method and its variants garnered significant attention within the machine learning community recently, due to their success in solving large-scale optimization problems arising from deep learning and artificial intelligence \cite{sutskever2013importance,lecun2015deep,vaswani2017attention}.

Subgradient sequences can behave erratically, even for functions that are $C^\infty$ \cite{palis2012geometric,absil2005convergence,daniilidis2020pathological}. It is for this reason that we assume additional geometric structures of the objective function $f$, that is, the function is definable in o-minimal structures \cite{pillay1986definable,van1996geometric,coste2000introduction}. We defer the definition and discussion of o-minimal structures to Section \ref{sec:o-minimal}. At a high level, o-minimal structures generalize semialgebraic sets \cite{tarski1951decision}, and are families of ``tame'' subsets of $\mathbb{R}^n$ that possess certain finiteness properties. The study of (sub)gradient dynamics for definable functions was initiated in \L{}ojasiewicz's pioneer works \cite{lojasiewicz1963propriete,lojasiewicz1982trajectoires} on (real) analytic functions. It was shown that bounded (continuous-time) gradient trajectories of analytic functions have finite length, a consequence of the gradient inequality (known as the \L{}ojasiewicz gradient inequality \cite{lojasiewicz1958}) of analytic functions. This inequality can be extended to smooth functions definable in arbitrary o-minimal structures \cite{kurdyka1998gradients} and later to nonsmooth functions in \cite{bolte2007lojasiewicz}. Consequently, the convergence of subgradient trajectories in these two settings is also established \cite{kurdyka1998gradients,bolte2010characterizations}.

In contrast to their continuous-time counterparts, bounded subgradient sequences are known to converge only when $f$ is either 1)  definable and differentiable with a locally Lipschitz gradient \cite{absil2005convergence} or 2) convex \cite{alber1998projected}, given that the step sizes $(\alpha_k)_{k\in \mathbb{N}}$ are square summable. For nonconvex nonsmooth functions, recent works proposed to analyze subgradient sequences under the assumption that $f$ is ``path-differentiable'' \cite{davis2020stochastic,bolte2022long,bolte2025inexact}. Path-differentiable functions are functions that satisfy a chain rule when precomposed with any absolutely continuous arc \cite[Definition 5.1]{davis2020stochastic}\cite[Definition 3]{bolte2020conservative}. Locally Lipschitz functions definable in o-minimal structures are path-differentiable \cite[Theorem 5.8]{davis2020stochastic}, as their graphs can be stratified into smooth manifolds. If in addition $f$ satisfies the weak Sard property, i.e., $f$ is constant on connected components of its critical set\footnote{$x\in \mathbb{R}^n$ is a critical point of $f$ if $0\in \partial f(x)$. The collection of all critical points is the critical set.}, then the limit points of any bounded subgradient sequence $(x_k)_{k\in \mathbb{N}}$ are critical points, and the function values $(f(x_k))_{k\in \mathbb{N}}$ converge \cite[Theorem 3.2]{davis2020stochastic}\cite[Theorem 5]{bolte2022long}. It is worth noticing that their approach regards the subgradient sequence as an approximation of subgradient trajectories, inspired by previous works in stochastic approximation \cite{ljung1977analysis,kushner1977general,benaim2005stochastic,duchi2018stochastic}. Drawing tools from the theory of closed measures, one can further study the oscillation of subgradient sequences \cite{bolte2022long}. 

It is natural to wonder whether and when the subgradient sequences with vanishing step sizes will converge. By an example of Ríos-Zertuche \cite[Section 2]{rios2022examples}, subgradient sequences of path-differentiable functions can indeed oscillate, with step sizes $\alpha_k:= 1/(k+1)$. In fact, the constructed ``pathological'' function is Whitney $C^\infty$ stratifiable and satisfies the nonsmooth Łojasiewicz gradient inequality \cite[Proposition 6]{rios2022examples}. It is noteworthy that the subgradient trajectories of the same function converge, due to the nonsmooth \L{}ojasiewicz gradient inequality \cite{bolte2010characterizations}. This highlights the distinct dynamics of subgradient sequences compared to their continuous-time counterparts, emphasizing the need for additional geometric structures to guarantee their convergence.

In this work, we aim to conduct a refined analysis on subgradient sequences of locally Lipschitz functions definable in o-minimal structures. We seek to identify conditions under which the sequence will converge if bounded. Recall that the \emph{diameter} of a set $A\subset \mathbb{R}^n$ is given by $\mathrm{diam}(A):= \sup\{|a-b|:a,b\in A\}$. In our main result (Theorem \ref{thm:diameter}), we estimate the diameter of subgradient sequences when they stay close to a level set of $f$. We show that the diameter is controlled by a difference in function values, up to high-order accumulations of the step sizes. Let $a,b\in \mathbb{N}$ such that $a \le b$, we denote by $\llbracket a,b \rrbracket:= \{a,\ldots,b\}$. Given a sequence $(x_k)_{k\in \mathbb{N}}$, we denote by $x_{\llbracket a,b \rrbracket}:= \{x_a,\ldots,x_b\}$. For a function $g:\mathbb{R}^n\rightarrow \mathbb{R}$ and $v\in \mathbb{R}^n$, we denote by $[g \le v]:=\{x\in \mathbb{R}^n:g(x) \le v\}$ the sublevel set of $g$ with respect to the value $v$. We also define the sign function $\mathrm{sgn}:\mathbb{R}\to \{-1,0,1\}$ that returns $-1$ for negative values, $1$ for positive values, and $0$ for zero. We are now ready to present the main result of this paper.
\begin{theorem}\label{thm:diameter}
    Let $f:\mathbb{R}^n\to \mathbb{R}$ be locally Lipschitz and definable in a polynomially bounded o-minimal structure. For any bounded $X\subset \mathbb{R}^n$, there exist $\bar{\alpha},\beta,\epsilon,\varsigma_1,\varsigma_2>0$ and $\theta\in (0,1)$ such that for any subgradient sequence $(x_k)_{k\in \mathbb{N}}$ with step sizes $0<\alpha_{K-1} \le \cdots \le \alpha_0 \le \bar{\alpha}$ and $x_{\llbracket 0,K \rrbracket} \subset X\cap [ |f| \le \epsilon]$ for some $K\in \mathbb{N}$, we have
        \begin{align*}
        \mathrm{diam}(x_{\llbracket 0,K \rrbracket}) \le~&\varsigma_1\left(\mathrm{sgn}(f(x_0))|f(x_0)|^{1-\theta}-\mathrm{sgn}(f(x_K))|f(x_K)|^{1-\theta}\right) \\
        &+\varsigma_2\left(\alpha_0^\beta+ \sum_{k=0}^{K-1}\alpha_k^{1+\beta} +\left(\sum_{k=0}^{K-1}\alpha_k^{1+\beta}\right)^{1-\theta}+  \sum_{k=0}^{K-1} \alpha_k \left(\sum_{j=k}^{K-1} \alpha_j^{1+\beta}\right)^\theta\right).
    \end{align*}
\end{theorem}
Theorem \ref{thm:diameter} requires the objective function to be definable in a polynomially bounded o-minimal structure, which we recall at the beginning of Section \ref{sec:o-minimal}. We also need the step sizes $(\alpha_k)_{k\in \mathbb{N}}$ associated with the sequence to be small and decreasing, which coincides with the classical choices of step sizes in nonsmooth optimization \cite{polyak1967general,polyak1978subgradient,alber1998projected}. The constants that appear in the theorem depend only on the local geometry of $f$. In particular, $\theta$ is an exponent that appears at the \L{}ojasiewicz gradient inequality of $f$ when restricted on certain smooth manifolds. 

Combining with the existing guarantees for path-differentiable functions, Theorem \ref{thm:diameter} implies that bounded subgradient sequences converge to critical points of $f$, given that certain summations of step sizes are finite. This is the case when the step sizes $(\alpha_k)_{k\in \mathbb{N}}$ decrease and are of order $1/k$, which is widely used in the subgradient method for convex functions \cite{shor1985book,beck2017first}. Given two sequences of positive scalars $(a_k)_{k\in \mathbb{N}}$ and $(b_k)_{k\in \mathbb{N}}$, we write $(a_k)_{k\in \mathbb{N}} \sim (b_k)_{k\in \mathbb{N}}$ if there exist $c,C>0$ such that $c \le a_k/b_k \le C$ for all $k\in \mathbb{N}$.
\begin{corollary}
    \label{cor:1/k}
    Let $f:\mathbb{R}^n\to \mathbb{R}$ be locally Lipschitz and definable in a polynomially bounded o-minimal structure. Any bounded subgradient sequence $(x_k)_{k\in \mathbb{N}}$ with decreasing step sizes $(\alpha_k)_{k\in \mathbb{N}} \sim (1/(k+1))_{k\in \mathbb{N}}$ converges to a critical point of $f$.
\end{corollary}
With Theorem \ref{thm:diameter}, the proof of Corollary \ref{cor:1/k} is quite straightforward: Given a bounded subgradient sequence $(x_k)_{k\in \mathbb{N}}$, by \cite[Theorem 3.2]{davis2020stochastic}, its limit points are critical and $(f(x_k))_{k\in \mathbb{N}}$ converges to $f^*\in \mathbb{R}$ as $(\alpha_k)_{k\in \mathbb{N}}$ is not summable. We apply Theorem \ref{thm:diameter} to the sequence with $f$ replaced by $f - f^*$, which yields an upper bound on $\mathrm{diam}(x_{\llbracket k_1,k_2 \rrbracket})$ for arbitrary $k_1,k_2\in \mathbb{N}$ that are sufficiently large. A direct calculation shows that this upper bound diminishes as $k_1\to \infty$, which implies that the sequence $(x_k)_{k\in \mathbb{N}}$ is Cauchy and thus convergent. This completes the proof of Corollary \ref{cor:1/k}.

From the previous discussions, it is evident that the proof of Theorem \ref{thm:diameter} must leverage the unique geometric properties of o-minimal structures. Our approach diverges from the literature \cite{davis2020stochastic,bolte2022long,bolte2025inexact}, which relies on the continuous-time limit of subgradient sequences. Instead, the proof hinges on the stratifications of definable sets with strong metric properties, which is discussed in Section \ref{sec:o-minimal}. Building on these stratifications, we decompose functions into smooth pieces with locally Lipschitz Riemannian gradients, which relate to their subdifferentials. This is the object of Section \ref{sec:decomposition}. Finally, we prove Theorem \ref{thm:diameter} in Section \ref{sec:main_proof}, which requires analyzing subgradient sequences as they alternate between the components proposed in Section \ref{sec:decomposition}.

\section{O-minimal structures and stratifications}\label{sec:o-minimal}
We begin by recalling some standard notations. Let $\mathbb{N}:=\{0,1,\ldots\}$ be the natural numbers. Let $|\cdot|$ be the induced norm of the Euclidean inner product $\langle \cdot, \cdot\rangle$ on $\mathbb{R}^n$. Let $B(a,r)$ and $\mathring{B}(a,r)$ respectively denote the closed ball and the open ball of center $a\in \mathbb{R}^n$ and radius $r > 0$. Given $A\subset \mathbb{R}^n$, we denote by $\overline{A}$ the closure of $A$, $\partial A:= \overline{A}\setminus A$ the frontier of $A$, and $B(A,r):= A + B(0,r)$. Given $x\in \mathbb{R}^n$, let $d(x,A) := \inf \{|x - y| : y \in A\}$ and $P_A(x):= \mathrm{argmin} \{y \in A:|x - y|\}$. We also use the convention that $d(x,\emptyset) = 1$. Given two sets $A,B\subset \mathbb{R}^n$, we define the distance between them by $d(A,B):= \inf\{d(a,B):a\in A\}$. Let $m$ and $p$ be positive integers, and let $f:A\rightarrow \mathbb{R}^m$. If $A$ is open, then $f$ is $p$ times continuously differentiable (or in short, $C^p$) if the $p$th Fr\'echet derivative of $f$ exists and is continuous in $A$. We denote by $D f$ (resp. $D^2 f$) the first (resp. second) order derivative of $f$. We can extend this definition to functions with non-open domains by saying that $f:A\rightarrow \mathbb{R}^m$ is $C^p$ if there exists a $C^p$ function $\bar{f}:U\rightarrow \mathbb{R}^m$ defined on an open neighborhood $U$ of $A$ such that $f(x) = \bar{f}(x)$ for every $x\in A$. Given an (embedded) smooth manifold $M\subset \mathbb{R}^n$ and $x\in M$, we denote by $T_M(x)$ and $N_M(x)$ respectively the tangent and normal spaces of $M$ at $x$. If $f:\mathbb{R}^n\rightarrow \mathbb{R}$ is $C^2$ on $M$, we denote by $\nabla_M f$ its Riemannian gradient on $M$ \cite[Definition 3.58]{boumal2023introduction}. By \cite[Proposition 3.61]{boumal2023introduction}, $\nabla_M f:M\to \mathbb R^n$ is a $C^1$ function given by $\nabla_M f(x) = P_{T_M(x)}(\nabla \bar{f}(x))$.

In this work, we consider functions and sets that are definable in a polynomially bounded o-minimal structure on the real field \cite{van1998tame,coste2000introduction}. A \emph{structure} on the real field $(\mathbb{R},+, \cdot)$ is a family 
$\mathcal{D} = (\mathcal{D}_n)_{n\ge 1}$,
where for each $n \ge 1$, $\mathcal{D}_n$ is a Boolean algebra of subsets of $\mathbb{R}^n$, satisfying the following properties:
\begin{enumerate}
    \item If $A\in\mathcal{D}_n$, then the sets $\mathbb{R}\times A$ and $A\times\mathbb{R}$ belong to $\mathcal{D}_{n+1}$.
    \item $\mathcal{D}_n$ contains $
    \{ x\in\mathbb{R}^n : P(x)=0 \}$ for all $P\in\mathbb{R}[X_1,\dots,X_n].
    $
    \item If $A\in\mathcal{D}_n$, then its projection 
    $\pi(A)\subset \mathbb{R}^{n-1}$, where $\pi:\mathbb{R}^n\to\mathbb{R}^{n-1}$ is the standard projection onto the first $n-1$ coordinates, belongs to $\mathcal{D}_{n-1}$.
\end{enumerate}
A structure $\mathcal{D}$ is said to be \emph{o-minimal} if, in addition, every set in $\mathcal{D}_1$ is a finite union of points and open intervals. A subset of $\mathbb{R}^n$ that belongs to $\mathcal{D}_n$ is called a \emph{definable set}, and a function whose graph is definable is called a \emph{definable function}. A structure $\mathcal{D}$ is said to be \emph{polynomially bounded} if for every definable function $f\colon\mathbb{R}\to\mathbb{R}$ there exist $a>0$ and $n\in\mathbb{N}$ such that
$|f(x)| < x^n$ for all $x > a$ \cite[p. 510]{van1996geometric}. Examples of polynomially bounded o-minimal structures include semialgebraic sets \cite{tarski1951decision} and globally subanalytic sets \cite{van1996geometric}. On one hand, sets and functions definable in these structures enjoy benign properties such as the \L{}ojasiewicz inequality \cite{lojasiewicz1958,bierstone1988semianalytic} and its consequences \cite{lojasiewicz1963propriete,van1996geometric}. On the other hand, in modern data science applications including the training of deep neural networks \cite{lecun2015deep,vaswani2017attention}, it appears that all the objective functions of interest are subanalytic \cite{bierstone1988semianalytic} and thus locally definable in the structure of global subanalytic sets. Throughout this paper, we fix an arbitrary polynomially bounded o-minimal structure on the real field, and say that the sets or functions are definable if they are definable in this structure.

An important line of research in the study of semialgebraic and o-minimal geometry focuses on the theory of stratification \cite{parusinski1994lipschitz, le1998verdier,trotman2020stratification}. Notably, definable sets can be stratified into a finite number of smooth manifolds that fit together nicely. We recall the following definition of stratification \cite{lojasiewicz1993geometrie,van1996geometric}. 
\begin{definition}
\label{def:stratification}
Let \(M\subset \mathbb{R}^n\) and  \(p\) be a positive integer.
A \emph{\(C^p\) stratification} of \(M\) is a finite partition $\mathcal{M} = \{M_i\}_{i\in I}$
of \(M\) into connected \(C^p\) manifolds \(M_i\subset \mathbb{R}^n\) (called strata) such that for each pair \(i\neq j\),
\[
\overline{M_i} \cap M_j   \neq  \emptyset 
\quad\Longrightarrow\quad
M_j   \subset   \partial M_i.
\]
\end{definition}
Let \(\mathcal{A}:=\{A_j\}_{j\in J}\) be a collection of subsets of $\mathbb{R}^n$. 
Then we say that a stratification \(\mathcal{M}\) is \emph{compatible} with 
$\mathcal{A}$ if for each pair \((M_i,A_j)\in \mathcal{M}\times \mathcal{A}\), it holds that either $M_i \subset A_j$ or $M_i\cap A_j = \emptyset$. We say that a stratification $\mathcal{M}$ is definable if
every stratum $M\in \mathcal{M}$ is definable. As the stratifications in this work can always be made definable and $C^p$ for arbitrary $p$, we will generally shorten ``definable $C^p$ stratification'' into ``stratification''.

Definable sets admit stratifications that come with extra conditions on the tangent spaces of adjacent strata. Examples of such conditions include Whitney's (a), (b), and (w) conditions\footnote{The (w) condition is also known as the Verdier condition \cite{verdier1976stratifications}.}  \cite{le1998verdier,whitney1992tangents}. Among these conditions, the (w) condition is the strongest and poses locally a Lipschitz-like condition on the tangent spaces. This condition has recently been used to analyze  (stochastic) subgradient sequences \cite{bianchi2023stochastic,davis2025active,josz2024sufficient}. In fact, the proof of Theorem \ref{thm:diameter} requires an even stronger form of stratification of definable sets, known as the Lipschitz stratification \cite{mostowski1985lipschitz,parusinski1988lipschitz}. As the definition of Lipschitz stratification is quite complex and has little to do with current work, we defer its definition to Appendix \ref{sec:Lipschitz} and refer the readers to \cite{mostowski1985lipschitz,parusinski1994lipschitz,nguyen2016lipschitz} for further details. Recall that compact sets definable in polynomially bounded o-minimal structures admit Lipschitz stratifications \cite{nguyen2016lipschitz,parusinski1994lipschitz}. By a counterexample of Parusi\'nski (see, for e.g., \cite[Example 2.9]{nguyen2016lipschitz}), this fails to hold in o-minimal structures that are not polynomially bounded. We will use the following fact that strata in a Lipschitz stratification satisfy Whitney's (w) condition with constants depending on distances to frontiers of the strata. This is a direct consequence of \cite[Corollary 1.6]{parusinski1988lipschitz} and the \L{}ojasiewicz inequality \cite[Theorem 6.4]{bierstone1988semianalytic}. Given $g:\mathbb{R}^n\to \mathbb{R}^m$, we denote by $|g|$ its operator norm.

\begin{proposition}\label{prop of lipschitz stratification}
    Let $\{M_i\}_{i\in I}$ be a definable Lipschitz stratification of a bounded set $M\subset \mathbb{R}^n$. There exist $C>0,\eta\ge0$ such that for any pair $i,j\in I$ such that $M_j\subset \partial M_i$, $x\in M_i$, and $y\in M_j$, we have
    \[
     |P_{N_{M_i}(x)}P_{T_{M_j}(y)}|\le \frac{C}{d(y,\partial M_j)^\eta}|x-y|
    \]
    with the convention that $d(y,\emptyset) = 1$.
\end{proposition}
\begin{proof}
    For any $k\in \llbracket 0, \mathrm{dim}(M)\rrbracket$, denote by $M^k$ the union of all the strata in $\{M_i\}_{i\in I}$ with dimension less than or equal to $k$. By \cite[Corollary 1.6]{parusinski1988lipschitz}, there exists $C>0$ such that for any pair $i,j\in I$ such that $M_j\subset \partial M_i$, $x\in M_i$, and $y\in M_j$, it holds that
        \[
     |P_{N_{M_i}(x)}P_{T_{M_j}(y)}|\le \frac{C}{d(y,M^k)}|x-y|
    \]
    where $k:= \mathrm{dim}(M_j) - 1$. By the definition of stratification, $\partial M_j \subset M^k$. If $M^k = \emptyset$, then the desired inequality holds. Otherwise, as $M$ is bounded and the strata are definable, and for any $y\in\overline{M_j}$, \begin{align*}
        d(y,M^k)=0\implies y\in M^k\implies y\notin M_j\implies y\in \partial M_j\implies d(y,\partial M_j)=0.
    \end{align*} By the \L{}ojasiewicz inequality, we have 
    $
    d(y,M^k)\ge cd(y,\partial M_j)^\eta
    $
    for some $c>0,\eta\ge0$ and any $y\in{M_j}$. Conclusion of the proposition then follows by applying the same arguments to every pair of strata.
\end{proof}

Lipschitz stratifications impose strong metric constraints between strata, but they have minimal requirements on the strata themselves, except that the strata can be taken to be definable \emph{cells} \cite{van1996geometric,nguyen2016lipschitz}. Cells are smooth manifolds defined recursively by definable smooth functions, and it is well known that definable sets can be decomposed into cells \cite[4.2 Cell decomposition]{van1996geometric}. It is also of interest to study stratifications with strata that possess strong metric properties \cite{kurdyka1997subanalytic,pawlucki2002decomposition,fischer2007minimal,pawlucki2008linear}. These stratifications have applications, including proving extension theorems \cite{kurdyka1997subanalytic} and analyzing gradient trajectories \cite{kurdyka2001quasi}. We next recall the definition of \emph{$L$-regular cells} \cite{fischer2007minimal,kurdyka2006subanalytic}, which are cells defined by Lipschitz definable functions.

\begin{definition}
The standard $L$-regular cells in $\mathbb{R}$ are precisely the open intervals and singletons.
Assume that standard $L$-regular cells in $\mathbb{R}^{n-1}$ have been defined, then a {standard $L$-regular cell} in $\mathbb{R}^{n}$ is one of the following forms:
\begin{align*}
        &\Gamma(\xi) := \{(x,y) \in B\times\mathbb{R} : y=\xi(x)\},
        \\&(\xi_1,\xi_2)_B := \{(x,y) \in B\times\mathbb{R} : \xi_1(x)<y<\xi_2(x)\},
        \\&(\xi,+\infty)_B := \{(x,y)\in B\times\mathbb{R} : y>\xi(x)\},
        \\&(-\infty,\xi)_B := \{(x,y)\in B\times\mathbb{R} : y<\xi(x)\},
    \end{align*}
    where $B\subset \mathbb R^{n-1}$ is a standard $L$-regular cell and $\xi,\xi_1,\xi_2:B\rightarrow \mathbb{R}$ are Lipschitz definable functions such that $\xi_1<\xi_2$.
\end{definition}
A set $M \subset \mathbb{R}^n$ is called an \emph{$L$-regular cell} if there is a linear orthogonal homeomorphism $\phi: \mathbb{R}^n \to \mathbb{R}^n$ such that $\phi(M)$ is a standard $L$-regular cell. An important property of $L$-regular cells is that they are quasiconvex \cite{kurdyka2006subanalytic,parusinski1994lipschitz}. Recall that a set $M\subset \mathbb{R}^n$ is \emph{$C$-quasiconvex} if for any $x,y\in M$, there exists a rectifiable arc $\gamma$ in $M$ connecting $x$ and $y$ with length at most $C|x-y|$ \cite[Appendix A]{gromov1999metric}. We say that a set $M\subset \mathbb{R}^n$ is \emph{quasiconvex} if there exists $C>0$ such that $M$ is $C$-quasiconvex. By \cite[Lemma 1.1]{kurdyka2001quasi} (see also \cite[Proposition 8]{kurdyka2006subanalytic}), any $L$-regular cell $M\subset \mathbb{R}^n$ is quasiconvex, with a constant $C$ that only depends on the dimension $n$ and Lipschitz constants of the defining functions.

A useful consequence of quasiconvexity is that any smooth function defined on an $L$-regular cell is Lipschitz if the function has a bounded derivative. It is natural to wonder what can be said for smooth functions defined on it, with potentially unbounded derivatives. Let $f$ be a definable smooth function defined on a bounded $L$-regular cell $M$. It is easy to see that there exist $c,\theta>0$ such that
\[
    |D f(x)| \le \frac{c}{d(x,\partial M)^\theta}
\]
for all $x\in M$. We will provide a proof of this simple fact later in Lemma \ref{scale of vector field}. Our goal is to demonstrate that a similar estimate holds for the Lipschitz modulus of $f$. In other words, we would like to prove that
\[
|f(x) - f(y)| \le \frac{c}{d(\{x,y\},\partial M)^\theta}|x-y|
\]
for all $x,y\in M$, with possibly different constants $c,\theta>0$. To achieve this, we show that points in an $L$-regular cell can be connected by arcs that do not come too close to the frontier of the cell. In fact, we prove a stronger statement (Proposition \ref{prop: L' regular}) that this property holds for any set that is bi-Lipschitz homeomorphic to an $L$-regular cell. Let $A\subset \mathbb{R}^n$ and $B\subset \mathbb{R}^m$, recall that a homeomorphism $\phi:A\to B$ is bi-Lipschitz if there exists $\overline{L}>0$ such that
\begin{equation*}
  \frac{1}{\overline{L}}|x-y| \le |\phi(x) - \phi(y)| \le \overline{L}|x-y|
\end{equation*}
for all $x,y\in A$. Clearly, if $\phi$ is bi-Lipschitz, then so is $\phi^{-1}$. We say a set $M\subset \mathbb{R}^n$ is \emph{$L'$-regular} if there is a bi-Lipschitz homeomorphism $\phi: M \to \mathbb{R}^m$ such that $\phi(M)\subset \mathbb{R}^m$ is an $L$-regular cell, or equivalently, a standard $L$-regular cell.
Proposition \ref{prop: L' regular} will be used in Section \ref{sec:decomposition} to estimate Lipschitz constants of Riemannian gradients.
\begin{proposition}
    \label{prop: L' regular}
     Let $M\subset \mathbb{R}^n$ be bounded and $L'$-regular. Then there exist $\varrho,C>0,\theta\ge 1$ such that for any $t\in(0,1]$, there exists an $L'$-regular $C$-quasiconvex $M(t)$ such that 
\[
M\setminus B(\partial M,t)\subset M(t)\subset M\setminus B(\partial M,\varrho t^{\theta}).
\]
\end{proposition}
\begin{proof}
We first prove the case where $M\subset \mathbb{R}^n$ is a bounded standard $L$-regular cell, and then show that all the desired properties are preserved under bi-Lipschitz homeomorphisms. Recall that an $L$-regular cell has constant $C>0$ if the norms of derivatives of all the defining functions of the cell are bounded by $C$ \cite{kurdyka2006subanalytic,kurdyka2001quasi}. If an $L$-regular cell $M\subset \mathbb{R}^n$ has constant $C$, then it is $(C+1)^{n-1}$-quasiconvex \cite[Lemma 1.1]{kurdyka2001quasi}. For bounded standard $L$-regular cells, we prove the following stronger claim.

Claim: \emph{Given a bounded standard $L$-regular cell $M\subset \mathbb{R}^n$, there exist $\varrho,C>0,\theta\ge1$ such that for any $t\in(0,1]$, there exists a standard $L$-regular cell $M(t)$ with constant $C$ such that }
\begin{equation}\label{eq:inclusion}
    M\setminus B(\partial M,t)\subset M(t)\subset M\setminus B(\partial M,\varrho t^{\theta}).
\end{equation}

We prove the above claim by an induction on the dimension $n$. The claim is clearly true for $n=1$. Assume that the claim holds for $n = 1,\ldots,N$, we will then show that the claim holds for $n = N+1$. Let $M\subset \mathbb{R}^{N+1}$ be a bounded standard $L$-regular cell. Denote by $\pi:\mathbb{R}^{N+1}\to \mathbb{R}^N$ the canonical projection. We have that $T:= \pi(M)\subset \mathbb{R}^N$ is a bounded standard $L$-regular cell, and either $M = \Gamma(\xi)$ or $M = (\xi_1,\xi_2)_T$ where $\xi,\xi_1,\xi_2:T\rightarrow \mathbb{R}$ are $L_0$-Lipschitz definable functions with $\xi_1<\xi_2$ and some $L_0>0$. Note that these functions can be extended continuously to $\overline{T}$.

By the inductive hypothesis, there exist $\varrho',C'>0$ and $\theta' \ge 1$ such that for any $t\in(0,1]$, there exists a standard $L$-regular cell $T(t)$ with constant $C'$ such that 
\begin{equation}\label{eq:T_inductive}
  T\setminus B(\partial T,t)\subset T(t)\subset T\setminus B(\partial T,\varrho' t^{\theta'}).  
\end{equation}

We will construct desired standard $L$-regular cells $M(t)$ based on $T(t)$. If $M = \Gamma(\xi)$, then $\partial M = \Gamma(\xi|_{\partial T})$. In this case, we may let 

\[M(t):= \left\{(y',\xi(y')): y'\in T\left(t/\sqrt{2+L_0^2}\right)\right\}\]
for any $t\in (0,1]$, which is a standard $L$-regular cell. Moreover, $M(t)$ is a standard $L$-regular cell with constant $C = \max\{C',L_0\}$. Fix any such $t$, it remains to show that $M(t)$ satisfies the desired inclusion \eqref{eq:inclusion}. On one hand, for any $(x',\xi(x'))\in M\setminus B(\partial M,t)$, it holds that
\begin{align*}
    t^2 \le d((x',\xi(x')),\partial M)^2 &= \inf_{(y',\xi(y'))\in \partial M} |x' - y'|^2 + |\xi(x') - \xi(y')|^2\\
    &\le \inf_{y'\in \partial T} |x' - y'|^2 + L_0^2 |x' - y'|^2\\
    &= (1+L_0^2) d(x',\partial T)^2< (2+L_0^2) d(x',\partial T)^2.
\end{align*}
Thus, by \eqref{eq:T_inductive},  $x'\in T\setminus B(\partial T,t/\sqrt{2+L_0^2}) \subset T(t/\sqrt{2+L_0^2})$ and $(x',\xi(x'))\in M(t)$. On the other hand, for any $(y',\xi(y'))\in M(t)$, it holds that $$d((y',\xi(y')),\partial M) \ge d(y', \partial T) \ge \varrho'\left(t/\sqrt{2+L_0^2}\right)^{\theta'}.$$ Clearly $(y',\xi(y'))\in \Gamma(\xi) = M$. Therefore, the desired inclusion \eqref{eq:inclusion} holds by letting $\theta = \theta'$ and choosing a small enough $\varrho>0$.

We next consider the case where $M = (\xi_1,\xi_2)_T$. In this case, the frontier of $M$ is given by
\begin{align}\label{eq:partial M}
    \partial M= \Gamma(\xi_1)\cup \Gamma(\xi_2)\cup \underbrace{\{(x',x_{N+1}):x'\in \partial T, x_{N+1} \in [\xi_1(x'), \xi_2(x')]\}}_{=:    M^\sharp}.
\end{align}
Since $\xi_2>\xi_1$ on $T$, by the \L{}ojasiewicz inequality, there exist $\kappa\ge1,c>0$ such that $\xi_2(x') - \xi_1(x') > c d(x',\partial T)^\kappa$ for all $x'\in T$. Let 
\begin{equation}\label{eq:def_Mt}
    M(t):= \left\{(y',y_{N+1}): y'\in T\left(t/\sqrt{2+L_0^2}\right), y_{N+1}\in (\xi_1(y')+\beta(t),\xi_2(y')-\beta(t))\right\}
\end{equation}
for $t\in (0,1]$ where \[\beta(t):= \frac{c}{2}(\varrho')^\kappa \left(t/\sqrt{2+L_0^2}\right)^{\theta'\kappa}.\] After possibly reducing $c$, we assume that $\beta(t)<t$. We first show that $M(t)$ is a standard $L$-regular cell. Fix any $t$, it suffices to prove that $\xi_1+\beta(t)<\xi_2-\beta(t)$ on $T\left(t/\sqrt{2+L_0^2}\right)$. Indeed, for any $y'\in T\left(t/\sqrt{2+L_0^2}\right)$, we have
\begin{align*}
    \xi_2(y') - \xi_1(y') &> c d(x',\partial T)^\kappa\\
    &\ge c \left(\varrho' \left(t/\sqrt{2+L_0^2}\right)^{\theta'}\right)^\kappa\\
    &= c(\varrho')^\kappa \left(t/\sqrt{2+L_0^2}\right)^{\theta'\kappa}= 2\beta(t).
\end{align*}
Since \(\beta(t)\) is a scalar depending only on \(t\), it is constant with respect to the
variable \(y'\). Hence the derivatives of the defining functions \(\xi_1+\beta(t)\) and
\(\xi_2-\beta(t)\) are the same as those of \(\xi_1\) and \(\xi_2\), respectively. Therefore
\(M(t)\) is a standard $L$-regular cell with constant \(\max\{C',L_0\}\).

 We next show that $M(t)$ satisfies the desired inclusion \eqref{eq:inclusion}. Let $(y',y_{N+1}) \in M\setminus B(\partial M,t)$, we first prove that $y'\in T(t/\sqrt{2+L_0^2})$. By the decomposition of the frontier in \eqref{eq:partial M},
\begin{align}
    d((y',y_{N+1}),\partial M) &\le d((y',y_{N+1}),M^\sharp)\notag\\
    &= \inf\{\sqrt{|x' - y'|^2 + |x_{N+1} - y_{N+1}|^2}: x'\in \partial T, x_{N+1} \in [\xi_1(x'), \xi_2(x')]\}\notag\\
    &\le  \inf\{\sqrt{|x' - y'|^2 + |x_{N+1} - y_{N+1}|^2}: x'\in \partial T, x_{N+1} = P_{[\xi_1(x'), \xi_2(x')]}(y_{N+1})\}\notag\\
    &= \inf\{\sqrt{|x' - y'|^2 + d(y_{N+1},[\xi_1(x'), \xi_2(x')])^2}: x'\in \partial T\} \notag\\
    &\le \inf\{\sqrt{|x' - y'|^2 + L_0^2 |x' - y'|^2}: x'\in \partial T\}\label{eq:both_interval}\\
    &= \sqrt{1+L_0^2} d(y',\partial T),\notag
\end{align}
where \eqref{eq:both_interval} is due to the fact that $y_{N+1}\in (\xi_1(y'),\xi_2(y'))$ and that both $\xi_1$ and $\xi_2$ are $L_0$-Lipschitz. As $d((y',y_{N+1}),\partial M) \ge t$, it holds that $d(y',\partial T) >d((y',y_{N+1}),\partial M)/\sqrt{2+L_0^2} \ge t/\sqrt{2+L_0^2}$. Thus, $y'\in T\left(t/\sqrt{2+L_0^2}\right)$. In addition, $y_{N+1}\in (\xi_1(y')+\beta(t),\xi_2(y')-\beta(t))$ since 
\[
    |y_{N+1} - \xi_1(y')| \ge d((y',y_{N+1}),\Gamma(\xi_1)) \ge d((y',y_{N+1}),\partial M) \ge t>\beta(t)
\]
and $|y_{N+1} - \xi_2(y')|>\beta(t)$ by the same reasoning. Therefore, $ M\setminus B(\partial M,t)\subset M(t)$ by the definition of $M(t)$ in \eqref{eq:def_Mt}.

It remains to show that $M(t)\subset M\setminus B(\partial M,\varrho t^{\theta})$ for some $\varrho>0,\theta\ge1$. Let $(y',y_{N+1})\in M(t)$. By the definition of $M(t)$ in \eqref{eq:def_Mt}, $y'\in T$ and $y_{N+1}\in (\xi_1(y') + \beta(t), \xi_2(y') - \beta(t))$. Thus $(y',y_{N+1})\in M$. It remains to show that $(y',y_{N+1})$ is away from the frontier $\partial M$. Recall from \eqref{eq:partial M} that $\partial M$ can be decomposed into three parts (i.e., $\Gamma(\xi_1)$, $\Gamma(\xi_2)$, and $M^\sharp$). We will proceed by lower bounding the distance from $(y',y_{N+1})$ to all these components.

Let $(x',\xi_1(x'))\in \Gamma(\xi_1)$. It holds that
\[
\begin{aligned}
|(x',\xi_1(x'))-(y',y_{N+1})|^2
&= |x'-y'|^2 + |\xi_1(x')-y_{N+1}|^2 \\
&\ge |x'-y'|^2
   + \Bigl||\xi_1(y')-y_{N+1}|-|\xi_1(x')-\xi_1(y')|\Bigr|^2 \\
&\ge |x'-y'|^2
   + \bigl(\max\{\beta(t)-L_0|x'-y'|,0\}\bigr)^2,
\end{aligned}
\]
due to the triangular inequality and the fact that
\[|\xi_1(y')-y_{N+1}| \in [\beta(t),\infty)\quad\text{and}\quad|\xi_1(x')-\xi_1(y')|\in [0,L_0|x'-y'|].\]
Set
\[
a:=|x'-y'|,\qquad
\phi(a):=a^2+\bigl(\max\{\beta(t)-L_0 a,0\}\bigr)^2,\qquad a\ge 0.
\]
If \(0\le a\le \beta(t)/L_0\), then
\[
\phi(a)=a^2+(\beta(t)-L_0a)^2,
\]
whose minimum is \(\beta(t)^2/(1+L_0^2)\). If \(a\ge \beta(t)/L_0\), then
\[
\phi(a)=a^2\ge \frac{\beta(t)^2}{L_0^2}\ge \frac{\beta(t)^2}{1+L_0^2}.
\]
Hence
\[
\phi(a)\ge \frac{\beta(t)^2}{1+L_0^2}
\qquad\text{for all }a\ge 0.
\]
Therefore,
\[
d\bigl((y',y_{N+1}),\Gamma(\xi_1)\bigr)\ge \frac{\beta(t)}{\sqrt{1+L_0^2}} = \frac{c}{2\sqrt{1+L_0^2}}(\varrho')^\kappa \left(t/\sqrt{2+L_0^2}\right)^{\theta'\kappa}.
\]
The same argument yields
\[
d\bigl((y',y_{N+1}),\Gamma(\xi_2)\bigr)\ge \frac{c}{2\sqrt{1+L_0^2}}(\varrho')^\kappa \left(t/\sqrt{2+L_0^2}\right)^{\theta'\kappa}.
\]
Finally, by the definition of $M^\sharp$ in \eqref{eq:partial M},
\[d\bigl((y',y_{N+1}),M^\sharp\bigr)\ge d\bigl(y',\partial T\bigr) \ge \varrho' \left(t/\sqrt{2+L_0^2}\right)^{\theta'}.\]
Combining the above bounds, there exists $\varrho>0$ such that
\[d\bigl((y',y_{N+1}),\partial M\bigr)>\varrho t^{\theta},\qquad\theta:=  \theta'\kappa \ge 1.\]
Thus $(y',y_{N+1})\in M\setminus B(\partial M,\varrho t^{\theta})$. This concludes the induction of the desired claim for standard
\(L\)-regular cells.

We proceed to show that both quasiconvexity and the desired inclusions are preserved by bi-Lipschitz homeomorphisms. Let $M\subset \mathbb{R}^n$ be a bounded $L'$-regular set. By the definition of \(L'\)-regular sets, there exists a bounded standard \(L\)-regular cell \(N\) and a bi-Lipschitz homeomorphism $\phi:N\rightarrow M$ with 
\begin{equation}\label{eq:biLip}
  \frac{1}{\overline{L}}|x-y| \le |\phi(x) - \phi(y)| \le \overline{L}|x-y|
\end{equation}
for all $x,y\in N$ and some $\overline{L}>0$. As $\phi$
 is bi-Lipschitz, it can be extended to a bi-Lipschitz homeomorphism $\phi:\overline{N}\rightarrow \overline{M}$. Thus, \(\partial M = \phi(\partial N)\). 

By the proved claim, there are \(\varrho',C>0\), $\theta\ge 1$ so that for each \(s\in(0,1]\), there exists a standard \(L\)-regular \(C\)-quasiconvex cell $N(s)$ with
\[
N\setminus B(\partial N,s) \subset N(s) \subset N\setminus B(\partial N,\varrho' s^{\theta}).
\]
Now for any \(t\in(0,1]\), we will show that the conclusion of the proposition holds with $M(t) := \phi\bigl(N(t/(2\overline{L}))\bigr)$. Since \(\phi\) is a bi-Lipschitz homeomorphism, \(M(t)\) is \(L'\)-regular. Applying the definition of bi-Lipschitz functions (i.e., equation \eqref{eq:biLip}) to arbitrary arcs in $N(t/(2\overline{L}))$, we have that $M(t)$ is $\overline{L}^2C$-quasiconvex. It remains to verify that the desired inclusions hold with \(\varrho := \varrho' (2\overline{L})^{-1-\theta}\).

On one hand, for any $x' = \phi(x)\in M\setminus B(\partial M,t)$ with $x\in N$, we have

\[
d(x,\partial N)
 = \inf_{y\in\partial N}|x-y|
 \ge \frac{1}{\overline{L}} \inf_{y'\in\partial M}|x'-y'|
 = \frac{d(x',\partial M)}{\overline{L}}
 \ge \frac{t}{\overline{L}},
\]
so \(x\in N\setminus B(\partial N,t/(2\overline{L}))\subset N(t/(2\overline{L}))\) and hence \(x'\in M(t)\).

On the other hand, let \(x'\in M(t)\). Since $M(t) = \phi\bigl(N(t/(2\overline{L}))\bigr)\subset \phi(N) = M$, it suffices to show that $x'\not\in B(\partial M,\varrho t^\theta)$. Let \(x:=\phi^{-1}(x')\in N(t/(2\overline{L}))\).  Then $d(x,\partial N)
 \ge \varrho'(t/(2\overline{L}))^{\theta}$,
and therefore
\[
d(x',\partial M)
 = \inf_{y'\in\partial M}|x'-y'|
 \ge \frac{1}{\overline{L}} \inf_{y\in\partial N}|x-y|
 \ge \frac{1}{\overline{L}} \varrho'\Bigl(\frac{t}{2\overline{L}} \Bigr)^{\theta}>\varrho t^\theta.
\]
It follows that \(x'\notin B(\partial M,\varrho t^{\theta})\) and $M(t) \subset M\setminus B(\partial M,\varrho t^{\theta})$. This completes the proof of the proposition.
\end{proof}

A stratification $\mathcal{M}$ is $L$-regular (resp. $L'$-regular) if every stratum $M\in \mathcal{M}$ is $L$-regular (resp. $L'$-regular). Definable sets admit $L$-regular stratifications, as shown in \cite{fischer2007minimal,kurdyka2006subanalytic}. In the next section, we aim to decompose definable functions into smooth pieces with controlled Lipschitz modulus of their Riemannian gradients. To achieve this, we require a stratification of definable sets that is both Lipschitz and $L$-regular. Fortunately, this is possible by applying $L$-regular refinements to the constructive proof of Lipschitz stratification in \cite{nguyen2016lipschitz} (see also \cite{parusinski1994lipschitz}). 
\begin{theorem}\label{lipschitz stratification}
    Let $X$ be a compact definable subset of $\mathbb{R}^n$ and $X_1,\dots,X_l$ be definable subsets of $X$. For any integer $m\ge 1$, there exists a Lipschitz $L$-regular $C^m$ definable stratification of $X$ compatible with $X_1,\dots,X_l$.
\end{theorem}
The proof of Theorem \ref{lipschitz stratification} is similar to that of \cite[Theorem 2.6]{nguyen2016lipschitz}. One only needs to begin
with the inductive assumption that every definable set of dimension < dim $X$ admits
a Lipschitz $L$-regular stratification; from that point on, the argument follows word for
word. We therefore omit the details.
\section{Piecewise smooth decomposition of definable functions}\label{sec:decomposition}
In this section, we present a decomposition of locally Lipschitz definable functions, building upon the stratifications of definable sets developed in the previous section. We first show that the domain of a locally Lipschitz definable function can be stratified so that we may estimate Lipschitz constants of Riemannian gradients on each stratum, and relate them to the subdifferential (Proposition \ref{prop:function_Lipschitz}). We also construct neighborhoods of these strata, where we provide additional estimates for the projection maps (Proposition \ref{proposition:neighborhoods of cells}). These results lay the groundwork for the study of subgradient sequences in Section \ref{sec:main_proof}.

We start this section by proving a useful lemma which controls how continuous definable maps blow up near boundaries of their domains. This is a simple application of the Łojasiewicz inequality.

\begin{lemma}\label{scale of vector field}
Let $M\subset \mathbb{R}^n$ be bounded and $V:M \rightarrow \mathbb{R}^m$ be continuous and definable. Then there exist $c>0,\theta \ge0$ such that $|V(x)| \le c/d(x,\partial M)^\theta$ for all $x\in M$.
\end{lemma}
\begin{proof}
Consider the function $Q:M\rightarrow \mathbb{R}$ defined by $Q(x):= \min\{1,1/|V(x)|\}$ for all $x\in M$. Clearly, $Q$ is continuous definable and takes only positive values on $M$. By the Łojasiewicz inequality \cite[Theorem 6.4]{bierstone1988semianalytic}, there exist $c>0,\theta\ge0$ such that $Q(x)\ge d(x,\partial M)^{\theta}/c$ for all $x\in M$. This yields the desired inequality.
\end{proof}
In the following proposition, we consider a decomposition of locally Lipschitz definable functions induced by Lipschitz $L$-regular stratifications (Theorem \ref{lipschitz stratification}) of their graphs. When projected to the domain of the function, this yields a stratification where Riemannian gradients and subdifferentials can be well controlled. These estimates are made possible due to Lemma \ref{scale of vector field} and the properties of $L'$-regular sets established in Proposition \ref{prop: L' regular}. A local version of the estimate in \eqref{eq:global_verdier} can be derived by considering a Verdier stratification of (epi)graphs, as discussed in \cite{bianchi2023stochastic} and \cite[Theorem 3.6]{davis2025active}. By leveraging the Lipschitz stratification, we obtain a stronger estimate that holds in any compact set.

\begin{proposition}\label{prop:function_Lipschitz}
     Let \( f: \mathbb{R}^n \to \mathbb{R} \) be locally Lipschitz definable and $U\subset \mathbb{R}^n$ be definable compact. Then for any integer $m \ge 1$, there exists a $C^m$ $L'$-regular stratification $\mathcal{M}$ of $U$ such that $f$ is $C^m$ on each $M\in \mathcal{M}$. In addition, there exist $L>0,\eta\ge0$ such that for any $M\in \mathcal{M}$, we have
     \begin{equation}\label{eq:Rie_grad_Lip}
         |\nabla_{M}f(x) - \nabla_{M}f(y)| \le \frac{L}{d(\{x,y\},\partial M)^{\eta}} |x - y|
     \end{equation}
     for every $x,y\in M$, and
     \begin{equation}\label{eq:global_verdier}
 |P_{T_{M'}(y)}(v) - \nabla_{M'} f(y)| \le \frac{L}{d(y,\partial M')^{\eta}}|x - y|         
     \end{equation}
for every $x\in M$, $v\in \partial f(x)$, $M'\in \mathcal{M}$ with $M'\subset \partial M$, and $y\in M'$.
     
\end{proposition}
\begin{proof}
By Theorem~\ref{lipschitz stratification}, the graph \(\Gamma(f|_U)\subset\mathbb{R}^n\times\mathbb{R}\) admits a Lipschitz \(L\)-regular $C^m$ stratification $\mathcal{X}$ with $m\ge 2$. Denote by $\pi:\mathbb{R}^{n+1}\to\mathbb{R}^n$ the canonical projection. Since the stratification $\mathcal{X}$ is Lipschitz (and thus satisfies the Whitney's condition (a)) and $f|_U$ is Lipschitz, $\mathcal{M}:=\{\pi(X):X\in \mathcal{X}\}$ is a $C^m$ stratification of $U$ such that $f$ is $C^m$ on each $M\in \mathcal{M}$ \cite[p. 561]{bolte2007clarke}. $\mathcal{M}$ is also $L'$-regular as $\pi|_{\Gamma(f|_M)}$ is a bi-Lipschitz homeomorphism for each $M\in \mathcal{M}$.

We proceed to prove the two inequalities. Fix $M\in \mathcal M$.
By Proposition \ref{prop: L' regular}, there exist $\varrho,C>0,\theta\ge 1$ such that for any $t\in(0,1]$, there exists a $C$-quasiconvex $M(t)$ such that  
\begin{equation} \label{eq:Mt_Riem}
 M\setminus B(\partial M,t)\subset M(t)\subset M\setminus B(\partial M,\varrho t^{\theta}).   
\end{equation}
Set 
\[
\Delta:=\max\{1,\operatorname{diam}(M)\},\qquad
C_0:=\frac{\varrho}{\Delta^\theta}.
\]
Fix any $x,y\in M$ and let $\delta:= d(\{x,y\},\partial M)$. By the definition of diameter, 
$$\mathrm{diam}(M) = \mathrm{diam}(\overline M) \ge d(\{x,y\},\partial M).$$
Thus, $\Delta \ge \delta$. Applying \eqref{eq:Mt_Riem} with \(t=\delta/\Delta\in (0,1]\), we obtain an arc
\(\gamma:[0,1]\to M\) with $\gamma(0) = x$ and $\gamma(1) = y$ such that
\[
\operatorname{length}(\gamma):= \int_0^1 |\gamma'(t)|~dt\le C|x-y|
\]
and
\[
\gamma([0,1])\subset M\setminus B\bigl(\partial M,\varrho (\delta/\Delta)^\theta\bigr) = M\setminus B\bigl(\partial M,C_0\delta^\theta\bigr).
\]
Therefore, for any \(x,y\in M\), there exists an arc $\gamma$ with \(\gamma([0,1])\subset
M\setminus B(\partial M,C_0 d(\{x,y\},\partial M)^\theta)\) that connects \(x\) and
\(y\), with length no greater than \(C|x-y|\).

Denote by \(D_M(\nabla_M f)\) the tangential differential of the map
\(\nabla_M f:M\to\mathbb R^n\); see, for example,
\cite[Definition 3.34]{boumal2023introduction}.
Since \(|D_M(\nabla_M f)|\) is definable and continuous,
by Lemma~\ref{scale of vector field} there exist \(c>0\) and \(\theta'\ge0\) such that
\[
|D_M(\nabla_M f)(\bar x)|
\le \frac{c}{d(\bar x,\partial M)^{\theta'}}
\qquad\text{for all }\bar x\in M.
\]
Let \(\gamma:[0,1]\to M\) be the arc from above. Then, for a.e. \(s\in[0,1]\), by the definition of the differential,
\[
\frac{d}{ds}(\nabla_M f\circ\gamma)(s)
= D_M(\nabla_M f)(\gamma(s))[\dot\gamma(s)].
\]
Hence
\[
\begin{aligned}
|\nabla_M f(x)-\nabla_M f(y)|
&= \left| \int_0^1 \frac{d}{ds}(\nabla_M f\circ\gamma)(s)\,ds\right|\\
&\le \int_0^1 |D_M(\nabla_M f)(\gamma(s))|\,|\dot\gamma(s)|\,ds \\
&\le \sup_{s\in [0,1]} \frac{c}{d(\gamma(s),\partial M)^{\theta'}} \int_0^1 |\dot\gamma(s)|\,ds \\
&\le \frac{c}{\bigl(C_0\,d(\{x,y\},\partial M)^\theta\bigr)^{\theta'}}\,\operatorname{length}(\gamma) \\
&\le \frac{cC}{C_0^{\theta'}\,d(\{x,y\},\partial M)^{\theta\theta'}}\,|x-y|.
\end{aligned}
\]
Applying the same argument to each $M\in \mathcal M$ yields \eqref{eq:Rie_grad_Lip}.

We next prove \eqref{eq:global_verdier}. Fix $M,M'\in \mathcal{M}$ with $M'\subset \partial M$, $x\in M$, $y\in M'$, and $v\in \partial f(x)$. Denote by $X := \Gamma(f|_M)$ and $X':= \Gamma(f|_{M'})$. We first show that $(v,-1)\in N_{X}(x,f(x))$. Note that the tangent space of $X$ at $(x,f(x))$ is given by $T_{X}(x,f(x)) = \{(u,\langle\nabla_{M}f(x),u\rangle):u\in T_{M}(x)\}$. Thus for any $(u,\langle\nabla_{M}f(x),u\rangle)\in T_{X}(x,f(x))$, we have $\langle (v,-1),(u,\langle\nabla_{M}f(x),u\rangle)\rangle = \langle v-\nabla_{M}f(x),u\rangle = 0$ as $P_{T_M}(v) = \nabla_{M}f(x)$ by \cite[Proposition 4]{bolte2007clarke}.

Now let $u\in T_{M'}(y)$ be arbitrary and we have $(u,\langle\nabla_{M'}f(y),u\rangle)\in T_{X'}(y,f(y))$. Notice that
\begin{equation}\label{eq:ip_1}
 \begin{aligned}
    \langle (v,-1),\bigl(u, \langle\nabla_{M'}f(y),u\rangle\bigr)\rangle &= \langle v,u\rangle - \langle\nabla_{M'}f(y),u\rangle\\
    &= \langle v-\nabla_{M'}f(y),u\rangle\\
    &= \langle P_{T_{M'}(y)}(v)-\nabla_{M'}f(y),u\rangle.
\end{aligned}   
\end{equation}

As \(U\) is compact and \(f\) is locally Lipschitz, there exists \(L_0>0\) such that
\[
|f(x)-f(y)|\le L_0|x-y|\quad\text{and}\quad|v|\le L_0
\qquad\text{for all }x,y\in U\quad\text{and}\quad v\in \partial f(x).
\]
We seek to lower bound the inner product on the left hand side of \eqref{eq:ip_1}. It holds that
\begin{subequations}\label{eq:ip_2}
   \begin{align}
    \langle (v,-1),\bigl(u,\langle\nabla_{M'}f(y), u\rangle\bigr)\rangle &= \left\langle(v,-1),P_{N_{X}(x,f(x))}\bigl(u,\langle\nabla_{M'}f(y), u\rangle\bigr)\right\rangle\label{eq:ip_a}\\
    &\le \sqrt{L_0^2+1}|P_{N_{X}(x,f(x))}(u,\langle \nabla_{M'}f(y),u\rangle)|\label{eq:ip_b}\\
    &=\sqrt{L_0^2+1}|P_{N_{X}(x,f(x))}P_{T_{X'}(y,f(y))}(u,\langle \nabla_{M'}f(y),u\rangle)|\label{eq:ip_c}
    \\&\le \frac{C\sqrt{L_0^2+1}|(u,\langle \nabla_{M'}f(y),u\rangle)|}{d((y,f(y)),\partial X')^\eta} |(x,f(x))-(y,f(y))|\label{eq:ip_d}\\
  &\le \frac{C\sqrt{L_0^2+1}\sqrt{|u|^2+L_0^2|u|^2}}{d((y,f(y)),\Gamma(f|_{\partial M'}))^\eta}\times \sqrt{L_0^2+1}|x-y|\label{eq:ip_e}\\
  &\le \frac{C({L_0^2+1})^{3/2}|u|}{d(y,\partial M')^\eta} |x-y|.\label{eq:ip_f}
\end{align} 
\end{subequations}
for some $C>0,\eta\ge0$. Above, \eqref{eq:ip_a} follows from the fact $(v,-1)\in N_{X}(x,f(x))$; \eqref{eq:ip_c} is due to $(u,\langle\nabla_{M'}f(y),u\rangle)\in T_{X'}(y,f(y))$; \eqref{eq:ip_d} is a consequence of Proposition \ref{prop of lipschitz stratification}; We use the Lipschitz continuity of $f$ in \eqref{eq:ip_e}. Finally, the desired inequality \eqref{eq:global_verdier} follows by combining \eqref{eq:ip_1} and \eqref{eq:ip_2} and letting $u=(P_{T_{M'}(y)}(v)-\nabla_{M'}f(y))/|P_{T_{M'}(y)}(v)-\nabla_{M'}f(y)|\in T_{M'}(y)$.
\end{proof}
The remainder of this section concerns neighborhoods of the strata constructed in Proposition \ref{prop:function_Lipschitz}. In the proof of Theorem \ref{thm:diameter} (Section \ref{sec:main_proof}), our main strategy is to analyze projected subgradient sequences. Therefore, it is necessary to study the  projections onto the strata. By the celebrated tubular neighborhood theorem, each smooth manifold admits a neighborhood where the projection is well-defined \cite[Theorem 6.24]{lee2012introduction}. If restricted to a smaller region, the projection becomes Lipschitz continuous \cite[4.8 Theorem]{federer1959curvature} and smooth \cite[(4.1) Theorem]{dudek1994nonlinear}. Applying the definable version of the tubular neighborhood theorem \cite[Theorem 6.11]{coste2000introduction}, the following lemma quantifies the size of such a neighborhood for definable manifolds.

\begin{lemma}\label{lemma:xi}
    Let $M\subset \mathbb{R}^n$ be a nonempty bounded definable $C^3$ manifold, then for any $L>1$ there exist $r>0$ and $\eta\ge 1$ such that $P_M$ is $L$-Lipschitz and $C^{2}$ in 
    $\cup_{t\in (0,1]}B\left(M\setminus B(\partial M,t),rt^\eta\right)$.
\end{lemma}
\begin{proof}
By the definable tubular neighborhood theorem \cite[Theorem 6.11]{coste2000introduction}, there is a $C^3$ definable $\epsilon:M\rightarrow(0,\infty)$ such that the projection onto $M$ (denoted by $P_M$) is single-valued and $C^2$ in $\cup_{x\in M}B(x,\epsilon(x))$. 
After possibly replacing \(\epsilon\) by \(\frac{L-1}{L}\epsilon\), which is still definable and
\(C^3\), we may moreover assume that
\[
\epsilon(x)\le \frac{L-1}{L}\operatorname{reach}(M,x)
\qquad\text{for all }x\in M,
\]
where \(\operatorname{reach}(M,x)\) denotes the pointwise reach in the sense of
\cite[4.1 Definition]{federer1959curvature}. Then, by \cite[Theorem~4.8(8)]{federer1959curvature},
\(P_M\) is \(L\)-Lipschitz on \(\bigcup_{x\in M} B(x,\epsilon(x))\).
 Since $\overline{M}\setminus\partial M = M \neq \emptyset$, there exists $\varrho\in (0,1]$ such that $\overline{M}\setminus \mathring{B}(\partial M,\varrho) \neq \emptyset$. Consider $\xi:[0,1]\rightarrow [0,\infty]$ defined by
    \[\xi(t):= \inf_{x\in \overline{M}\setminus \mathring{B}(\partial M,t)}\epsilon(x)\]
for $t\in (0,1]$, and $\xi(0) := \inf_{t\in (0,\varrho]}\xi(t)$. By continuity of $\epsilon$, $\xi$ is continuous and positive near $0$. Since $\xi$ is increasing on $(0,1]$, it is continuous at $0$ as well. In addition, it is definable since $\epsilon$ is definable. By the \L{}ojasiewicz's inequality \cite[Theorem 6.4]{bierstone1988semianalytic}, there exist $r>0$ and $\eta\ge 1$ such that $\xi(t) \ge r t^\eta$ for all $t\in [0,1]$. Therefore, we have $\epsilon(x)\ge \xi(t) \ge rt^\eta$ for all $x\in M\setminus B(\partial M,t)$ and $t\in [0,1]$. It follows that
\begin{equation*}
    B(M\setminus B(\partial M,t),rt^\eta)\subset \bigcup_{x\in M\setminus B(\partial M,t)}B(x,\epsilon(x)).
\end{equation*}
The conclusion then follows by taking union over $t\in (0,1]$ for both sides on the above inclusion.
\end{proof}
While the projection is Lipschitz in the neighborhood constructed in Lemma \ref{lemma:xi}, its derivative is not, as the second-order derivative blows up near the frontier of the manifold. Indeed, the same can be said for the Riemannian gradients, which are only locally Lipschitz on the manifold. This hinders one from applying classical arguments for analyzing gradient sequences \cite{absil2005convergence}, which generally require gradients to be Lipschitz. To overcome this hurdle, we propose to construct regions which exclude areas near the frontiers (see $\mathcal{N}_0(i,\alpha)$ defined in Proposition \ref{proposition:neighborhoods of cells}). In such regions, we may estimate Lipschitz constants of the aforementioned maps (Proposition \ref{proposition:neighborhoods of cells}). To this end, we need to study arc-connectivity of such regions, which is the object of the following lemma. Recall from Proposition \ref{prop:function_Lipschitz}, we may assume the strata are $L'$-regular.
\begin{lemma}\label{lem: connected away boundary in ball}
Let \(M\subset\mathbb{R}^n\) be a nonempty bounded definable \(L'\)-regular \(C^3\) manifold and let \(L>1\). Then there exist constants
\[
C>0,\qquad\varrho\in (0,1],\qquad \theta\ge 1,\qquad \overline{\eta}\ge \theta,\qquad \overline r\in \bigl(0,\min\{\varrho,1/3\}\bigr),
\]
such that the following holds:
\begin{enumerate}
    \item \(M\setminus B(\partial M,\varrho)\neq\emptyset\);
    \item For any \(t\in(0,1]\), \(\eta\in[\overline\eta,\infty)\), and \(r\in(0,\overline r]\), the projection \(P_M\) is \(L\)-Lipschitz and \(C^2\) in
\[
\bigcup_{t\in (0,1]} B\bigl(M\setminus B(\partial M,t),\,rt^\eta\bigr),
\]
and for any two points
\[
x,y\in B\bigl(M\setminus B(\partial M,t),\,rt^\eta\bigr),
\]
there exists an arc connecting \(x\) and \(y\) in
\[
B\bigl(M\setminus B(\partial M,\varrho t^\theta),\,rt^\eta\bigr)
\]
with length no greater than \(C|x-y|\);
\item Consequently, for any \(t\in(0,1]\), \(\eta\in[\overline\eta,\infty)\), and
\(0<r_1<r_2\le \overline r\), let
\[
\Omega:=\mathring{B}\bigl(M\setminus B(\partial M,\varrho t^\theta),\,r_2 t^\eta\bigr).
\]
If \(F:\Omega\to\mathbb R^m\) is a \(C^1\) definable map, then there exist constants
\(c>0\) and \(\omega\ge 0\) such that
\[
|F(x)-F(y)|\le \frac{c}{t^\omega}|x-y|
\qquad
\forall x,y\in B\bigl(M\setminus B(\partial M,t),\,r_1 t^\eta\bigr).
\]
\end{enumerate}
\end{lemma}

\begin{proof}
The proof is divided into four steps. In Step 0, we apply Proposition \ref{prop: L' regular} and Lemma \ref{lemma:xi} to obtain quasiconvex sets in the manifold and its tubular neighborhoods. We then prove items 1-3 of the lemma in Steps 1-3 respectively.

\medskip
\noindent\underline{Step 0: Auxiliary quasiconvex sets and tubular neighborhoods.}
By Proposition~\ref{prop: L' regular}, there exist constants \(C'>0\), \(\varrho'>0\), and
\(\theta\ge 1\) such that for every \(t\in(0,1]\), one can find a \(C'\)-quasiconvex set
\(M(t)\) satisfying
\[
M\setminus B(\partial M,t)\subset M(t)\subset M\setminus B(\partial M,\varrho' t^\theta).
\]

Fix \(L>1\). Applying Lemma~\ref{lemma:xi}, we obtain constants
\[
\overline r_0\in(0,1/3),\qquad \overline\eta\ge \theta,
\]
such that for every \(t\in(0,1]\), \(\eta\in[\overline\eta,\infty)\), and
\(r\in(0,\overline r_0]\), the projection
\[
P_M: \bigcup_{t\in (0,1]} B\bigl(M\setminus B(\partial M,t),\,rt^\eta\bigr)\to M
\]
is single-valued, \(L\)-Lipschitz, and \(C^2\).

\medskip
\noindent\underline{Step 1: Choice of \(\varrho\) and \(\overline r\), and proof of item 1.}
We now choose \(\varrho\in(0,1]\) so that
\[
M\setminus B(\partial M,\varrho)\neq\emptyset.
\]

If \(\partial M=\emptyset\), we simply set
\(
\varrho_0:=\varrho'/4^\theta
\),
so that
\[
M\setminus B(\partial M,\varrho_0)=M\neq\emptyset.
\]

Assume next that \(\partial M\neq\emptyset\). Since \(M\) is a nonempty manifold,
\(M\setminus \partial M\neq\emptyset\). Choose \(x_*\in M\setminus \partial M\). As
\(\partial M\) is closed and \(x_*\notin \partial M\), one has
\[
\delta_*:=d(x_*,\partial M)>0.
\]
Set
\[
\varrho_0:=\min\Bigl\{\frac{\varrho'}{4^\theta},\,\frac{\delta_*}{2}\Bigr\}.
\]
Then \(x_*\in M\setminus B(\partial M,\varrho_0)\), hence
\[
M\setminus B(\partial M,\varrho_0)\neq\emptyset.
\]

In either case, define
\[
\varrho:=\varrho_0,
\qquad
\overline r:=\min\Bigl\{\overline r_0,\frac{\varrho}{2}\Bigr\}\in \bigl(0,\min\{\varrho,1/3\}\bigr).
\]
This proves item~1.

\medskip
\noindent\underline{Step 2: Proof of item 2.}
Let \(t\in(0,1]\), \(\eta\in[\overline\eta,\infty)\), and \(r\in(0,\overline r]\). Since
\(\overline r\le \overline r_0\), Step~1 gives immediately that \(P_M\) is single-valued,
\(L\)-Lipschitz, and \(C^2\) in
\[
B\bigl(M\setminus B(\partial M,t),\,rt^\eta\bigr).
\]

Fix
\[
x,y\in B\bigl(M\setminus B(\partial M,t),\,rt^\eta\bigr).
\]
We first show that their projections lie in the quasiconvex set \(M(t/4)\).

Since \(x\in B(M\setminus B(\partial M,t),rt^\eta)\), we have
\[
d(x,\partial M)\ge t-rt^\eta.
\]
Because \(\eta\ge 1\), \(t\in(0,1]\), and \(r\le 1/3\), we have \(rt^\eta\le t/3\), hence
\[
d(x,\partial M)\ge t-rt^\eta\ge \frac{2t}{3}.
\]
Also,
\[
d\bigl(x,M\setminus B(\partial M,t)\bigr)\le rt^\eta\le \frac{t}{3}.
\]
Therefore,
\[
d\bigl(P_M(x),\partial M\bigr)
\ge d(x,\partial M)-|x-P_M(x)|
\ge \frac{2t}{3}-rt^\eta
\ge \frac{t}{3}.
\]
Thus \(P_M(x)\in M(t/4)\). The same argument shows that \(P_M(y)\in M(t/4)\).

Since \(M(t/4)\) is \(C'\)-quasiconvex, there exists an arc
\(\vartheta:[0,1]\to M(t/4)\) such that
\[
\vartheta(0)=P_M(x),\qquad \vartheta(1)=P_M(y),
\]
and
\[
\int_0^1 |\vartheta'(s)|\,ds
\le C'|P_M(x)-P_M(y)|
\le C'L|x-y|.
\]

Now define \(\gamma:[0,2]\to\mathbb R^n\) by
\[
\gamma(s):=\vartheta(s)+\bigl(x-P_M(x)\bigr)
\qquad\text{for }s\in[0,1],
\]
and
\[
\gamma(s):=(2-s)\gamma(1)+(s-1)y
\qquad\text{for }s\in(1,2].
\]
Then \(\gamma(0)=x\) and \(\gamma(2)=y\).

\smallskip
\noindent\emph{Length estimate.}
We have
\[
\begin{aligned}
\operatorname{length}(\gamma)
&=\int_0^1 |\vartheta'(s)|\,ds
   +\bigl|\gamma(1)-y\bigr| \\
&=\int_0^1 |\vartheta'(s)|\,ds
   +\bigl|P_M(y)+x-P_M(x)-y\bigr| \\
&\le C'L|x-y|+|P_M(y)-P_M(x)|+|x-y| \\
&\le \bigl(C'L+L+1\bigr)|x-y|.
\end{aligned}
\]
Thus the required length bound holds with
\[
C:=C'L+L+1.
\]

\smallskip
\noindent\emph{Location of the arc.}
It remains to prove that
\[
\gamma([0,2])\subset
B\bigl(M\setminus B(\partial M,\varrho t^\theta),\,rt^\eta\bigr).
\]

For \(s\in[0,1]\),
\[
|\gamma(s)-\vartheta(s)|=|x-P_M(x)|\le rt^\eta,
\]
hence
\[
\gamma([0,1])\subset B\bigl(M(t/4),rt^\eta\bigr).
\]
Since
\[
M(t/4)\subset M\setminus B\bigl(\partial M,\varrho'(t/4)^\theta\bigr)
\]
and \(\varrho\le \varrho'/4^\theta\), we have
\[
\varrho t^\theta\le \varrho'(t/4)^\theta,
\]
so
\[
M(t/4)\subset M\setminus B(\partial M,\varrho t^\theta).
\]
Therefore,
\[
\gamma([0,1])\subset
B\bigl(M\setminus B(\partial M,\varrho t^\theta),\,rt^\eta\bigr).
\]

For the second part of the path, note that
\[
|\gamma(1)-P_M(y)|=|x-P_M(x)|\le rt^\eta,
\qquad
|\gamma(2)-P_M(y)|=|y-P_M(y)|\le rt^\eta.
\]
Hence the line segment \(\gamma([1,2])\) is contained in the ball
\(
B\bigl(P_M(y),rt^\eta\bigr).
\)
Since \(P_M(y)\in M(t/4)\subset M\setminus B(\partial M,\varrho t^\theta)\), we obtain
\[
\gamma([1,2])\subset
B\bigl(M\setminus B(\partial M,\varrho t^\theta),\,rt^\eta\bigr).
\]
This completes the proof of item~2.

\medskip
\noindent\underline{Step 3: Proof of item 3.}
Fix \(t\in(0,1]\), \(\eta\in[\overline\eta,\infty)\), and \(0<r_1<r_2\le \overline r\). Set
\[
A:=M\setminus B(\partial M,\varrho t^\theta),
\qquad
\Omega:=\mathring B(A,r_2t^\eta).
\]
By Step~2, \(A\neq\emptyset\), hence \(\Omega\neq\emptyset\).

Let \(F:\Omega\to\mathbb R^m\) be a \(C^1\) definable map. By
Lemma~\ref{scale of vector field}, applied to the continuous definable map \(|DF|\), there
exist constants \(c_0>0\) and \(\omega\ge 0\) such that
\[
|DF(z)|\le \frac{c_0}{d(z,\partial\Omega)^\omega}
\qquad\forall z\in\Omega.
\]

Now take
\[
x,y\in B\bigl(M\setminus B(\partial M,t),\,r_1t^\eta\bigr).
\]
By item~2, there exists an arc \(\gamma\) connecting \(x\) and \(y\) in
\[
B\bigl(M\setminus B(\partial M,\varrho t^\theta),\,r_1t^\eta\bigr)
=
B(A,r_1t^\eta),
\]
with $\operatorname{length}(\gamma)\le C|x-y|$. Since \(r_1<r_2\), we have
\[
B(A,r_1t^\eta)\subset \mathring B(A,r_2t^\eta)=\Omega,
\]
so \(\gamma([0,2])\subset\Omega\).

We claim that
\[
d(z,\partial\Omega)\ge (r_2-r_1)t^\eta
\qquad\forall z\in\gamma([0,2]).
\]
Indeed, for such \(z\), $d(z,A)\le r_1t^\eta$. On the other hand, if \(w\in\partial\Omega\), then \(d(w,A)=r_2t^\eta\). Since the distance
function \(u\mapsto d(u,A)\) is \(1\)-Lipschitz,
\[
|z-w|\ge d(w,A)-d(z,A)\ge (r_2-r_1)t^\eta.
\]
Taking the infimum over \(w\in\partial\Omega\) proves the claim. Therefore,
\[
\begin{aligned}
|F(x)-F(y)|
&\le \int_\gamma |DF(z)|\,ds \\
&\le \frac{c_0}{((r_2-r_1)t^\eta)^\omega}\,\operatorname{length}(\gamma) \\
&\le \frac{Cc_0}{(r_2-r_1)^\omega}\,\frac{1}{t^{\eta\omega}}\,|x-y|.
\end{aligned}
\]
Thus the desired estimate holds with
\[
c:=\frac{Cc_0}{(r_2-r_1)^\omega}.
\]
This proves item~3.
\end{proof}

We conclude Section~\ref{sec:decomposition} with Proposition \ref{proposition:neighborhoods of cells}, which decomposes a bounded definable set into regions where the derivatives of the function and the projections are Lipschitz. These regions correspond to the stratification constructed in Proposition \ref{prop:function_Lipschitz}. For each stratum, this region is essentially a ball around it, after excluding a ball of the stratum's frontier. Note that these regions are non-uniform as each stratum is associated with a different radius, as illustrated in Figure \ref{fig:Ni}. This ensures that the projection maps are well-defined, and that the subdifferential can be related to the Riemannian gradient on the corresponding stratum. Also, regions of adjacent strata have nonempty intersections, which is essential for the later algorithmic analysis. Given a stratification $\{M_i\}_{i\in I}$ of $X\subset \mathbb{R}^n$ and $x\in X$, we denote by $M_x$ the only stratum that contains $x$.

\begin{proposition}\label{proposition:neighborhoods of cells}
Let \( f: \mathbb{R}^n \to \mathbb{R} \) be locally Lipschitz definable, $X\subset \mathbb{R}^n$ be definable compact, $L>1$, and $\{M_1,\ldots, M_T\}$ be a stratification of $X$ given by Proposition~\ref{prop:function_Lipschitz} with $m\ge 3$. There exist $\eta_{i}\ge 1$ for $i\in \llbracket 1,T\rrbracket$ such that for any $\beta_i,\gamma_j\ge0$ that satisfy
\begin{equation*}
\forall i,j\in \llbracket 1,T\rrbracket,\quad    M_j\subset \partial M_i \implies 0\le \eta_{i}\gamma_j \le \beta_i,
\end{equation*}
there exist $c_i,c_{1,i},c_{2,i},c_{3,i}>0$ and $\omega_{1,i},\omega_{2,i},\omega_{3,i}\in[0,\beta_i/4]$ for $i\in \llbracket 1,T\rrbracket$ such that for any $\alpha\in (0,1]$, the following holds:
\begin{enumerate}
\item $P_{M_i}$ is $L$-Lipschitz continuous and $C^{2}$ in $\bigcup_{\alpha\in(0,1]}\mathcal N_0(i,\alpha)$, where
\begin{equation}\label{eq:partition}
    \mathcal{N}_0(i,\alpha):= X\cap   B\left(M_i \setminus \bigcup_{j:M_j \subset \partial M_i} B(M_j, c_j \alpha^{\gamma_j}/2),2c_i\alpha^{\beta_i}\right).
\end{equation}
 \item For any $j\in\llbracket 1,T\rrbracket$ such that $\overline{M_i}\cap M_j=M_i\cap\overline{M_j}=\emptyset$, we have
\[\bigcup _{\alpha\in(0,1]}\mathcal{N}_0(i,\alpha)\cap\bigcup_{\alpha\in(0,1]}\mathcal{N}_0(j,\alpha)=\emptyset.\]
    \item $M_i \subset \overline{M_x}$ for any $x\in \mathcal{N}_0(i,\alpha)$.
\item It holds that   
\begin{align}
\label{eq:Lip_1}    &|\nabla_{M_i} f(x)-\nabla_{M_i} f(y)|\le \frac{c_{1,i}}{\alpha ^{\omega_{1,i}}} |x - y|,\quad \forall x,y\in \mathcal{N}_0(i,\alpha)\cap M_i,\\
\label{eq:Lip_2}    &|P_{T_{M_i}(y)}(v)-\nabla_{M_i}f(y)|\le \frac{c_{2,i}}{\alpha^{\omega_{2,i}}}|x-y|,\quad \forall x\in \mathcal{N}_0(i,\alpha),y\in \mathcal{N}_0(i,\alpha)\cap M_i,v\in\partial f(x),\\
\label{eq:Lip_3}    &|DP_{M_i}(x)-DP_{M_i}(y)|\le \frac{c_{3,i}}{\alpha^{\omega_{3,i}}}|x-y|,\quad \forall x,y\in \mathcal{N}_0(i,\alpha).
\end{align}
\end{enumerate}
\end{proposition}

\definecolor{MyBlue}{RGB}{65,105,225}
\definecolor{MyMagenta}{RGB}{204,51,153}
\definecolor{FillColor}{RGB}{200,220,255}
\definecolor{Area1}{RGB}{255,200,200}
\definecolor{Area2}{RGB}{200,255,200}
\definecolor{Area3}{RGB}{200,200,255}
\definecolor{Area12}{RGB}{255,255,200}
\definecolor{Area13}{RGB}{255,200,255}
\definecolor{Area23}{RGB}{200,255,255}
\definecolor{Area123}{RGB}{255,255,255}

\begin{figure}
\centering
    \begin{tikzpicture}[scale=1.8]
\tikzset{
  levelset/.style={draw=black, dashed, line cap=round},
  sublevel/.style={draw=orange!80!black, dash pattern=on 6pt off 4pt},
  valley/.style={draw=MyBlue, very thick},
  iterline/.style={draw=MyMagenta, line width=1.0pt, -{Latex[length=2.2mm]}},
  iterdot/.style={fill=MyMagenta, draw=MyMagenta},
  switchdot/.style={draw=MyMagenta, fill=white, very thick},
  brace/.style={decorate, decoration={brace, amplitude=5pt}, thick},
  lab/.style={font=\scriptsize}
}

;
\begin{scope}
\fill[Area1, opacity=0.4]
    (-0.13,0.58) .. controls (-0.8,1.2) and (-1.2,1.9) .. (-1.6,2.4)
    --
    (-2.5,1.6) .. controls (-1.85,0.9) and (-1.7,0.6) .. (-0.48,-0.38)
    (0,0) circle (0.6);
\fill[Area2, opacity=0.4]
    (.45,.47).. controls (1.2,0.5) and (2.4,-.1) .. (2.9,.4)
    --
    (2.9,-.6) .. controls (2.4,-1.1) and (1.2,-0.5) .. (.45,-.45)
    (0,0) circle (0.6);;
\fill[white]
    (0,0) circle (0.6);
\fill[Area3, opacity=0.4]
    (0,0) circle (0.85);
\end{scope}

\fill[Area1] (4.1,2.1) rectangle (4.6,2.4);
\node[align=left, font=\small, anchor=west] at (4.8,2.25) {$\mathcal{N}_0(1,\alpha)$};

\fill[Area2] (4.1,0.7) rectangle (4.6,1);
\node[align=left, font=\small, anchor=west] at (4.8,0.85) {$\mathcal{N}_0(3,\alpha)$};

\fill[Area3] (4.1,1.4) rectangle (4.6,1.7);
\node[align=left, font=\small, anchor=west] at (4.8,1.55) {$\mathcal{N}_0(2,\alpha)$};

\draw[levelset] (4.1,0.15) -- (4.6,0.15);
\node[align=left, font=\small, anchor=west] at (4.8,0.15) {$2c_i\alpha^{\beta_i}$};

\draw[sublevel] (4.1,-0.55) -- (4.6,-0.55);
\node[align=left, font=\small, anchor=west] at (4.8,-0.55) {$c_i\alpha^{\gamma_i}/2$};

\draw[levelset] (-1.6,2.4) .. controls (-1.2,1.9) and (-0.8,1.2) .. (-0.13,0.58);
\draw[sublevel]  (-1.8,2.1) .. controls (-1.4,1.6) and (-1.0,0.9) .. (-0.39,0.45);
\draw[sublevel]  (-2.15,1.8) .. controls (-1.75,1.3) and (-1.35,0.6) .. (-0.60,0.05);
\draw[levelset] (-2.5,1.6) .. controls (-1.85,0.9) and (-1.7,0.6) .. (-0.48,-0.38);

\draw[levelset] (0,0) circle (0.85);
\draw[sublevel] (0,0) circle (0.6);

\draw[levelset](.45,.47).. controls (1.2,0.5) and (2.4,-.1) .. (2.9,.4);
\draw[levelset](.45,-.45).. controls (1.2,-0.5) and (2.4,-1.1) .. (2.9,-.6);
\draw[sublevel](.61,.2).. controls (1.2,0.2) and (2.4,-.35) .. (2.9,.15);
\draw[sublevel](.61,-.1).. controls (0.9,-0.1) and (2.2,-0.8) .. (2.9,-.3);
\draw[valley] (-2,2) .. controls (-1.1,0.6) and (-0.5,0.2) .. (0,0);

\draw[valley] (0,0) .. controls (1.2,0.3) and (2.2,-0.7) .. (3,0);
\filldraw[MyBlue] (0,0) circle (1pt);
\node[align=left, font=\small, anchor=west] at (-2.4,2.2) {$M_1$};
\node[align=left, font=\small, anchor=west] at (0,-0.2) {$M_2$};
\node[align=left, font=\small, anchor=west] at (3.0,-0.15) {$M_3$};
\end{tikzpicture}
\caption{Constructed regions $\mathcal{N}_0(i,\alpha)$ for three adjacent strata.}
\label{fig:Ni}
\end{figure}

\begin{proof}
The proof proceeds in seven steps. In Step 1, we record the consequences of
Lemma~\ref{lem: connected away boundary in ball} and
Proposition~\ref{prop:function_Lipschitz} for each stratum \(M_i\), obtaining
uniform neighborhoods on which the relevant maps satisfy quantitative Lipschitz estimates.
In Step 2, we treat the lowest-dimensional strata separately. In Step 3, for a
general stratum \(M_i\), we choose the constants \(\eta_i,\beta_i,c_i,\gamma_i\)
recursively so that all geometric separation requirements are satisfied. In
Steps 4-7, we prove items 1-4 of the proposition respectively.

\medskip
\noindent\underline{Step 1: Preliminary consequences of Lemma~\ref{lem: connected away boundary in ball}.}
For each \(i\in \llbracket 1,T\rrbracket\), apply Lemma~\ref{lem: connected away boundary in ball} to the stratum \(M_i\) with the prescribed constant \(L>1\). We obtain constants
\[
\varrho_i\in (0,1],\qquad \theta_i\ge 1,\qquad \overline\eta_i\ge \theta_i,\qquad
\overline r_i\in (0,\min\{\varrho_i,1/3\})
\]
such that
\[
M_i\setminus B(\partial M_i,\varrho_i)\neq \emptyset,
\]
and hence
\[
M_i\setminus B(\partial M_i,\varrho_i t^{\theta_i})\neq \emptyset
\qquad \forall\, t\in(0,1].
\]

For each \(t\in(0,1]\), define
\begin{equation}\label{eq:omega_i_t}
    \Omega_{i,t}
:=
\mathring B\Bigl(
M_i\setminus B(\partial M_i,\varrho_i t^{\theta_i}),
\ \overline r_i\varrho_i^{\overline\eta_i}t^{\theta_i\overline\eta_i}
\Bigr).
\end{equation}
Since \(M_i\setminus B(\partial M_i,\varrho_i t^{\theta_i})\neq\emptyset\), the set \(\Omega_{i,t}\) is nonempty for every \(t\in(0,1]\). By the second item of Lemma~\ref{lem: connected away boundary in ball},
\(P_{M_i}\) is \(L\)-Lipschitz and \(C^2\) on \(\bigcup_{t\in (0,1]}\Omega_{i,t}\). Thus, the maps
\[
\widetilde F_{1,i}(z):=P_{T_{M_i}(P_{M_i}(z))}\in \mathbb R^{n\times n},
\qquad
\widetilde F_{2,i}(z):=\nabla_{M_i}f\circ P_{M_i}(z)\in \mathbb R^{n},
\qquad
\widetilde F_{3,i}(z):=DP_{M_i}(z)\in \mathbb R^{n\times n}
\]
are definable \(C^1\) maps on \(\Omega_{i,t}\). We now apply the third item of Lemma~\ref{lem: connected away boundary in ball} to each
\(\widetilde F_{j,i}\), with
\[
\eta:=\theta_i\overline\eta_i,\qquad
r_1:=\frac{\overline r_i\varrho_i^{\overline\eta_i}}{2},
\qquad
r_2:=\overline r_i\varrho_i^{\overline\eta_i}.
\]
Since \(\theta_i\overline\eta_i\ge \overline\eta_i\) and $0<r_1<r_2 \le\overline{r}_i$, these choices are admissible. Therefore, for each \(j=1,2,3\), there exist constants
\[
\widetilde c_{j,i}>0,\qquad \widetilde\omega_{j,i}\ge 0
\]
such that
\begin{equation}\label{eq: tilde F ineq}
|\widetilde F_{j,i}(x)-\widetilde F_{j,i}(y)|
\le
\frac{\widetilde c_{j,i}}{t^{\widetilde\omega_{j,i}}}|x-y|,\quad \forall x,y\in
B\Bigl(
M_i\setminus B(\partial M_i,t),
\ \frac{\overline r_i\varrho_i^{\overline\eta_i}}{2}\,t^{\theta_i\overline\eta_i}
\Bigr).
\end{equation}

By Proposition~\ref{prop:function_Lipschitz}, there exist \(C>0\) and \(\iota\ge 0\) such that for each \(i\in\llbracket1,T\rrbracket\),
\begin{equation}\label{eq: 3.5Rieman lip}
|\nabla_{M_i}f(x)-\nabla_{M_i}f(y)|
\le \frac{C}{d(\{x,y\},\partial M_i)^{\iota}}|x-y|,
\qquad \forall x,y\in M_i,
\end{equation}
and
\begin{equation}\label{eq: 3.5verdier}
\begin{aligned}
    &|P_{T_{M_j}(y)}(v)-\nabla_{M_j}f(y)|
\le \frac{C}{d(y,\partial M_j)^\iota}|x-y|,\\
&\forall i,j:\ M_j\subset\partial M_i,\ \forall x\in M_i,\ \forall y\in M_j,\ \forall v\in\partial f(x).
\end{aligned}
\end{equation}

\medskip
\noindent\underline{Step 2: Lowest-dimensional strata.}
Let \(M_i\) be a stratum of lowest dimension in the stratification. Then \(\partial M_i=\emptyset\), and \(M_i\) is compact. We claim that any choice of constants satisfying
\begin{equation}\label{eq:lowest}
c_i\in\left(0,\frac{1}{6}\min\left\{\overline r_i,\inf\left\{d(M_i,M_j):\overline{M_i}\cap \overline{M_j} = \emptyset \right\}\right\}\right),\qquad
\beta_i,\gamma_i\in [0,\infty),\qquad \eta_i\in [1,\infty)
\end{equation}
is sufficient to establish items 1,3, and 4 of the proposition; The proof of item 2 for those strata is deferred to Step 5.

Indeed, since \(\partial M_i=\emptyset\), \(\mathcal N_0(i,\alpha)=X\cap B(M_i,2c_i\alpha^{\beta_i})\) for all $\alpha\in(0,1]$, and hence
\[
\bigcup_{\alpha\in(0,1]}\mathcal N_0(i,\alpha)=X\cap B(M_i,2c_i).
\]
If \(x\in X\cap B(M_i,2c_i)\), then any stratum whose closure is disjoint from \(\overline{M_i}=M_i\) cannot contain \(x\), by the choice of \(c_i\). Therefore, by the frontier condition, either \(M_x=M_i\) or \(M_i\subset \partial M_x\). In either case,
\[
M_i\subset \overline{M_x}.
\]
Also, since \(2c_i<\overline r_i\), the second item of Lemma~\ref{lem: connected away boundary in ball} implies that \(P_{M_i}\) is \(L\)-Lipschitz and \(C^2\) on \(X\cap B(M_i,2c_i)\).

We next prove \eqref{eq:Lip_1}--\eqref{eq:Lip_3} for such a lowest-dimensional stratum. Since \(M_i\) is compact and \(P_{M_i}\) is \(C^2\) on \(X\cap B(M_i,2c_i)\), the maps
\[
\nabla_{M_i}f \quad\text{on }M_i,
\qquad
DP_{M_i}\quad\text{on }X\cap B(M_i,2c_i)
\]
are Lipschitz on compact sets. Hence \eqref{eq:Lip_1} and \eqref{eq:Lip_3} hold with
\[
\omega_{1,i}=\omega_{3,i}=0
\]
and some \(c_{1,i},c_{3,i}>0\).

To prove \eqref{eq:Lip_2}, let \(x\in \mathcal N_0(i,\alpha)\), \(y\in \mathcal N_0(i,\alpha)\cap M_i\), and \(v\in\partial f(x)\). Since \(f\) is Lipschitz on the compact set \(X\), there exists \(L_0>0\) such that
\begin{equation}\label{eq:Lip_L0}
    |v|\le L_0
\qquad\text{for all }z\in X,\ v\in\partial f(z).
\end{equation}

\smallskip
\noindent\emph{Case 2.1: \(x\in M_i\).}
Then by \cite[Proposition 4]{bolte2007clarke}, $P_{T_{M_i}(x)}(v)=\nabla_{M_i}f(x)$. Therefore
\begin{align*}
|P_{T_{M_i}(y)}(v)-\nabla_{M_i}f(y)|
&\le |P_{T_{M_i}(y)}(v)-P_{T_{M_i}(x)}(v)|
   +|\nabla_{M_i}f(x)-\nabla_{M_i}f(y)| \\
&\le |P_{T_{M_i}(y)}-P_{T_{M_i}(x)}|\,|v|
   +|\nabla_{M_i}f(x)-\nabla_{M_i}f(y)| \\
&\le C_i |x-y|
\end{align*}
for some constant \(C_i>0\), because the maps
\[
z\mapsto P_{T_{M_i}(z)}
\qquad\text{and}\qquad
z\mapsto \nabla_{M_i}f(z)
\]
are Lipschitz on the compact manifold \(M_i\).

\smallskip
\noindent\emph{Case 2.2: \(x\notin M_i\).} Since $M_i \subset \overline{M_x}$, 
\(x\in M_k\) for some stratum \(M_k\) such that \(M_i\subset \partial M_k\). Applying \eqref{eq: 3.5verdier} to the pair \((M_k,M_i)\), and using compactness of \(M_i\), we obtain
\[
|P_{T_{M_i}(y)}(v)-\nabla_{M_i}f(y)|\le C_i |x-y|
\]
for some constant \(C_i>0\) independent of \(\alpha\).

Thus \eqref{eq:Lip_2} holds with \(\omega_{2,i}=0\) and some \(c_{2,i}>0\). Therefore \eqref{eq:Lip_1}--\eqref{eq:Lip_3} hold for all lowest-dimensional strata, and
\[
\omega_{1,i},\omega_{2,i},\omega_{3,i}\le \beta_i/4
\]
holds trivially.

\medskip
\noindent\underline{Step 3: Recursive choice of the constants for a general stratum.}
Fix now a stratum \(M_i\) with \(\partial M_i\neq \emptyset\), and assume that the constants have already been fixed for all strata contained in \(\partial M_i\).

Define
\begin{equation}\label{eq:def_I_i}
    I_i:=\left\{k\in \llbracket 1,T\rrbracket:
\overline{M_i}\cap \overline{M_k}\neq \emptyset,\ 
M_i\cap \overline{M_k} = 
\overline{M_i}\cap M_k=\emptyset
\right\},
\end{equation}
and
\[
J_i:=\{k\in \llbracket 1,T\rrbracket:\ M_k\subset \partial M_i\}.
\]
For every \(k\in I_i\), by the \L{}ojasiewicz inequality there exist \(a_{ik}>0\) and \(\lambda_{ik}\ge 1\) such that
\[
d(y,M_k)\ge a_{ik}\, d(y,\partial M_i)^{\lambda_{ik}}
\qquad\text{for all }y\in \overline{M_i}.
\]
Let
\[
D:=\max\{1,\operatorname{diam}(X)\}.
\]
Choose \(\eta_i\) so that
\[
\eta_i\ge \theta_i \overline{\eta}_i,
\qquad
\eta_i\ge \lambda_{ik}\ \text{for all }k\in I_i,
\qquad
\eta_i\ge4\max\{\iota,\widetilde\omega_{1,i},\widetilde\omega_{2,i},\widetilde\omega_{3,i}\},
\]
where $\theta_i,\bar{\eta}_i,\widetilde\omega_{\cdot,i}$ are constants defined in Step 1 of the proof. We then choose \(r_i\in(0,\overline r_i]\) so small that
\[
r_i\le \frac13 a_{ik}D^{\lambda_{ik}-\eta_i}
\qquad\text{for every }k\in I_i.
\]
With this choice, for all $y\in \overline{M_i}$ and $k\in I_i$, we have
\begin{equation}\label{eq:adjacent-fixed}
\begin{aligned}
d(y,M_k)&\ge a_{ik}\, d(y,\partial M_i)^{\lambda_{ik}}\\
&= 3r_i D^{\eta_i - \lambda_{ik}} d(y,\partial M_i)^{\lambda_{ik}}\\
&= 3r_i D^{\eta_i} (d(y,\partial M_i)/D)^{\lambda_{ik}}\\
&\ge 3r_i D^{\eta_i} (d(y,\partial M_i)/D)^{\eta_i}\\
&= 3r_i d(y,\partial M_i)^{\eta_i}.
\end{aligned}
\end{equation}

It remains to fix the constants that define $\mathcal N_0(i,\alpha)$.  define
\[
\underline{c}_i:=\min\{c_j:M_j\subset \partial M_i\},
\qquad
\overline\gamma_i:=\max\{\gamma_j:M_j\subset \partial M_i\},
\]
and
\begin{equation}\label{eq:delta_i}
    \delta_i:=\inf\{d(M_i,M_k):\overline{M_i}\cap \overline{M_k}=\emptyset\}>0.
\end{equation}
We choose any
\begin{equation}\label{eq:beta_i_range}
    \beta_i\ge\eta_i\overline\gamma_i,
\end{equation}
choose \(c_i>0\) so small that
\begin{equation}\label{eq:range_c}
c_i\le \min\left\{
\frac{\overline r_i\varrho_i^{\overline{\eta}_i}}{4}\Bigl(\frac{\underline{c}_i}{2}\Bigr)^{\theta_i\overline{\eta}_i},
\ \frac{\underline{c}_i}{8},
\ \frac{\delta_i}{6},
\ r_i \Bigl(\frac{\underline{c}_i}{2}\Bigr)^{\eta_i}/2
\right\},
\end{equation}
and finally fix any \(\gamma_i\ge0\).

\medskip
\noindent\underline{Step 4: Regularity of \(P_{M_i}\).}
We claim that with the above choices, \(P_{M_i}\) is \(L\)-Lipschitz and \(C^2\) on
\[
\bigcup_{\alpha\in(0,1]}\mathcal N_0(i,\alpha).
\]

Fix \(\alpha\in(0,1]\), and define
\begin{equation}\label{eq:t_ia}
t_{i,\alpha}:=\min\{c_j\alpha^{\gamma_j}/2:M_j\subset \partial M_i\}\ge \frac{\underline{c}_i}{2}\alpha^{\overline\gamma_i}.    
\end{equation}
Also, by the recursive definition $c_j \le 1/3$ for all $M_j\subset \partial M_i$, so $t_{i,\alpha} \le 1$.
Because
\[
\beta_i\ge \eta_i\overline\gamma_i\ge \theta_i\overline\eta_i\,\overline\gamma_i,
\]
we have
\[
\alpha^{\beta_i}\le \alpha^{\theta_i\overline\eta_i\,\overline\gamma_i}
\qquad\text{for all }\alpha\in(0,1].
\]
Hence, by \eqref{eq:range_c},
\[
2c_i\alpha^{\beta_i}
\le
\frac{\overline r_i\varrho_i^{\overline\eta_i}}{2}\,
t_{i,\alpha}^{\theta_i\overline\eta_i}.
\]
Therefore
\begin{equation}\label{eq: ball contains ball 2}
\mathcal N_0(i,\alpha)\subset 
B\Bigl(
M_i\setminus B(\partial M_i,t_{i,\alpha}),
\ \frac{\overline r_i\varrho_i^{\overline\eta_i}}{2}\,t_{i,\alpha}^{\theta_i\overline\eta_i}
\Bigr) \subset \Omega_{i,t_{i,\alpha}},
\end{equation}
where $\Omega_{i,\cdot}$ is defined in \eqref{eq:omega_i_t}. By the second item of Lemma \ref{lem: connected away boundary in ball}, \(P_{M_i}\) is \(L\)-Lipschitz and \(C^2\) on $\bigcup_{\alpha\in(0,1]}\mathcal N_0(i,\alpha)$. This proves the claim.

\medskip
\noindent\underline{Step 5: Empty intersection of adjacent neighborhoods.} Fix any $j\in\llbracket 1,T\rrbracket$ such that $\overline{M_i}\cap M_j=M_i\cap\overline{M_j}=\emptyset$. We will prove that
\begin{equation}\label{eq:no_intersect}
    \bigcup _{\alpha\in(0,1]}\mathcal{N}_0(i,\alpha)\cap\bigcup_{\alpha\in(0,1]}\mathcal{N}_0(j,\alpha)=\emptyset.
\end{equation}
Assume the contrary and let $x\in \mathcal{N}_0(i,\alpha_1)\cap\mathcal{N}_0(j,\alpha_2)$ for some $\alpha_1,\alpha_2\in (0,1]$. Then there exist $y_i\in M_i \setminus \bigcup_{k:M_k \subset \partial M_i} B(M_k, c_k \alpha_1^{\gamma_k}/2)$ and $y_j\in M_j \setminus \bigcup_{k:M_k \subset \partial M_j} B(M_k, c_k \alpha_2^{\gamma_k}/2)$, such that 
\begin{equation}\label{eq:dist_y}
    |y_i-y_j|\le |x-y_i|+|x-y_j|\le 2c_i\alpha_1^{\beta_i}+2c_j\alpha_2^{\beta_j}.
\end{equation}
We consider the following two cases.

\smallskip
\noindent\emph{Case 5.1: $\overline{M_i} \cap \overline{M_j} = \emptyset$.} In this case, by the choice of $c_i$ in \eqref{eq:lowest} and \eqref{eq:range_c}, we have
\[|y_i-y_j| \ge \frac{1}{2}\left(d(M_i,M_j)+d(M_j,M_i) \right)\ge \frac{1}{2}\left(6c_i + 6c_j \right) = 3c_i + 3c_j,\]
which is a contradiction with \eqref{eq:dist_y}.

\smallskip
\noindent\emph{Case 5.2: $\overline{M_i} \cap \overline{M_j} \ne \emptyset$.} In this case, neither $M_i$ nor $M_j$ is of the lowest dimension in the stratification. By \eqref{eq:def_I_i}, we have $j\in I_i$ and $i\in I_j$.
Thus, by \eqref{eq:adjacent-fixed}, \begin{align*}
    |y_i-y_j|&\ge \frac{1}{2}(d(y_i,M_j)+d(y_j,M_i))
    \\&\ge \frac32 (r_i d(y_i,\partial M_i)^{\eta_i}+r_j d(y_j,\partial M_j)^{\eta_j})
    \\&\ge \frac32 \left(r_i\left(\frac{\underline{c}_i}{2}\right)^{\eta_i}\alpha_1^{\overline{\gamma}_i\eta_i}+r_j\left(\frac{\underline{c}_j}{2}\right)^{\eta_j}\alpha_2^{\overline{\gamma}_j\eta_j}\right)
    \\&\ge 3(c_i\alpha_1^{\beta_i}+c_j\alpha_2^{\beta_j}),
\end{align*}
where we used the definitions of $\underline c_i,\overline \gamma_i$, and the range of $c_i$ in \eqref{eq:range_c}. This is again a contradiction with \eqref{eq:dist_y}.

Combining the above two cases, we conclude \eqref{eq:no_intersect}.

\medskip
\noindent\underline{Step 6: Proof of \(M_i\subset \overline{M_x}\) for \(x\in\mathcal N_0(i,\alpha)\).}
Fix \(\alpha\in(0,1]\) and \(x\in \mathcal N_0(i,\alpha)\). By definition of \(\mathcal N_0(i,\alpha)\), there exists
\[
z\in M_i\setminus \bigcup_{j:M_j\subset \partial M_i} B(M_j,c_j\alpha^{\gamma_j}/2)
\]
such that
\[
|x-z|\le 2c_i\alpha^{\beta_i}.
\]

We first claim that
\[
\overline{M_x}\cap \overline{M_i}\neq \emptyset.
\]
Indeed, otherwise by the definition of $\delta_i$ in \eqref{eq:delta_i}, $\delta_i<\infty$ and
\[
|x-z|\ge d(M_x,M_i)\ge \delta_i,
\]
while by \eqref{eq:range_c},
\[
|x-z|\le 2c_i\alpha^{\beta_i}\le 2c_i \le 2\delta_i/6< \delta_i,
\]
a contradiction.

If \(M_x=M_i\), then the conclusion is immediate. Assume from now on that \(M_x\neq M_i\). We split the argument into two cases.

\smallskip
\noindent\emph{Case 6.1: one closure meets the other stratum.}
Assume
\[
\overline{M_i} \cap M_x \neq \emptyset
\qquad\text{or}\qquad
M_i \cap \overline{M_x} \neq \emptyset.
\]
By the frontier condition, either \(M_x\subset \partial M_i\) or \(M_i\subset\partial M_x\).

We claim that \(M_x\subset\partial M_i\) is impossible. Indeed, if \(M_x=M_j\) for some \(j\in J_i\), then
\[
z\notin B(M_j,c_j\alpha^{\gamma_j}/2),
\]
so
\[
|x-z|\ge d(z,M_j)\ge c_j\alpha^{\gamma_j}/2.
\]
On the other hand, since \(c_i\le \underline{c}_i/8\) by \eqref{eq:range_c} and \(\beta_i\ge \eta_i\overline\gamma_i\ge \gamma_j\) by \eqref{eq:beta_i_range},
\[
|x-z|\le 2c_i\alpha^{\beta_i}
\le \frac{\underline{c}_i}{4}\alpha^{\beta_i}
\le \frac{c_j}{4}\alpha^{\gamma_j},
\]
a contradiction. Therefore the only possibility is
\[
M_i\subset \partial M_x,
\]
and thus \(M_i\subset \overline{M_x}\).

\smallskip
\noindent\emph{Case 6.2: adjacent but disjoint.}
Assume now that
\begin{equation}\label{eq:adjacent}
M_i \cap \overline{M_x} = \overline{M_i} \cap M_x = \emptyset.
\end{equation}
Let \(M_x=M_j\) for some $j\in \llbracket 1,T\rrbracket$. By the conclusion from Step 5, we have
\[  \bigcup _{\alpha\in(0,1]}\mathcal{N}_0(i,\alpha)\cap\bigcup_{\alpha\in(0,1]}\mathcal{N}_0(j,\alpha)=\emptyset.\]
As $x\in M_j\subset \cup_{\alpha\in(0,1]}\mathcal{N}_0(j,\alpha)$, this is in contradiction with $x\in \mathcal N_0(i,\alpha)$.

Thus Case 6.2 is impossible, and we conclude that
\[
M_i\subset \overline{M_x}
\qquad\text{for all }x\in \mathcal N_0(i,\alpha).
\]
\medskip
\noindent\underline{Step 7: Proof of \eqref{eq:Lip_1}--\eqref{eq:Lip_3}.}
Fix \(i\in \llbracket 1,T\rrbracket\) and \(\alpha\in(0,1]\). For the lowest-dimensional strata, these estimates were already established in Step 2. So we assume \(\partial M_i\neq\emptyset\).

We begin with a lower bound on the distance to the boundary. Let \(z\in \mathcal N_0(i,\alpha)\). Then there exists
\[
y_z\in M_i\setminus \bigcup_{j:M_j\subset \partial M_i} B(M_j,c_j\alpha^{\gamma_j}/2)
\]
such that
\[
|z-y_z|\le 2c_i\alpha^{\beta_i}.
\]
Therefore
\begin{align}
d(z,\partial M_i)
&\ge d(y_z,\partial M_i)-|z-y_z| \nonumber\\
&\ge t_{i,\alpha}-2c_i\alpha^{\beta_i} \nonumber\\
&\ge \frac{\underline{c}_i}{2}\alpha^{\overline\gamma_i}-\frac{\underline{c}_i}{4}\alpha^{\beta_i} \nonumber\\
&\ge \frac{\underline{c}_i}{4}\alpha^{\overline\gamma_i}, \label{eq:dis-to-boundary-final}
\end{align}
where we used \(c_i\le \underline{c}_i/8\) and \(\beta_i\ge \overline\gamma_i\) from \eqref{eq:range_c} and \eqref{eq:beta_i_range}, respectively.

\smallskip
\noindent\emph{Step 7.1: Proof of \eqref{eq:Lip_1}.}
Let \(x,y\in \mathcal N_0(i,\alpha)\cap M_i\). By \eqref{eq: 3.5Rieman lip} and \eqref{eq:dis-to-boundary-final},
\[
|\nabla_{M_i}f(x)-\nabla_{M_i}f(y)|
\le \frac{C}{d(\{x,y\},\partial M_i)^{\iota}}|x-y|
\le \frac{C4^{\iota}}{\underline{c}_i^{\iota}}\,
\alpha^{-\overline\gamma_i\iota}|x-y|.
\]
Thus \eqref{eq:Lip_1} holds with
\[
c_{1,i}:=\frac{C4^{\iota}}{\underline{c}_i^{\iota}},
\qquad
\omega_{1,i}:=\overline\gamma_i\iota.
\]

\smallskip
\noindent\emph{Step 7.2: Proof of \eqref{eq:Lip_2}.}
Let \(x\in \mathcal N_0(i,\alpha)\), \(y\in \mathcal N_0(i,\alpha)\cap M_i\), and \(v\in\partial f(x)\). We distinguish two cases.

\smallskip
\noindent\underline{Case 7.2(a): \(x\in M_i\).}
Then, by \eqref{eq: ball contains ball 2}, both \(x\) and \(y\) belong to
\[
B\Bigl(
M_i\setminus B(\partial M_i,t_{i,\alpha}),
\ \frac{\overline r_i\varrho_i^{\overline\eta_i}}{2}\,t_{i,\alpha}^{\theta_i\overline\eta_i}
\Bigr).
\]
Applying \eqref{eq: tilde F ineq} with \(j=1,2\), we obtain
\[
|P_{T_{M_i}(y)}-P_{T_{M_i}(x)}|
\le \frac{\widetilde c_{1,i}}{t_{i,\alpha}^{\widetilde\omega_{1,i}}}|x-y|,
\qquad
|\nabla_{M_i}f(y)-\nabla_{M_i}f(x)|
\le \frac{\widetilde c_{2,i}}{t_{i,\alpha}^{\widetilde\omega_{2,i}}}|x-y|.
\]
By Lipschitz continuity of $f$ on $X$ (see \eqref{eq:Lip_L0}) and the fact that \(
P_{T_{M_i}(x)}(v)=\nabla_{M_i}f(x)
\)
due to \cite[Proposition 4]{bolte2007clarke}, we have
\begin{align*}
|P_{T_{M_i}(y)}(v)-\nabla_{M_i}f(y)|
&\le |P_{T_{M_i}(y)}(v)-P_{T_{M_i}(x)}(v)|
   + |\nabla_{M_i}f(x)-\nabla_{M_i}f(y)| \\
&\le \left(
\frac{L_0\widetilde c_{1,i}}{t_{i,\alpha}^{\widetilde\omega_{1,i}}}
+\frac{\widetilde c_{2,i}}{t_{i,\alpha}^{\widetilde\omega_{2,i}}}
\right)|x-y|.
\end{align*}
Using
\(
t_{i,\alpha}\ge \underline{c}_i\alpha^{\overline\gamma_i}/2
\) from \eqref{eq:t_ia},
we obtain
\[
|P_{T_{M_i}(y)}(v)-\nabla_{M_i}f(y)|
\le
\left(
L_0\widetilde c_{1,i}\Bigl(\frac{2}{\underline{c}_i}\Bigr)^{\widetilde\omega_{1,i}}
\alpha^{-\overline\gamma_i\widetilde\omega_{1,i}}
+
\widetilde c_{2,i}\Bigl(\frac{2}{\underline{c}_i}\Bigr)^{\widetilde\omega_{2,i}}
\alpha^{-\overline\gamma_i\widetilde\omega_{2,i}}
\right)|x-y|.
\]
Hence \eqref{eq:Lip_2} holds in this case with
\[
\omega_{2,i}^{(1)}:=\overline\gamma_i\max\{\widetilde\omega_{1,i},\widetilde\omega_{2,i}\}.
\]

\smallskip
\noindent\underline{Case 7.2(b): \(x\notin M_i\).}
By Step 5, we have \(M_i\subset \partial M_x\). Applying \eqref{eq: 3.5verdier} to the pair
\((M_x,M_i)\), we get
\[
|P_{T_{M_i}(y)}(v)-\nabla_{M_i}f(y)|
\le \frac{C}{d(y,\partial M_i)^{\iota}}|x-y|.
\]
Using \eqref{eq:dis-to-boundary-final},
\[
|P_{T_{M_i}(y)}(v)-\nabla_{M_i}f(y)|
\le
\frac{C4^{\iota}}{\underline{c}_i^{\iota}}
\alpha^{-\overline\gamma_i\iota}|x-y|.
\]

Combining Cases 7.2(a) and 7.2(b), we conclude that \eqref{eq:Lip_2} holds with
\[
\omega_{2,i}:=\max\{\omega_{2,i}^{(1)},\overline\gamma_i\iota\}
\]
and some \(c_{2,i}>0\).

\smallskip
\noindent\emph{Step 7.3: Proof of \eqref{eq:Lip_3}.}
Let \(x,y\in \mathcal N_0(i,\alpha)\). By \eqref{eq: ball contains ball 2},
both \(x\) and \(y\) belong to
\[
B\Bigl(
M_i\setminus B(\partial M_i,t_{i,\alpha}),
\ \frac{\overline r_i\varrho_i^{\overline\eta_i}}{2}\,t_{i,\alpha}^{\theta_i\overline\eta_i}
\Bigr).
\]
Applying \eqref{eq: tilde F ineq} with \(j=3\), we have
\[
|DP_{M_i}(x)-DP_{M_i}(y)|
\le \frac{\widetilde c_{3,i}}{t_{i,\alpha}^{\widetilde\omega_{3,i}}}|x-y| \le \widetilde c_{3,i}\Bigl(\frac{2}{\underline{c}_i}\Bigr)^{\widetilde\omega_{3,i}}
\alpha^{-\overline\gamma_i\widetilde\omega_{3,i}}|x-y|,
\]
using again \(t_{i,\alpha}\ge \frac{\underline{c}_i}{2}\alpha^{\overline\gamma_i}\). Thus \eqref{eq:Lip_3} holds with
\[
c_{3,i}:=\widetilde c_{3,i}\Bigl(\frac{2}{\underline{c}_i}\Bigr)^{\widetilde\omega_{3,i}},
\qquad
\omega_{3,i}:=\overline\gamma_i\widetilde\omega_{3,i}.
\]

Finally, by the choice of \(\eta_i\),
\[
\eta_i\ge 4\max\{\iota,\widetilde\omega_{1,i},\widetilde\omega_{2,i},\widetilde\omega_{3,i}\},
\]
and since \(\beta_i\ge \eta_i\overline\gamma_i\), we obtain
\[
\beta_i\ge 4\overline\gamma_i\iota=4\omega_{1,i},
\qquad
\beta_i\ge 4\overline\gamma_i\widetilde\omega_{3,i}=4\omega_{3,i},
\]
and also
\[
\beta_i\ge 4\overline\gamma_i\max\{\widetilde\omega_{1,i},\widetilde\omega_{2,i},\iota\}\ge 4\omega_{2,i}.
\]
Therefore
\[
\omega_{1,i},\omega_{2,i},\omega_{3,i}\in[0,\beta_i/4].
\]

This completes the proof.
\end{proof}

\section{Bounding the diameter of subgradient sequences}\label{sec:main_proof}

In this section, we prove Theorem~\ref{thm:diameter}. The argument builds on the decomposition results established in Section~\ref{sec:decomposition}. There, the domain is partitioned into strata on which the subdifferential is related to the Riemannian gradient, and for each stratum we construct neighborhoods where the projection map is well defined and its derivatives are controlled. This allows us to first analyze the subgradient sequence while it remains near a fixed stratum, and then to combine these local estimates as the sequence moves between different strata.

In Section~\ref{sec:length_formula}, we derive a local length estimate for the projection of a subgradient sequence onto a fixed stratum. In Section~\ref{sec:RM_def}, we introduce the indices and the notion of relative length that describe how the sequence evolves with respect to the stratification. In Section~\ref{sec:bound_rm}, we bound this relative length by combining the local estimates on individual strata with the control of transitions between strata. Finally, in Section~\ref{sec:final_proof}, we assemble these ingredients and complete the proof of Theorem~\ref{thm:diameter}.
\subsection{Length estimates of projected sequences}\label{sec:length_formula}
Let \(M\subset \mathbb{R}^n\) be a bounded \(C^2\) manifold, and let
\(g:M\to \mathbb{R}\) be a \(C^1\) definable function. By the
\L{}ojasiewicz gradient inequality \cite{lojasiewicz1958,kurdyka1994wf}
(see also \cite[Theorem 11]{bolte2007clarke}), there exist
\(\epsilon_+>0\), \(\theta_+\in[0,1)\), and \(\mu_+>0\) such that, defining
\[
\psi_+(t):=\frac{t^{1-\theta_+}}{(1-\theta_+)\mu_+},
\qquad t\in[0,\epsilon_+],
\]
one has
\[
|\nabla g(x)|\ge \frac{1}{\psi_+'(g(x))}
\]
for all \(x\in M\) satisfying \(0<g(x)\le \epsilon_+\).

Applying the \L{}ojasiewicz gradient inequality to \(-g\), there likewise exist
\(\epsilon_->0\), \(\theta_-\in[0,1)\), and \(\mu_->0\) such that, defining
\[
\psi_-(t):=\frac{t^{1-\theta_-}}{(1-\theta_-)\mu_-},
\qquad t\in[0,\epsilon_-],
\]
one has
\[
|\nabla g(x)|=|\nabla(-g)(x)|\ge \frac{1}{\psi_-'(-g(x))}
\]
for all \(x\in M\) satisfying \(-\epsilon_-\le g(x)<0\).

Set
\[
\epsilon:=\min\{\epsilon_+,\epsilon_-,1\},\qquad
\theta:=\max\{\theta_+,\theta_-\},\qquad
\mu:=\min\{\mu_+,\mu_-\},
\]
and define \(\psi:\mathbb R\to\mathbb{R}\) by
\[
\psi(t):=\mathrm{sgn}(t)\frac{|t|^{1-\theta}}{(1-\theta)\mu}.
\]
Then \(\psi'(t)>0\) for all \(t\neq 0\), and
\begin{equation}\label{eq:signed_KL}
    |\nabla g(x)|\ge \frac{1}{\psi'(g(x))}=\mu |g(x)|^\theta
\end{equation}
for all \(x\in M\) such that \(0<|g(x)|\le \epsilon\). Since \(\epsilon\le 1\),
this inequality remains valid after increasing \(\theta\). We may therefore
assume from now on that \(\theta\in(0,1)\). With this choice, the above
inequality also holds when \(g(x)=0\). We will refer to any such function
\(\psi\) as a \emph{desingularizing function} of \(g\) on $[|g|\le \epsilon]$.

Recall that bounded continuous-time subgradient trajectories converge because their lengths are finite and are governed by function values through desingularizing functions \cite{lojasiewicz1982trajectoires,kurdyka1998gradients,bolte2010characterizations}. In contrast, the length of subgradient sequences is not finite. Inspired by recent works on the active saddle avoidance properties of stochastic subgradient sequences \cite{bianchi2023stochastic,davis2025active}, we study the projections of subgradient sequences onto a smooth manifold where $f$ is smooth, as outlined in the following lemma.

\begin{lemma}\label{lemma:drop of fy}
Let \(f:\mathbb{R}^n\to\mathbb{R}\), \(M\subset \mathbb{R}^n\) be a \(C^3\) manifold, and assume that \(f|_M\) is \(C^2\).
Assume that there exist \(L> 1\), \(U\subset \mathbb{R}^n\), and \(\epsilon>0\) such that
\begin{itemize}
    \item \(P_M\) is single-valued, \(C^2\), and \(L\)-Lipschitz continuous on \(U\),
    \item \(f\) is \(L\)-Lipschitz continuous on \(U\), and $f|_{M}$ is $L$-Lipschitz continuous on $P_M(U)$,
    \item \(\psi:\mathbb{R}\to\mathbb{R}\) is a desingularizing function of \(f|_M\) on
    \([ |f|\le \epsilon]\).
\end{itemize}

Let \((x_k)_{k=0}^K\) be a subgradient sequence with step sizes \((\alpha_k)_{k=0}^{K-1}\), and set
\[
y_k:=P_M(x_k),
\qquad
d_k:=d(x_k,M),
\qquad k\in \llbracket 0,K\rrbracket.
\]
Assume that
\[
 B(x_k,\alpha_k L)\subset U
\qquad\text{and}\qquad
y_k\in [ |f|\le \epsilon]
\quad \text{for all } k\in \llbracket 0,K-1\rrbracket.
\]
Assume moreover that there exist constants \(L_{f,k},L_{V,k},L_{P,k}\ge L\) for
\(k\in \llbracket 0,K-1\rrbracket\) such that, for each such \(k\),
\begin{enumerate}
    \item $|\nabla_M f(x)-\nabla_M f(x')|
    \le L_{f,k}|x-x'|$, for all $x,x'\in B(y_k,\alpha_kL^2)\cap M$,
    \item $|P_{T_M(y_k)}(v_k)-\nabla_M f(y_k)|
    \le L_{V,k}|x_k-y_k|$, for all $v_k\in \partial f(x_k)$,

    \item $|DP_M(x)-DP_M(x')|
    \le L_{P,k}|x-x'|$ for all $x,x'\in B(x_k,\alpha_kL)\cup\{y_k\}$.
\end{enumerate}

Let \(g_0,\dots,g_K\) be any nonnegative scalars such that
\[
g_k-g_{k+1}
\ge
\frac{\alpha_kL_{V,k}^2d_k^2}{2}
+L^2\alpha_kL_{P,k}d_k
+\frac{L^4\alpha_k^2}{2}(L_{f,k}+L_{P,k})
\qquad
\forall k\in \llbracket 0,K-1\rrbracket .
\]
Define $z_k:=f(y_k)+g_k$ for $k\in \llbracket 0,K\rrbracket$. Then
\(
z_0\ge z_1\ge \cdots \ge z_K
\)
and
\begin{align*}
\sum_{k=0}^{K-1}|y_{k+1}-y_k|
&\le 2L\bigl(\psi(z_0)-\psi(z_K)\bigr) \\
&\quad
+\sum_{k=0}^{K-1}
\left(
L^3L_{f,k}\alpha_k^2
+L\alpha_kL_{V,k}d_k
+\frac{L\alpha_k}{\psi'(g_k)}
\right)
+L^2\max_{k\in \llbracket 0,K-1\rrbracket}\alpha_k .
\end{align*}
\end{lemma}
\begin{proof}
Since $B(x_k,\alpha_kL)\subset U$, the map $P_M$ is $C^2$ on $B(x_k,\alpha_kL)$.
As $f|_M$ is $C^2$, the composition $f\circ P_M$ is $C^2$ on $B(x_k,\alpha_kL)$. For any $x,x'\in B(x_k,\alpha_k L)$, we have
\begin{subequations}\label{eq:gradient_holder}
    \begin{align}
        &| \nabla (f\circ P_M)(x)-  \nabla (f\circ P_M)(x')|\\
        =& |D P_M(x) \nabla_M f(P_M(x)) - D P_M(x') \nabla_M f(P_M(x'))|\\
        \le& |D P_M(x) (\nabla_M f(P_M(x)) - \nabla_M f(P_M(x'))) | \notag\\
        &\quad + |(D P_M(x) - D P_M(x')) \nabla_M f(P_M(x'))|\\
        \le&L\times L_{f,k} |P_M(x) - P_M(x')|
        + L_{P,k} |x - x'|\times L\label{eq:gradient_holder_divide}\\
        \le& L^{2}L_{f,k} |x - x'| + LL_{P,k} |x - x'|\\
        \le& L^2(L_{f,k}+L_{P,k}) |x - x'|.
    \end{align}
\end{subequations}
Since $x_{k+1}\in B(x_k,\alpha_kL)$, a bound on the Taylor expansion of $f\circ P_M$ yields
\begin{subequations}
\label{eq:approximate f(y)}
        \begin{align}
        &f(y_{k+1})-f(y_k)\\
        =& f\circ P_M(x_{k+1}) - f\circ P_M(x_{k})\\
        \le& \langle x_{k+1}-x_k,\nabla (f\circ P_M)(x_k)\rangle +\frac{L^2(L_{f,k}+L_{P,k})}{2} |x_{k+1}-x_k|^{2} \label{eq:taylor-c}\\
        =&-\alpha_k\langle v_k,\nabla (f\circ P_M)(x_k)\rangle +\frac{L^2(L_{f,k}+L_{P,k})}{2}\alpha_k^{2}|v_k|^{2}\label{eq:taylor-d}\\
        \le& -\alpha_k\langle v_k,\nabla_Mf(y_k)\rangle -\alpha_k \langle  v_k,(D P_M(x_k) - D P_M(y_k))\nabla_M f(y_k)\rangle \notag\\
        &\quad + \frac{L^4(L_{f,k}+L_{P,k})}{2}\alpha_k^{2}\label{eq:taylor-e}\\
        \le & -\alpha_k\langle v_k,\nabla_Mf(y_k)\rangle+\alpha_k L^2L_{P,k}  d_k  + \frac{L^4(L_{f,k}+L_{P,k})}{2}\alpha_k^{2}\label{eq:taylor-f}
    \end{align}
\end{subequations}
where $v_k\in \partial f(x_k)$. Above, \eqref{eq:taylor-c} is due to the local Lipschitz continuity of $\nabla(f\circ P_M)$. \eqref{eq:taylor-e} follows from the fact that $\nabla (f\circ P_M)(y_k) = DP_M(y_k) \nabla_M f(y_k) = P_{T_M(y_k)}(\nabla_M f(y_k)) = \nabla_M f(y_k)$ by \cite[Theorem 4.1]{dudek1994nonlinear}. Finally, \eqref{eq:taylor-f} follows from the Lipschitz continuity of $D P_M$. Since $|P_{T_{M}(y_k)}(v_k)-\nabla_Mf(y_k)| \le L_{V,k} d_k$ for all $v_k\in \partial f(x_k)$, we have
    \begin{align*}
         L_{V,k}^2 d_k^2
         &\ge |P_{T_{M}(y_k)}(v_k)-\nabla_Mf(y_k)|^2\\
         &=|P_{T_{M}(y_k)}(v_k)|^2+|\nabla_M f(y_k)|^2 -2\langle P_{T_{M}(y_k)}(v_k),\nabla_Mf(y_k)\rangle\\
         &= |P_{T_{M}(y_k)}(v_k)|^2+|\nabla_M f(y_k)|^2 -2\langle v_k,\nabla_Mf(y_k)\rangle
    \end{align*}
as $\nabla_M f(y_k)\in T_{M}(y_k)$.
Combining the above two inequalities, we have
    \begin{align*}
    2(f(y_{k+1})-f(y_k))
    &\le -\alpha_k|P_{T_{M}(y_k)}(v_k)|^2-\alpha_k|\nabla_Mf(y_k)|^2\\
    &\quad +\alpha_k L_{V,k}^2  d_k^2+2\alpha_k L^2L_{P,k}  d_k\\
    &\quad + \alpha_k^2 L^4(L_{f,k}+L_{P,k})\\
    &\le \alpha_k L_{V,k}^2  d_k^2+2L^2\alpha_k L_{P,k}  d_k \\
    &\quad + L^4\alpha_k^2 (L_{f,k}+L_{P,k})-\alpha_k|\nabla_Mf(y_k)|^2\\
    &\le 2g_k-2g_{k+1}-\alpha_k|\nabla_Mf(y_k)|^2,
\end{align*}
where $g_0 \ge  g_1\ge\cdots \ge g_K \ge 0$ are any scalars that satisfy
\begin{equation*}
    g_{k} - g_{k+1}
    \ge \alpha_k L_{V,k}^2  d_k^2/2
    +L^2\alpha_k L_{P,k}  d_k
    + L^4\alpha_k^2 (L_{f,k}+L_{P,k})/2
\end{equation*}
for $k\in \llbracket 0,K-1\rrbracket$. Thus,
\begin{equation}\label{eq:gradient descent}
    \frac{1}{2} \alpha_k|\nabla_Mf(y_k)|^2 \le z_k - z_{k+1},
\end{equation}
where $z_k:= f(y_k) + g_k$. If follows that $z_k$ is decreasing. We first suppose that $(z_{K},z_0)$ excludes $0$. Since $z_0\ge  z_1\ge  \dots \ge  z_{K}$, either $z_0\le 0$ or $z_K \ge 0$. In the first case, we have
\begin{subequations}
\label{eq:first case delta phi}
\begin{align}
    \psi(z_k)-\psi(z_{k+1}) &= \psi(|z_{k+1}|) - \psi(|z_k|)\\
    &\ge  (|z_{k+1}|-|z_{k}|)\psi'(|z_{k+1}|)\label{eq:paritition 1}
    \\&=(z_k-z_{k+1})\psi'(|f(y_{k+1})+g_{k+1}|))\label{eq:paritition 2}
    \\&\ge  (z_k-z_{k+1})\psi'(|f(y_{k+1})|+g_{k+1})\label{eq:paritition 3}
    \\&\ge  \frac{\frac{1}{2}\alpha_k|\nabla_Mf(y_k)|^2}{1/\psi'(|f(y_{k+1})|)+1/\psi'(g_{k+1})}\label{eq:paritition 4}
    \\&\ge  \frac{\frac{1}{2}\alpha_k|\nabla_Mf(y_k)|^2}{|\nabla_Mf(y_{k+1})|+1/\psi'(g_{k+1})}.\label{eq:paritition 5}
\end{align}
\end{subequations}
Indeed, \eqref{eq:paritition 1} and \eqref{eq:paritition 3} hold because $\psi$ is concave on $[0,\infty)$. \eqref{eq:paritition 4} is due to the subadditivity of $1/\psi'$ on $[0,\infty)$. We apply the definition of desingularizing function to get the last inequality \eqref{eq:paritition 5}.  By the AM-GM inequality, we have

    \begin{align*}
    |\nabla_Mf(y_k)|
    &\le\sqrt{\frac{2}{\alpha_k}(\psi(z_k)-\psi(z_{k+1}))\left(|\nabla_Mf(y_{k+1})|+1/\psi'(g_{k+1})\right)}
    \\
    &\le\frac{1}{\alpha_k}(\psi(z_k)-\psi(z_{k+1}))
    +\frac{1}{2}\left(|\nabla_Mf(y_{k+1})|+1/\psi'(g_{k+1})\right).
\end{align*}
In addition,
    \begin{align*}
    |\nabla_Mf(y_{k+1})|
    &\le|\nabla_Mf(y_k)|+L_{f,k}|y_{k+1}-y_k|\\
        &= |\nabla_Mf(y_k)|+L_{f,k}|P_M(x_{k+1}) - P_M(x_{k})| \\
        &\le |\nabla_Mf(y_k)|+LL_{f,k}|x_{k+1} - x_k|\\
        &\le |\nabla_Mf(y_k)|+L^2L_{f,k} \alpha_k.
    \end{align*}
Thus, since $1/\psi'(g_{k})$ is decreasing, we have 
    \begin{align*}
        \frac{1}{2}\alpha_k|\nabla_Mf(y_k)|
        \le \psi(z_k)-\psi(z_{k+1})
        +\frac{L^2L_{f,k} \alpha_k^2+\alpha_k/\psi'(g_{k})}{2}.
    \end{align*}
Telescoping yields

\begin{align}\label{eq:first case telescope}
   \sum_{k=0}^{K-1}\alpha_k|\nabla_Mf(y_k)|
   &\le 2\left(\psi(z_0)-\psi(z_K)\right)  +\sum_{k=0}^{K-1}
   \left(
   L^2L_{f,k} \alpha_k^2+\frac{\alpha_k}{\psi'(g_{k})}
   \right).
\end{align}

In the second case, i.e. $z_K\ge 0$, following similar arguments as in \eqref{eq:first case delta phi}, we have
\begin{align*}
    \psi(z_k)-\psi(z_{k+1})
    \ge  \frac{\frac{1}{2}\alpha_k|\nabla_Mf(y_k)|^2}{|\nabla_Mf(y_k)|+1/\psi'(g_{k})}.
\end{align*}
Note that 
    \begin{align*}
    |\nabla_Mf(y_k)|
    &\le  \sqrt{\frac{2}{\alpha_k}(\psi(z_{k})-\psi(z_{k+1}))(|\nabla_Mf(y_{k})|+1/\psi'(g_{k}))}
    \\
    &\le  \frac{1}{\alpha_k}(\psi(z_{k})-\psi(z_{k+1}))
    +\frac{1}{2}(|\nabla_Mf(y_{k})|+1/\psi'(g_{k})).
\end{align*}
Thus, we have
\begin{align*}
    \alpha_k|\nabla_Mf(y_k)|
    \le \psi(z_{k})-\psi(z_{k+1})
    +\frac{\alpha_k}{2}(|\nabla_Mf(y_{k})|+1/\psi'(g_{k})).
\end{align*}
Telescoping yields
\begin{align}\label{eq:second case telescope}
   \sum_{k=0}^{K-1}\alpha_k|\nabla_Mf(y_k)|
   &\le   2(\psi(z_0)-\psi(z_K))
   +\sum_{k=0}^{K-1}\frac{\alpha_k}{\psi'(g_{k})}.
\end{align}
In view of \eqref{eq:first case telescope} and \eqref{eq:second case telescope}, we have \eqref{eq:first case telescope} holds as long as $(z_K,z_0)$ excludes 0.

Now suppose that $z_{0}\ge \dots\ge  z_{\underline{K}}>0 = z_{\underline{K}+1}=\dots=z_{\overline{K}}>z_{\overline{K}+1}\ge \dots\ge  z_{K}$, for some $0<\underline{K}\le  \overline{K}<K$. According to \eqref{eq:gradient descent}, we have $\alpha_k|\nabla_M f(y_k)|=0$ for all $\underline{K}+1\le  k<\overline{K}$. As both $(z_{\underline{K}},z_0)$ and $(z_{K},z_{\overline{K}})$ exclude $0$, we have
    \begin{align*}
        \sum_{k=0}^{K-1}\alpha_k|\nabla_Mf(y_k)|
        &=\sum_{k=0}^{\underline{K}-1}\alpha_k|\nabla_Mf(y_k)|
        +\alpha_{\underline{K}}|\nabla_Mf(y_{\underline{K}})|  + \sum_{k=\overline{K}}^{K-1}\alpha_k|\nabla_Mf(y_k)|\\
        &\le 2\left(\psi(z_0)-\psi(z_{\underline{K}})\right)
        + 2\left(\psi(z_{\overline{K}})-\psi(z_{K})\right) \\
        &\quad +\sum_{k=0}^{K-1}
        \left(
        L^2L_{f,k} \alpha_k^2+\frac{\alpha_k}{\psi'(g_{k})}
        \right)
        +L\max_{k\in \llbracket 0,K-1\rrbracket}\alpha_k\\
        &\le 2\left(\psi(z_0)-\psi(z_{K})\right)
        +\sum_{k=0}^{K-1}
        \left(
        L^2L_{f,k} \alpha_k^2+\frac{\alpha_k}{\psi'(g_{k})}
        \right)
        +L\max_{k\in \llbracket 0,K-1\rrbracket}\alpha_k.
    \end{align*}
Finally,  we have
\begin{subequations}
    \begin{align}
        |y_{k+1} - y_k| &= | P_M(x_{k+1})-P_M(x_{k})|\\
        &=| P_M(x_{k}-\alpha_k v_k)-P_M(x_{k}-\alpha_k P_{N_{M}(y_k)}(v_k))|\label{eq:projection eq}\\
        &\le L\alpha_k|v_k - P_{N_{M}(y_k)}(v_k)|\\
        &=L\alpha_k| P_{T_{M}(y_k)}(v_k)|\\
        &\le L\alpha_k| P_{T_{M}(y_k)}(v_k) - \nabla_M f(y_k)| \notag\\
        &\quad + L\alpha_k |\nabla_M f(y_k)|\\
        &\le L\alpha_k L_{V,k} d_k+ L\alpha_k |\nabla_M f(y_k)|,
    \end{align}
\end{subequations}
where \eqref{eq:projection eq} holds due to \cite[(3.13) Theorem]{dudek1994nonlinear} and $x_k-\alpha_kP_{N_M(y_k)}(v_k)\in B(x_k,\alpha_kL)\subset U$. Therefore, in all of the above cases, we have
    \begin{align*}
        \sum_{k=0}^{K-1}|y_{k+1}-y_k|
        &\le \sum_{k=0}^{K-1}  L\alpha_k L_{V,k} d_k+ \alpha_kL |\nabla_M f(y_k)|\\
        &\le 2L\left(\psi(z_0)-\psi(z_{K})\right)
        +\sum_{k=0}^{K-1}L^3L_{f,k} \alpha_k^2 \\
        &\quad +\sum_{k=0}^{K-1}
        \left(
        L \alpha_k L_{V,k} d_k 
        +\frac{L\alpha_k}{\psi'(g_{k})}
        \right)
        +L^2\max_{k\in \llbracket 0,K-1\rrbracket}\alpha_k
        \\
        &=2L\left(\psi(z_0)-\psi(z_{K})\right)+\sum_{k=0}^{K-1}
        \left(
        L^3L_{f,k} \alpha_k^2
        +L\alpha_k L_{V,k} d_k 
        +\frac{L\alpha_k}{\psi'(g_{k})}
        \right) \\
        &\quad +L^2\max_{k\in \llbracket 0,K-1\rrbracket}\alpha_k.
    \end{align*}
\end{proof}
Lemma \ref{lemma:drop of fy} allows one to estimate the distance traveled by the iterates along a smooth manifold using the variation on function values, given that the iterates stay close to the manifold. This result has its counterpart in the literature for smooth functions (for e.g., \cite[Proposition 8]{josz2023global}\cite[Proposition 4.12]{josz2024proximal}), which also study the length of trajectories via change in function values. The main difference is that those results apply directly to the sequence $(x_k)_{k\in \mathbb{N}}$, while Lemma \ref{lemma:drop of fy} holds only for the projected sequence $(y_k)_{k\in \mathbb{N}}$. In the proof of Lemma \ref{lemma:drop of fy}, we derive an approximate descent property of the projected sequence (see equation \eqref{eq:gradient descent}). The upper bound on its length then follows by a similar line of reasoning as in \cite[Theorem 3.6]{li2023convergence} and \cite[Proposition 4.12]{josz2024proximal}, which analyze the sequences of (proximal) random reshuffling algorithms that share a similar property \cite[Lemma 3.2]{li2023convergence}.

While the assumptions posed by Lemma \ref{lemma:drop of fy} appear to be quite complicated, they can be satisfied if we consider the decomposition of definable functions studied in Section \ref{sec:decomposition}. This is the object of the following corollary, where we verify the assumptions required in Lemma \ref{lemma:drop of fy} and obtain a simpler expression for the length of the projected sequence. Note that the regions $\mathcal N(\cdot,\cdot)$ defined in the following corollary are different from $\mathcal N_0(\cdot,\cdot)$ given by Proposition \ref{proposition:neighborhoods of cells}.

\begin{corollary}\label{cor:y_k^i}
Let \(f:\mathbb{R}^n\to\mathbb{R}\) be a locally Lipschitz definable function, and let \(X\subset \mathbb{R}^n\) be definable and compact. Then the following holds:
\begin{enumerate}
    \item There exist constants
\[
C,\hat\beta>0,
\qquad
\epsilon,\hat\alpha,\theta\in(0,1),
\]
a stratification \(\{M_1,\ldots,M_T\}\) of \(X\), and constants \(c_i,\beta_i,\gamma_i>0\), \(i\in \llbracket 1,T\rrbracket\), such that
\[
\hat\beta<\beta_i<\gamma_i(1-\theta)<1
\qquad
\text{for all } i\in \llbracket 1,T\rrbracket,
\]
\[
\forall i,j\in \llbracket 1,T\rrbracket,
\qquad
M_j\subset \partial M_i \implies \beta_i>\gamma_j.
\]
\item For any $i,j\in\llbracket 1,T\rrbracket$ such that $\overline{M_i}\cap M_j=M_i\cap\overline{M_j}=\emptyset$, we have
\[\bigcup_{\alpha\in(0,1]}\mathcal{N}(i,\alpha)\cap \bigcup_{\alpha\in(0,1]}\mathcal{N}(j,\alpha)=\emptyset,\]
where
\[
\mathcal N(i,\alpha)
:=
\bigl(X\cap [|f|\le \epsilon]\bigr)
\cap
\left(
B(M_i,c_i\alpha^{\beta_i})
\setminus
\bigcup_{j:\, M_j\subset \partial M_i} B(M_j,c_j\alpha^{\gamma_j})
\right).
\]
\item For any \(i\in \llbracket 1,T\rrbracket\), any \(K\in \mathbb N\), and any subgradient sequence \((x_k)_{k\in \mathbb N}\) with step sizes \((\alpha_k)_{k\in \mathbb N}\) satisfying
\[
\alpha_k\in (0,\hat\alpha]
\qquad\text{for all } k\in \llbracket 0,K-1\rrbracket,
\]
and
\[
x_k\in \mathcal N(i,\alpha_k)
\qquad\text{for all } k\in \llbracket 0,K\rrbracket,
\]
we have
\begin{equation*}
\mathrm{sgn}(z_0^i)|z_0^i|^{1-\theta}
-\mathrm{sgn}(z_K^i)|z_K^i|^{1-\theta}
\ge
\frac{1}{C}\sum_{k=0}^{K-1}|y_{k+1}^i-y_k^i|
-\sum_{k=0}^{K-1}\left(\alpha_k^{1+\hat\beta}+\alpha_k g_k^\theta\right)
-\max_{k\in \llbracket 0,K-1\rrbracket}\alpha_k,
\end{equation*}
where
\[
y_k^i:=P_{M_i}(x_k),
\qquad
z_k^i:=f(y_k^i)+g_k,
\]
and \(g_0,\ldots,g_K\) are any nonnegative scalars such that
\[
g_k-g_{k+1}\ge C\alpha_k^{1+\hat\beta}
\quad\text{for all } k\in \llbracket 0,K-1\rrbracket.
\]
In addition, $z_0^i\ge z_1^i\ge \cdots \ge z_K^i$.
\end{enumerate}
\end{corollary}
\begin{proof}
We organize the proof into six steps. In Step~1, we choose a desingularizing function that works uniformly for the restrictions of the objective function to all strata. In Step~2, for each stratum \(M_i\), we fix the constants \(c_i\), \(\beta_i\), and \(\gamma_i\), and apply Proposition~\ref{proposition:neighborhoods of cells} to obtain the corresponding neighborhoods and estimates. In Step~3, we choose the upper bound \(\hat{\alpha}\) for the step sizes so that all subsequent inclusions and estimates hold uniformly. In Step 4, we prove that the given neighborhoods satisfy the nonintersecting conditions. In Step~5, we verify the assumptions of Lemma~\ref{lemma:drop of fy} for iterates lying in \(\mathcal N(i,\alpha_k)\). In Step~6, we apply Lemma~\ref{lemma:drop of fy} and simplify the resulting error terms to conclude the proof.

\medskip
\noindent\underline{Step 1: Uniform desingularizing function on the strata.}
Applying Proposition~\ref{prop:function_Lipschitz} with \(m=3\) and \(U=X\), we obtain
a stratification \(\{M_1,\ldots,M_T\}\) of \(X\) such that \(f|_{M_i}\) is \(C^3\)
for every \(i\in \llbracket 1,T\rrbracket\).

For each \(i\in \llbracket 1,T\rrbracket\), apply the signed
\L{}ojasiewicz gradient inequality in \eqref{eq:signed_KL} to the function
\(f|_{M_i}\). Since there are only finitely many strata, after shrinking the
corresponding neighborhoods and enlarging the exponents if necessary, we may choose
constants
\[
\epsilon\in(0,1),\qquad \theta\in(0,1),\qquad \mu>0
\]
such that
\[
\psi(t):=\frac{\operatorname{sgn}(t)|t|^{1-\theta}}{(1-\theta)\mu}
\]
is a desingularizing function of \(f|_{M_i}\) on \([|f|\le 2\epsilon]\)
for every \(i\in \llbracket 1,T\rrbracket\). 

\medskip
\noindent\underline{Step 2: Choice of the constants.}
Fix \(L>1\) such that \(f\) is \(L\)-Lipschitz on \(X\) and $f|_{M_i}$ is $L$-Lipschitz for each $i\in \llbracket 1,T\rrbracket$. Applying Proposition~\ref{proposition:neighborhoods of cells} to the stratification $\{M_i\}$ yields constants \(\eta_i\ge 1\), \(i\in\llbracket 1,T\rrbracket\). We now choose
\(\gamma_i\) and \(\beta_i\) so that the admissibility condition
\[
M_j\subset \partial M_i \implies \eta_i\gamma_j\le \beta_i
\]
is satisfied, and
\[
0<\beta_i<\gamma_i(1-\theta)<1,
\qquad
M_j\subset \partial M_i \implies \beta_i>\gamma_j.
\]

Set
\[
d_*:=\max_{i\in\llbracket 1,T\rrbracket}\dim M_i,
\qquad
\eta_*:=\max_{i\in\llbracket 1,T\rrbracket}\eta_i,
\]
and choose
\[
\delta\in\Bigl(0,\frac{1-\theta}{2\eta_*}\Bigr).
\]
For each \(i\in\llbracket 1,T\rrbracket\), define
\[
\gamma_i:=\delta^{\,d_*-\dim M_i+1}.
\]
Then \(0<\gamma_i<1\) for all \(i\). Moreover, if \(M_j\subset \partial M_i\), then
\(\dim M_j<\dim M_i\), and hence
\[
\gamma_j
=
\delta^{\,d_*-\dim M_j+1}
\le
\delta^{\,d_*-\dim M_i+2}
=
\delta\,\gamma_i.
\]
Therefore, for every such pair \((i,j)\),
\[
\eta_i\gamma_j
\le
\eta_*\delta\,\gamma_i
\le 
\frac{1-\theta}{2}\gamma_i
<
(1-\theta)\gamma_i.
\]
Hence, for each \(i\in\llbracket 1,T\rrbracket\), we may choose
\[
\beta_i\in\Bigl(
\max\{0,\sup_{j:\,M_j\subset\partial M_i}\eta_i\gamma_j\},
\ (1-\theta)\gamma_i
\Bigr).
\]
 With this choice,
\[
0<\beta_i<\gamma_i(1-\theta)<1.
\]
Furthermore, if \(M_j\subset \partial M_i\), then
\[
\beta_i>\eta_i\gamma_j\ge \gamma_j
\]
because \(\eta_i\ge 1\). Thus the families \((\beta_i)\) and \((\gamma_i)\) satisfy
all the hypotheses of Proposition~\ref{proposition:neighborhoods of cells}.

We now apply Proposition~\ref{proposition:neighborhoods of cells} with these choices.
It yields constants
\[
 c_i,c_{1,i},c_{2,i},c_{3,i}>0,
\qquad
\omega_{1,i},\omega_{2,i},\omega_{3,i}\in[0,\beta_i/4],
\qquad
i\in\llbracket 1,T\rrbracket,
\]
such that, for every \(\alpha\in(0,1]\), if we define
\[
\mathcal N_0(i,\alpha)
:=
X\cap
B\!\left(
M_i\setminus \bigcup_{j:M_j\subset\partial M_i}B(M_j,c_j\alpha^{\gamma_j}/2),
\,2c_i\alpha^{\beta_i}
\right),
\]
then \(P_{M_i}\) is \(L\)-Lipschitz and \(C^2\) on
\(\bigcup_{\alpha\in(0,1]}\mathcal N_0(i,\alpha)\), for any $j\in\llbracket 1,T\rrbracket$ such that $\overline{M_i}\cap M_j=M_i\cap\overline{M_j}=\emptyset$, we have
\begin{equation}\label{eq:N_0_no_intersect}
    \bigcup _{\alpha\in(0,1]}\mathcal{N}_0(i,\alpha)\cap\bigcup_{\alpha\in(0,1]}\mathcal{N}_0(j,\alpha)=\emptyset,
\end{equation}
and
\begin{align}
\label{eq:cor-proof-L1}
|\nabla_{M_i}f(x)-\nabla_{M_i}f(x')|
&\le \frac{c_{1,i}}{\alpha^{\omega_{1,i}}}|x-x'|,
&&\forall x,x'\in \mathcal N_0(i,\alpha)\cap M_i,\\
\label{eq:cor-proof-L2}
|P_{T_{M_i}(y)}(v)-\nabla_{M_i}f(y)|
&\le \frac{c_{2,i}}{\alpha^{\omega_{2,i}}}|x-y|,
&&\forall x\in \mathcal N_0(i,\alpha),\ y\in \mathcal N_0(i,\alpha)\cap M_i,\ v\in\partial f(x),\\
\label{eq:cor-proof-L3}
|DP_{M_i}(x)-DP_{M_i}(x')|
&\le \frac{c_{3,i}}{\alpha^{\omega_{3,i}}}|x-x'|,
&&\forall x,x'\in \mathcal N_0(i,\alpha).
\end{align}

Define
\[
\omega_i:=\max\{\omega_{1,i},\omega_{2,i},\omega_{3,i}\}\in[0,\beta_i/4],
\qquad
A:=\max\Bigl\{L,\max_{i\in\llbracket 1,T\rrbracket}\max\{c_{1,i},c_{2,i},c_{3,i}\}\Bigr\}.
\]
Since \(\alpha\in(0,1]\), from \eqref{eq:cor-proof-L1}--\eqref{eq:cor-proof-L3} we obtain
\begin{align*}
|\nabla_{M_i}f(x)-\nabla_{M_i}f(x')|
&\le \frac{A}{\alpha^{\omega_i}}|x-x'|,
&&\forall x,x'\in \mathcal N_0(i,\alpha)\cap M_i,\\
|P_{T_{M_i}(y)}(v)-\nabla_{M_i}f(y)|
&\le \frac{A}{\alpha^{\omega_i}}|x-y|,
&&\forall x\in \mathcal N_0(i,\alpha),\ y\in \mathcal N_0(i,\alpha)\cap M_i,\ v\in\partial f(x),\\
|DP_{M_i}(x)-DP_{M_i}(x')|
&\le \frac{A}{\alpha^{\omega_i}}|x-x'|,
&&\forall x,x'\in \mathcal N_0(i,\alpha).
\end{align*}

Finally, set
\begin{equation}\label{eq:cor-beta}
\hat\beta:=\min_{i\in\llbracket 1,T\rrbracket}\min\{\beta_i-\omega_i,\ 1-\omega_i\}>0.
\end{equation}
This is well defined because \(\omega_i\le \beta_i/4<\beta_i<1\).

\medskip
\noindent\underline{Step 3: Choice of \(\hat{\alpha}\).}
Since there are only finitely many relevant pairs \((i,j)\) with \(M_j\subset \partial M_i\),
and since \(\beta_i>\gamma_j\) and \(\beta_i<1\), we may choose
\(\hat{\alpha}\in(0,1)\) so small that, for every \(\alpha\in(0,\hat{\alpha}]\),
\begin{align}
L\alpha &\le c_i\alpha^{\beta_i},
&&\forall i\in\llbracket 1,T\rrbracket,\label{eq:cor-small-1}\\
c_i\alpha^{\beta_i}+L^2\alpha &\le \frac{c_j}{3}\alpha^{\gamma_j},
&&\forall i,j\in\llbracket 1,T\rrbracket \text{ with } M_j\subset \partial M_i,\label{eq:cor-small-2}\\
Lc_i\alpha^{\beta_i} &\le \epsilon,
&&\forall i\in\llbracket 1,T\rrbracket.\label{eq:cor-small-3}
\end{align}

\medskip
\noindent\underline{Step 4: Nonintersecting neighborhoods.} For all $i\in \llbracket 1,T\rrbracket$ and $\alpha \in (0,1]$, let
\[\mathcal N(i,\alpha)
:=
(X\cap[|f|\le \epsilon])\cap
\left(
B(M_i,c_i\alpha^{\beta_i})
\setminus \bigcup_{j:M_j\subset\partial M_i}B(M_j,c_j\alpha^{\gamma_j})
\right).\]
For any $j\in\llbracket 1,T\rrbracket$ such that $\overline{M_i}\cap M_j=M_i\cap\overline{M_j}=\emptyset$, to prove
\[\bigcup_{\alpha\in(0,1]}\mathcal{N}(i,\alpha)\cap \bigcup_{\alpha\in(0,1]}\mathcal{N}(j,\alpha)=\emptyset,\]
it suffices to show that
\[\mathcal N(i,\alpha)\subset \mathcal N_0(i,\alpha),\]
thanks to \eqref{eq:N_0_no_intersect}.

To prove the above inclusion, consider any $i\in \llbracket 1,T\rrbracket,\alpha\in (0,1]$, and $x\in \mathcal N(i,\alpha)$. There exists $y\in M_i \setminus \bigcup_{j:M_j\subset\partial M_i}B(M_j,c_j\alpha^{\gamma_j})$ such that $|x - y| \le c_i\alpha^{\beta_i}$. Thus,
\[x\in X\cap B\left(M_i \setminus \bigcup_{j:M_j\subset\partial M_i}B(M_j,c_j\alpha^{\gamma_j}),c_i\alpha^{\beta_i}\right)\subset \mathcal N_0(i,\alpha).\]
This proves item 2 in the corollary.

\medskip
\noindent\underline{Step 5: Verification of the assumptions of Lemma~\ref{lemma:drop of fy}.}
Fix \(i\in\llbracket 1,T\rrbracket\), and let \((x_k)_{k\in\mathbb N}\) be a subgradient
sequence with \(\alpha_{\llbracket 0,K-1\rrbracket}\subset(0,\hat{\alpha}]\) and
\[
x_k\in \mathcal N(i,\alpha_k)
:=
(X\cap[|f|\le \epsilon])\cap
\left(
B(M_i,c_i\alpha_k^{\beta_i})
\setminus \bigcup_{j:M_j\subset\partial M_i}B(M_j,c_j\alpha_k^{\gamma_j})
\right)
\]
for all \(k\in\llbracket 0,K\rrbracket\).

We verify Lemma~\ref{lemma:drop of fy} with the following correspondence:
\[
M=M_i,\qquad U=U_i:=\bigcup_{\alpha\in(0,1]}\mathcal N_0(i,\alpha),\qquad
y_k=P_{M_i}(x_k),\qquad d_k=d(x_k,M_i),
\]
and
\[
L_{f,k}=L_{V,k}=L_{P,k}:=\frac{A}{\alpha_k^{\omega_i}}
\qquad\text{for }k\in\llbracket 0,K-1\rrbracket.
\]
We now check the required assumptions.

For each \(k\in\llbracket 0,K\rrbracket\), choose \(\xi_k\in M_i\) such that
\[
|x_k-\xi_k|\le c_i\alpha_k^{\beta_i}.
\]
If \(M_j\subset \partial M_i\), then
\[
d(\xi_k,M_j)\ge d(x_k,M_j)-|x_k-\xi_k|
\ge c_j\alpha_k^{\gamma_j}-c_i\alpha_k^{\beta_i}
\ge \frac{2c_j}{3}\alpha_k^{\gamma_j}
>\frac{c_j}{2}\alpha_k^{\gamma_j}
\]
by \eqref{eq:cor-small-2}. Hence
\[
\xi_k\in
M_i\setminus \bigcup_{j:M_j\subset\partial M_i}B(M_j,c_j\alpha_k^{\gamma_j}/2).
\]

We first show that
\begin{equation}\label{eq:cor-ball-xk}
B(x_k,L\alpha_k)\subset \mathcal N_0(i,\alpha_k).
\end{equation}
Indeed, for any \(u\in B(x_k,L\alpha_k)\),
\[
|u-\xi_k|
\le |u-x_k|+|x_k-\xi_k|
< L\alpha_k+c_i\alpha_k^{\beta_i}
\le 2c_i\alpha_k^{\beta_i}
\]
by \eqref{eq:cor-small-1}. Since \(\xi_k\) belongs to the inner set defining
\(\mathcal N_0(i,\alpha_k)\), this gives \(u\in \mathcal N_0(i,\alpha_k)\), proving
\eqref{eq:cor-ball-xk}. In particular, \(P_{M_i}\) is well defined at \(x_k\).

Set
\[
y_k^i:=P_{M_i}(x_k),
\qquad
d_k^i:=|x_k-y_k^i|=d(x_k,M_i).
\]
Since \(x_k\in B(M_i,c_i\alpha_k^{\beta_i})\), we have
\begin{equation}\label{eq:cor-dk}
d_k^i\le c_i\alpha_k^{\beta_i}.
\end{equation}

Next we prove that
\begin{equation}\label{eq:cor-ball-yk}
B(y_k^i,L^2\alpha_k)\cap M_i\subset \mathcal N_0(i,\alpha_k)\cap M_i.
\end{equation}
Fix \(u\in B(y_k^i,L^2\alpha_k)\cap M_i\), and let \(M_j\subset \partial M_i\). Then
\[
d(y_k^i,M_j)\ge d(x_k,M_j)-|x_k-y_k^i|
\ge c_j\alpha_k^{\gamma_j}-c_i\alpha_k^{\beta_i}.
\]
Hence
\[
d(u,M_j)\ge d(y_k^i,M_j)-|u-y_k^i|
\ge c_j\alpha_k^{\gamma_j}-c_i\alpha_k^{\beta_i}-L^2\alpha_k
\ge \frac{2c_j}{3}\alpha_k^{\gamma_j}
>\frac{c_j}{2}\alpha_k^{\gamma_j}
\]
by \eqref{eq:cor-small-2}. Since \(u\in M_i\), it follows that
\[
u\in
M_i\setminus \bigcup_{j:M_j\subset\partial M_i}B(M_j,c_j\alpha_k^{\gamma_j}/2)
\subset \mathcal N_0(i,\alpha_k)\cap M_i,
\]
which proves \eqref{eq:cor-ball-yk}.

By construction, \(P_{M_i}\) is single-valued, \(L\)-Lipschitz, and \(C^2\) on \(U_i\).
Moreover, \eqref{eq:cor-ball-xk} yields
\[
\bigcup_{k=0}^{K-1}B(x_k,L\alpha_k)\subset U_i.
\]
Also, \eqref{eq:cor-ball-yk} and the estimates obtained in Step~2 show that, for every
\(k\in\llbracket 0,K-1\rrbracket\),
\begin{enumerate}
    \item $|\nabla_M f(x)-\nabla_M f(x')|
    \le \frac{A}{\alpha_k^{\omega_i}} |x-x'|$, for all $x,x'\in B(y_k,\alpha_kL^2)\cap M_i$,
    \item $|P_{T_M(y_k)}(v_k)-\nabla_M f(y_k)|
    \le  \frac{A}{\alpha_k^{\omega_i}}|x_k-y_k|$, for all $v_k\in \partial f(x_k)$,

    \item $|DP_M(x)-DP_M(x')|
    \le  \frac{A}{\alpha_k^{\omega_i}}|x-x'|$ for all $x,x'\in B(x_k,\alpha_kL)\cup\{y_k\}$.
\end{enumerate}
Thus the Lipschitz-type requirements of Lemma~\ref{lemma:drop of fy} hold with
\[
L_{f,k}=L_{V,k}=L_{P,k}=\frac{A}{\alpha_k^{\omega_i}}.
\]
Since \(A\ge L\) and \(\alpha_k\le 1\), we also have
\(
L_{f,k},\,L_{V,k},\,L_{P,k}\ge L.\)

Finally, using \eqref{eq:cor-dk} and \eqref{eq:cor-small-3}, we obtain
\[
|f(y_k^i)|
\le |f(x_k)|+L|x_k-y_k^i|
\le \epsilon+Lc_i\alpha_k^{\beta_i}
\le 2\epsilon.
\]
Hence \(y_k^i\in M_i\cap[|f|\le 2\epsilon]\) for all \(k\in\llbracket 0,K\rrbracket\). Therefore,
by Step~1, the chosen desingularizing function \(\psi\) is valid at every \(y_k^i\), and all
assumptions of Lemma~\ref{lemma:drop of fy} are satisfied.

\medskip
\noindent\underline{Step 6: Application of Lemma~\ref{lemma:drop of fy}.}
Using \eqref{eq:cor-dk}, we have
\begin{align*}
\frac12\alpha_kL_{V,k}^2(d_k^i)^2
+L^2\alpha_kL_{P,k}d_k^i
+\frac{L^4\alpha_k^2(L_{f,k}+L_{P,k})}{2}
&\le
\frac{A^2c_i^2}{2}\alpha_k^{1+2\beta_i-2\omega_i}
+L^2Ac_i\,\alpha_k^{1+\beta_i-\omega_i}
+L^4A\,\alpha_k^{2-\omega_i}.
\end{align*}
By \eqref{eq:cor-beta},
\[
2\beta_i-2\omega_i\ge \beta_i-\omega_i\ge \hat\beta,
\qquad
1-\omega_i\ge \hat\beta.
\]
Hence, there exists \(C>0\) such that
\begin{equation}\label{eq:cor-g}
\frac12\alpha_kL_{V,k}^2(d_k^i)^2
+L^2\alpha_kL_{P,k}d_k^i
+\frac{L^4\alpha_k^2(L_{f,k}+L_{P,k})}{2}
\le C\,\alpha_k^{1+\hat\beta}.
\end{equation}
Therefore, any sequence \(g_0\ge g_1\ge \cdots \ge g_K\ge 0\) satisfying
\[
g_k-g_{k+1}\ge C\,\alpha_k^{1+\hat\beta}
\qquad\forall k\in\llbracket 0,K-1\rrbracket
\]
also satisfies the \(g\)-assumption in Lemma~\ref{lemma:drop of fy}.

Likewise,
\begin{align*}
L^3L_{f,k}\alpha_k^2
+L\alpha_kL_{V,k}d_k^i
+\frac{L\alpha_k}{\psi'(g_k)}
&\le
L^3A\,\alpha_k^{2-\omega_i}
+LAc_i\,\alpha_k^{1+\beta_i-\omega_i}
+L\mu\,\alpha_k g_k^\theta\\
&\le C\bigl(\alpha_k^{1+\hat\beta}+\alpha_k g_k^\theta\bigr),
\end{align*}
after enlarging \(C\) if necessary. 
Applying Lemma~\ref{lemma:drop of fy}, we obtain
\begin{align*}
\sum_{k=0}^{K-1}|y_{k+1}^i-y_k^i|
&\le
2L\bigl(\psi(z_0^i)-\psi(z_K^i)\bigr)
+
C\sum_{k=0}^{K-1}\bigl(\alpha_k^{1+\hat\beta}+\alpha_k g_k^\theta\bigr)
+
L^2\max_{k\in\llbracket 0,K-1\rrbracket}\alpha_k,
\end{align*}
where
\[
z_k^i:=f(y_k^i)+g_k,
\qquad
z_0^i\ge z_1^i\ge \cdots \ge z_K^i.
\]
Since
\[
\psi(z_0^i)-\psi(z_K^i)
=
\frac{
\operatorname{sgn}(z_0^i)|z_0^i|^{1-\theta}
-
\operatorname{sgn}(z_K^i)|z_K^i|^{1-\theta}
}{(1-\theta)\mu},
\]
we may enlarge \(C\) once more so that
\[
\frac{2L}{(1-\theta)\mu}\le C
\qquad\text{and}\qquad
L^2\le C.
\]
Then
\begin{align*}
\sum_{k=0}^{K-1}|y_{k+1}^i-y_k^i|
\le
C\Biggl(
&
\operatorname{sgn}(z_0^i)|z_0^i|^{1-\theta}
-
\operatorname{sgn}(z_K^i)|z_K^i|^{1-\theta}\\
&\quad
+\sum_{k=0}^{K-1}\bigl(\alpha_k^{1+\hat\beta}+\alpha_k g_k^\theta\bigr)
+
\max_{k\in\llbracket 0,K-1\rrbracket}\alpha_k
\Biggr).
\end{align*}
Rearranging yields
\[
\operatorname{sgn}(z_0^i)|z_0^i|^{1-\theta}
-
\operatorname{sgn}(z_K^i)|z_K^i|^{1-\theta}
\ge
\frac1C\sum_{k=0}^{K-1}|y_{k+1}^i-y_k^i|
-
\sum_{k=0}^{K-1}\bigl(\alpha_k^{1+\hat\beta}+\alpha_k g_k^\theta\bigr)
-
\max_{k\in\llbracket 0,K-1\rrbracket}\alpha_k.
\]
This is exactly the desired conclusion.
\end{proof}
\subsection{Relative length of subgradient sequences}\label{sec:RM_def}

To estimate the diameter of a subgradient sequence, we need to understand how the sequence evolves relative to a given stratification, especially when it moves near one stratum and then transitions to another. Corollary~\ref{cor:y_k^i} provides exactly the local stratified control needed for this purpose.

Accordingly, throughout this subsection we fix a locally Lipschitz definable function $f:\mathbb R^n\to \mathbb R$ and a compact definable set \(X\subset \mathbb R^n\). Corollary~\ref{cor:y_k^i} then yields constants
\[
\epsilon>0,\qquad C>0,\qquad \hat\beta>0,\qquad \hat\alpha\in (0,1),\qquad \theta\in (0,1),
\]
together with a stratification
\[
\mathcal M:=\{M_1,\ldots,M_{\overline T}\}
\]
of \(X\). We relabel the strata so that \(M_1,\ldots,M_T\) are the non-open strata (i.e., those of dimension strictly smaller than \(n\)), while \(M_{T+1},\ldots,M_{\overline T}\) are the open strata. Thus, \(\overline T\) denotes the total number of strata given by Corollary~\ref{cor:y_k^i}, whereas \(T\) denotes the number of non-open strata after this relabeling. We also fix a Lipschitz constant \(L>1\) of \(f\) on \(X\).

We first collect the notation and standing relations supplied by Corollary~\ref{cor:y_k^i}. We then introduce the key indices that record when the sequence approaches or leaves neighborhoods of non-open strata, and finally define a notion of relative length adapted to the stratification.

We begin by collecting the notation and relations that will be used throughout the rest of this section.

\begin{itemize}
    \item \textbf{Strata.}  
    The sets
    \[
    M_1,\ldots,M_T
    \quad\text{and}\quad
    M_{T+1},\ldots,M_{\overline T}
    \]
    denote respectively the non-open and open strata of the fixed stratification of \(X\).

    \item \textbf{Neighborhood constants.}  
    For each \(i\in \llbracket 1,\overline T\rrbracket\), Corollary~\ref{cor:y_k^i} provides constants
    \[
    c_i>0,
    \qquad
    \hat\beta<\beta_i<\gamma_i(1-\theta)<1.
    \]
    In addition, these constants satisfy the compatibility relation
    \[
    M_j\subset \partial M_i
    \quad\Longrightarrow\quad
    \beta_i>\gamma_j.
    \]
    Also,
    \begin{align}\label{eq: neighborhoods no intersection}
    \bigcup_{\alpha\in(0,1]}\mathcal{N}(i,\alpha)\cap \bigcup_{\alpha\in(0,1]}\mathcal{N}(j,\alpha)=\emptyset\quad \forall i,j\in\llbracket 1,T\rrbracket \text{ s.t. }\overline{M_i}\cap M_j=M_i\cap\overline{M_j}=\emptyset,\end{align}
    where for \(i\in \llbracket 1,\overline T\rrbracket\) and \(\alpha\in (0,\hat\alpha]\), we define
    \begin{equation}\label{eq:def_N_summary}
        \mathcal N(i,\alpha):=(X\cap [|f|\le \epsilon])\cap
        \Bigl(B(M_i,c_i\alpha^{\beta_i})\setminus \bigcup_{j:\,M_j\subset \partial M_i} B(M_j,c_j\alpha^{\gamma_j})\Bigr).
    \end{equation}

    \item \textbf{Derived exponents and uniform constants.}  
    For \(i\in \llbracket 1,T\rrbracket\), we define
    \[
    p_i:=\gamma_i(1-\theta),
    \qquad
    \underline\beta_i:=\min\{\beta_j:M_i\subset \partial M_j\}.
    \]
    We also set
    \[
    \underline\gamma:=\min_{i\in \llbracket 1,T\rrbracket}\gamma_i,
    \qquad
    \underline\beta:=\min_{i\in \llbracket 1,T\rrbracket}\beta_i,
    \qquad
    \underline p:=\min_{i\in \llbracket 1,T\rrbracket}p_i,
    \qquad
    \bar c:=\max_{i\in \llbracket 1,\overline T\rrbracket}c_i.
    \]
    From the above relations, we note in particular that
    \begin{equation}\label{eq:relation_beta_p}
    \beta_i<p_i<\gamma_i<1,
    \qquad
    M_j\subset \partial M_i \Longrightarrow \gamma_j<p_i,
    \qquad
    \underline\beta_i>p_i.      
    \end{equation}
    \item \textbf{The power function.} $\psi(t):=\mathrm{sgn}(t)|t|^{1-\theta}$.

\end{itemize}

All the definitions and statements below are uniform for subgradient sequences
\[
(x_k)_{k\in \mathbb N}
\quad\text{such that}\quad
x_{\llbracket 0,K\rrbracket}\subset X\cap [|f|\le \epsilon]
\]
for some \(K\in \mathbb N\), and whose step sizes satisfy
\[
0<\alpha_{K-1}\le \cdots \le \alpha_0\le \bar\alpha
\]
with $\bar\alpha \le \hat\alpha$ given by Corollary \ref{cor:y_k^i} and we assume that
\begin{equation}\label{eq:standing_alpha_small}
L\alpha\le c_i\alpha^{\gamma_i}/2\le c_i\alpha^{\beta_i}/6,
\qquad
\forall i\in \llbracket 1,\overline T\rrbracket,\ \forall \alpha\in (0,\bar{\alpha}].
\end{equation}
The above bounds on  \(\bar\alpha\) are sufficient for all constructions below in Section \ref{sec:RM_def}. Whenever a later result requires a further reduction of the step-size upper bound, we will state it explicitly. For a fixed subgradient sequence \((x_k)_{k\in\mathbb N}\), we shall also use the following auxiliary notation.

\begin{itemize}
        \item \textbf{Projected iterates and distances.}  
    Whenever \(x_k\in \mathcal N(i,\alpha_k)\), we write
    \[
    y_k^i:=P_{M_i}(x_k),
    \qquad
    d_k^i:=d(x_k,M_i).
    \]

    \item \textbf{Auxiliary scalar sequence and modified function values.}  
    We set
    \begin{equation}\label{eq:expr_g_E_z}
    g_k:=C\sum_{j=k}^{K-1}\alpha_j^{1+\hat\beta},\qquad E_k:=\alpha_k^{1+\hat\beta}+\alpha_k g_k^\theta,
    \qquad
    z_k^i:=f(y_k^i)+g_k,
    \qquad
    z_k:=f(x_k)+g_k.        
    \end{equation}
\end{itemize}

We now introduce the indices that mark when the sequence comes close to non-open strata and when it subsequently leaves the corresponding neighborhoods. These indices will later be used to decompose the trajectory into pieces on which the relative length is defined. Denote by $[x_{k},x_{k+1}]$ the line segment between $x_{k}$ and $x_{k+1}$. Let 
\begin{equation}\label{eq:IC}
    I_C:= \left\{k\in \llbracket 0,K-1\rrbracket: [x_{k},x_{k+1}]\bigcap\left(\bigcup_{i = 1}^T B(M_i,c_i\alpha_k^{\gamma_i})\right) \neq \emptyset\right\}
\end{equation} be the indices where the iterates are about to leave some open stratum. The following facts hold for the iterates $(x_k)_{k\in \mathbb N}$ depending on whether the indices belong to $I_C$ or not.
\begin{fact}\label{fact:IC_close}
For every \(k\in I_C\), the set
\[
\mathcal A_k
:=
\left\{
i\in \llbracket 1,T\rrbracket:
d(x_k,M_i)\le 2c_i\alpha_k^{\gamma_i}
\right\}
\]
is nonempty. Moreover, among the strata indexed by \(\mathcal A_k\), there is a unique one
of minimal dimension. Denoting its index by \(i_0\), we have $x_k\in \mathcal N(i_0,\alpha_k)$.
\end{fact}

\begin{proof}
Fix \(k\in I_C\). By the definition of \(I_C\) in \eqref{eq:IC}, there exist \(i\in\llbracket 1,T\rrbracket\) and
\(z\in [x_k,x_{k+1}]\) such that $d(z,M_i)\le c_i\alpha_k^{\gamma_i}$. Since \((x_k)_{k\in \mathbb N}\) is a subgradient sequence in $X$ where \(f\) is $L$-Lipschitz,
we have \(|x_{k+1}-x_k|\le L\alpha_k\). Hence
\[
d(x_k,M_i)
\le |x_k-z|+d(z,M_i)
\le |x_{k+1}-x_k|+c_i\alpha_k^{\gamma_i}
\le L\alpha_k+c_i\alpha_k^{\gamma_i}
\le 2c_i\alpha_k^{\gamma_i},
\]
where the last inequality follows from \eqref{eq:standing_alpha_small}. Thus
\(i\in\mathcal A_k\), and \(\mathcal A_k\neq\emptyset\).

Let \(i\in\mathcal A_k\) be such that
\[
\dim M_i=\min\{\dim M_h:h\in\mathcal A_k\}.
\]
We claim that \(x_k\in \mathcal N(i,\alpha_k)\). Indeed, since \(i\in\mathcal A_k\), we have $d(x_k,M_i)\le 2c_i\alpha_k^{\gamma_i}$. Moreover, if \(M_h\subset\partial M_i\), then \(\dim M_h<\dim M_i\). By the minimality
of \(\dim M_i\), such an index \(h\) cannot belong to \(\mathcal A_k\). Hence
\[
d(x_k,M_h)>2c_h\alpha_k^{\gamma_h}
\qquad
\text{for every }h\text{ with }M_h\subset\partial M_i.
\]
Therefore
\[
x_k\in
(X\cap[|f|\le \epsilon])
\cap B(M_i,2c_i\alpha_k^{\gamma_i})
\setminus
\bigcup_{h:M_h\subset\partial M_i}
B(M_h,2c_h\alpha_k^{\gamma_h}).
\]
By \eqref{eq:standing_alpha_small}, the right-hand side is contained in
\(\mathcal N(i,\alpha_k)\). Thus \(x_k\in \mathcal N(i,\alpha_k)\).

It remains to show such stratum with the lowest dimension is unique. Suppose, to the contrary, that two distinct indices \(i,j\in\mathcal A_k\) both have minimal
dimension among the indices in \(\mathcal A_k\).  Then \(\dim M_i=\dim M_j\).
By the frontier condition in Definition \ref{def:stratification}, this implies
\[
\overline{M_i}\cap M_j=M_i\cap \overline{M_j}=\emptyset .
\]
Moreover, we have
\[
x_k\in \mathcal N(i,\alpha_k)
\qquad\text{and}\qquad
x_k\in \mathcal N(j,\alpha_k).
\]
This contradicts the disjointness property in \eqref{eq: neighborhoods no intersection}.
\end{proof}
\begin{fact}\label{fact:IC_away}
    Let $k_1,k_2\in \mathbb N$ such that $0 \le k_1 < k_2 \le K$ and $[k_1,k_2)\cap I_C = \emptyset$. Then there exists $i\in \llbracket T+1,\overline{T}\rrbracket$ such that $x_k\in \mathcal N(i,\alpha_k)$ for all $k\in \llbracket k_1,k_2\rrbracket$.
\end{fact}
\begin{proof}
    As $k_1 \not\in I_C$, it holds that $x_{k_1}\in X\cap [ |f| \le \epsilon] \setminus \cup_{i = 1}^T B(M_i,c_i\alpha_{k_1}^{\gamma_i})$. By the definition of $\mathcal N(\cdot,\cdot)$ \eqref{eq:def_N_summary}, there exists $i\in \llbracket T+1,\overline{T}\rrbracket$ such that $x_{k_1}\in \mathcal N(i,\alpha_{k_1})$. We next show that
    \begin{equation}\label{eq:in_open}
        x_{k}\in \mathcal N(i,\alpha_k)\text{ and }k\not\in I_C\implies x_{k+1}\in \mathcal N(i,\alpha_{k+1}).
    \end{equation}
    The fact then follows immediately. To prove \eqref{eq:in_open}, note that by $k\not\in I_C$, we have $x_{k+1}\in X\cap [ |f| \le \epsilon] \setminus \cup_{i = 1}^T B(M_i,c_i\alpha_{k}^{\gamma_i})\subset X\cap [ |f| \le \epsilon] \setminus \cup_{i = 1}^T B(M_i,c_i\alpha_{k+1}^{\gamma_i})$ since $\alpha_{k+1} \le \alpha_k \le \hat{\alpha}\le 1$. Therefore, there exists $i'\in \llbracket T+1,\overline{T}\rrbracket$ such that $x_{k+1}\in M_{i'}$. Since
\(
x_{k+1}\in X\cap [ |f| \le \epsilon]\setminus \bigcup_{j=1}^{T} B(M_j,c_j\alpha_{k+1}^{\gamma_j}),
\)
it follows from the definition of \(\mathcal N(i',\alpha_{k+1})\) that
\(
x_{k+1}\in \mathcal N(i',\alpha_{k+1}).
\) Assume the contrary that $i' \ne i$, then $x_k$ and $x_{k+1}$ belong to $M_i$ and $M_{i'}$ respectively, which are two distinct open strata. Let $d:= x_{k+1} -x_k$, $t^* := \sup\{t\in [0,1]: x_k + \tau d\in M_i,\forall \tau \in [0,t)\}$, and $\bar{x}:= x_k + t^*d$. Then $\bar{x} \in \partial M_i\subset X$ by the definition of $t^*$ and compactness of $X$. Thus, $\bar{x}\in M_j$ for some $j\in \llbracket 1,T\rrbracket$, which is a contradiction with the fact that $k\not\in I_C$. This proves \eqref{eq:in_open}.
\end{proof}
Fact~\ref{fact:IC_close} ensures that each index in \(I_C\) can be associated
with at least one non-open stratum sufficiently close to \(x_k\), with a unique one of the lowest dimension.  Hence we may define $G:I_C\to \llbracket 1,T\rrbracket$
by
\begin{equation}\label{eq:def_G}
G(k):=
\arg\min\{\mathrm{dim}(M_i): d(x_k,M_i)
\le 2c_i\alpha_k^{\gamma_i},~i\in \llbracket 1,T\rrbracket\}.
\end{equation}
By Fact \ref{fact:IC_close}, this choice guarantees that $x_k\in \mathcal N(G(k),\alpha_k)$.

Assuming that \(I_C\ne \emptyset\), we now define a recursive family of indices by the selection $G$ that records how the sequence alternates between neighborhoods of non-open strata and the open strata connecting them.
\begin{equation}\label{eq:def_l}
    \begin{aligned}
            &l_0:= \min\{k: k\in I_C\},
            \\&s(l_m):=\max\{k\in \llbracket l_m,K\rrbracket: x_j\in \mathcal{N}(G(l_m),\alpha_j),~\forall j\in \llbracket l_m,k\rrbracket\},
            \\&q(l_m):=\max\{k\in \llbracket l_m,s(l_m)\rrbracket: d(x_k,M_{G(l_m)}) \le 2c_{G(l_m)}\alpha_k^{\gamma_{G(l_m)}}\},
            \\&l_{m+1}:=\inf\{k\in (q(l_m),\infty):k\in I_C\}       
    \end{aligned}
\end{equation}
for $m=0,\ldots \overline{m}-1$ where $\overline{m}:= \max\{m\in \mathbb{N}:l_m<+\infty\}$.  These indices isolate one segment of the sequence near a non-open stratum at a time. More precisely, \(l_m\) is the first index of such a segment, \(s(l_m)\) is its last index inside the neighborhood \(\mathcal N(G(l_m),\alpha_k)\), \(q(l_m)\) is the last index between \(l_m\) and \(s(l_m)\) that is still very close to the stratum \(M_{G(l_m)}\), and \(l_{m+1}\) is the next index at which the sequence again comes close to a non-open stratum. An illustration of these indices is given in Figure~\ref{fig:rm-illustration}.

Let
\[
\mathcal L:=\{l_0,\ldots,l_{\overline m}\},
\]
and regard \(s\) and \(q\) as mappings defined on \(\mathcal L\). The recursive construction immediately yields two facts that will be used repeatedly later.
 \begin{fact}\label{fact:q_open}
For each $m$ such that $l_{m+1}>q(l_m)+1$, there exists $i\in \llbracket T+1,\overline{T}\rrbracket$ such that $x_k\in \mathcal N(i,\alpha_k)$ for all $k\in \llbracket q(l_m)+1,l_{m+1}\rrbracket$.    
\end{fact}
\begin{proof}
    It follows directly from Fact \ref{fact:IC_away}.
\end{proof}
\begin{fact}\label{fact:no_repeat}
 Let $l\in \mathcal L$, then it holds that $G(k)\neq G(l)$ for all $k\in (q(l),s(l)]\cap I_C$. Thus, if $l,\ell \in \mathcal L$ satisfy $\ell>l$ and $G(l) = G(\ell)$, then $\ell>s(l)$.
 \end{fact}
\begin{proof}
The first part is immediate from the definition of $G(\cdot)$, $q(l)$, and $s(l)$ (see \eqref{eq:def_G} and \eqref{eq:def_l}). The second part follows from $\ell > q(l)$ by the definition of $\mathcal L$.
\end{proof}

We are now ready to define the length of the subgradient sequence relative to a fixed stratification. Roughly speaking, this quantity agrees with the usual length when the sequence stays inside a single open stratum, but near a non-open stratum it measures the motion through projected iterates until the trajectory leaves the very-close regime at index \(q(l_m)\), after which it follows the original iterates again.
\begin{definition}[Relative length]\label{def:RL}
Let \(0\le k_1<k_2\le K\). We define the relative length of
\(x_{\llbracket k_1,k_2\rrbracket}\), denoted by
\(\mathrm{RL}(x_{\llbracket k_1,k_2\rrbracket})\), in the following cases.

\smallskip
\noindent\emph{Case 1.}
If $k_1<k_2 \le l_0$, then
\[\mathrm{RL}(x_{\llbracket k_1,k_2\rrbracket})
:=
\sum_{k=k_1}^{k_2-1}|x_{k+1}-x_k|.\]

\smallskip
\noindent\emph{Case 2.}
Assume that there exists \(m\in \llbracket 0,\overline m\rrbracket\) such that
\[
l_m\le k_1<k_2\le \inf\{\ell\in \mathcal L: \ell >l_m\},
\]
and write \(i:=G(l_m)\) and \(q:=q(l_m)\).
\begin{enumerate}
    \item If \(q\le k_1<k_2\), set
\[
\mathrm{RL}(x_{\llbracket k_1,k_2\rrbracket})
:=
\sum_{k=k_1}^{k_2-1}|x_{k+1}-x_k|.
\]
\item If \(l_m<k_1<q\), set
\[
\mathrm{RL}(x_{\llbracket k_1,k_2\rrbracket})
:=
\begin{cases}
\displaystyle
\sum_{k=k_1}^{k_2-1}|y_{k+1}^{i}-y_k^{i}|,
& \text{if } k_2<q,\\[1.2em]
\displaystyle
\sum_{k=k_1}^{q-1}|y_{k+1}^{i}-y_k^{i}|
+|x_q-y_q^{i}|
+\sum_{k=q}^{k_2-1}|x_{k+1}-x_k|,
& \text{if } k_1<q\le k_2.
\end{cases}
\]
\item If \(k_1=l_m<q\), set
\[
\mathrm{RL}(x_{\llbracket l_m,k_2\rrbracket})
:=
|x_{l_m}-y_{l_m}^{i}|
+|y_{l_m+1}^{i}-y_{l_m}^{i}|
+\mathrm{RL}(x_{\llbracket l_m+1,k_2\rrbracket}),
\]
with the convention that
\[
\mathrm{RL}(x_{\llbracket k_2,k_2\rrbracket})=0.
\]
\end{enumerate}

\smallskip
\noindent\emph{Case 3.}
Assume that
\(
(k_1,k_2)\cap \mathcal L\neq \emptyset.
\)
Let
\[
m_1:=\min\{m\in \llbracket 0,\overline m\rrbracket:l_m>k_1\},
\qquad
m_2:=\max\{m\in \llbracket 0,\overline m\rrbracket:l_m<k_2\}.
\]
Then
\begin{equation*}
\mathrm{RL}(x_{\llbracket k_1,k_2\rrbracket})
:=
\mathrm{RL}(x_{\llbracket k_1,l_{m_1}\rrbracket})
+\sum_{m=m_1}^{m_2-1}\mathrm{RL}(x_{\llbracket l_m,l_{m+1}\rrbracket})
+\mathrm{RL}(x_{\llbracket l_{m_2},k_2\rrbracket}).
\end{equation*}
\end{definition}
Figure~\ref{fig:rm-illustration} illustrates the indices \(l_m\), \(q(l_m)\), \(s(l_m)\), as well as the corresponding relative length \(\mathrm{RL}\). The definition above is designed so that \(\mathrm{RL}\) is available on every interval and behaves like an effective length adapted to the stratification. We record these basic properties in the next remark.

\usetikzlibrary{calc,decorations.pathreplacing}

\definecolor{MyBlue}{RGB}{65,105,225}
\definecolor{MyMagenta}{RGB}{204,51,153}

\definecolor{AreaOne}{RGB}{255,200,200}
\definecolor{AreaTwo}{RGB}{200,255,200}
\definecolor{AreaThree}{RGB}{200,200,255}

\tikzset{
  levelset/.style={draw=cyan, dashed, line cap=round},
  sublevel/.style={draw=orange!80!black, dash pattern=on 6pt off 4pt},
  valley/.style={draw=MyBlue, very thick},
  iterline/.style={draw=MyMagenta, line width=1.2pt},
  iterlineproj/.style={draw=black, line width=1.2pt},
  iterdot/.style={fill=MyMagenta, draw=MyMagenta},
  projdot/.style={fill=black, draw=black},
  switchdot/.style={draw=MyMagenta, fill=white, very thick},
  lab/.style={font=\scriptsize}
}

\newcommand{\DrawCommonStratification}{

  \draw[levelset] (-1.6,2.4) .. controls (-1.2,1.9) and (-0.8,1.2) .. (-0.03,0.48);
  \draw[sublevel] (-1.8,2.1) .. controls (-1.4,1.6) and (-1.0,0.9) .. (-0.28,0.389);
  \draw[sublevel] (-2.15,1.8) .. controls (-1.75,1.3) and (-1.35,0.6) .. (-0.44,0.00);
  \draw[levelset] (-2.4,1.7) .. controls (-1.85,1.0) and (-1.7,0.6) .. (-0.45,-0.24);

  \draw[levelset] (0,0) circle (1.32);
  \draw[sublevel] (0,0) circle (0.9);

  \draw[levelset] (.25,.27) .. controls (1.2,0.3) and (2.4,-.3) .. (2.9,.25);
  \draw[levelset] (.45,-.2) .. controls (1.2,-0.15) and (2.4,-.9) .. (2.9,-.3);
  \draw[sublevel] (.41,.18) .. controls (1.2,0.15) and (2.4,-.42) .. (2.9,.08);
  \draw[sublevel] (.41,-.05) .. controls (0.8,0.1) and (2.2,-0.65) .. (2.9,-.20);

  \begin{scope}
    \fill[AreaOne, opacity=0.4]
      (-0.03,0.48) .. controls (-0.8,1.2) and (-1.2,1.9) .. (-1.6,2.4)
      --
      (-2.4,1.7) .. controls (-1.85,0.9) and (-1.7,0.6) .. (-0.40,-0.26)
      -- cycle;

    \fill[AreaTwo, opacity=0.4]
      (.25,.27) .. controls (1.2,0.3) and (2.4,-.3) .. (2.9,.25)
      --
      (2.9,-.3) .. controls (2.4,-.9) and (1.2,-0.15) .. (.45,-.2)
      -- cycle;

     \fill[white] (0,0) circle (0.45);
    \fill[AreaThree, opacity=0.4] (0,0) circle (1.32);
  \end{scope}

  \draw[valley]
    (-2,2) .. controls (-1.1,0.6) and (-0.5,0.2) .. (0,0)
           .. controls (1.2,0.3) and (2.2,-0.7) .. (3,0);
  \filldraw[MyBlue] (0,0) circle (1pt);
}

\newcommand{\DefineMainIterates}{
  \coordinate (p0)  at (-1.00, 2.50);
  \coordinate (p1)  at (-1.20, 2.10);
  \coordinate (p2)  at (-1.40, 1.80);
  \coordinate (p3)  at (-1.50, 1.50);
  \coordinate (p7)  at (-1.15, 1.15);
  \coordinate (p8)  at (-1.00, 1.30);
  \coordinate (p9)  at (-0.85, 1.15);
  \coordinate (p10) at (-0.60, 1.00);
  \coordinate (p11) at (-0.35, 1.15);
  \coordinate (p12) at (-0.10, 1.10);
  \coordinate (p13) at ( 0.10, 0.70);
  \coordinate (p14) at (-0.10, 0.25);
  \coordinate (p18) at ( 0.30,-0.30);
  \coordinate (p19) at ( 0.65,-0.35);
  \coordinate (p20) at ( 0.95,-0.40);
  \coordinate (p21) at ( 1.15,-0.45);
  \coordinate (p22) at ( 1.40,-0.60);
  \coordinate (p23) at ( 1.60,-0.80);
}

\newcommand{\DrawSwitchingLabels}{
  \node[lab,MyMagenta, anchor=west] at ($(p0)+(0.02,-0.10)$) {$0$};
  \node[lab,MyMagenta, anchor=west] at ($(p23)+(0.02,-0.10)$) {$K$};

  \fill[switchdot] (p3) circle (1.5pt);
  \node[lab,MyMagenta, anchor=west] at ($(p3)+(0.02,-0.10)$) {$l_0$};

  \fill[switchdot] (p13) circle (1.5pt);
  \node[lab,MyMagenta, anchor=west] at ($(p13)+(-0.30,0.00)$) {$l_1$};

  \fill[switchdot] (p10) circle (1.5pt);
  \node[lab,MyMagenta, anchor=west] at ($(p10)+(-0.27,-0.16)$) {$s(l_0)$};

  \fill[switchdot] (p21) circle (1.5pt);
  \node[lab,MyMagenta, anchor=west] at ($(p21)+(-0.19,-0.22)$) {$s(l_1)$};

  \fill[switchdot] (p7) circle (1.5pt);
  \node[lab,MyMagenta, anchor=west] at ($(p7)+(-0.12,-0.16)$) {$q(l_0)$};

  \fill[switchdot] (p19) circle (1.5pt);
  \node[lab,MyMagenta, anchor=west] at ($(p19)+(-0.28,-0.18)$) {$q(l_1)$};
}

\newcommand{\DrawNeighborhoodLegendHorizontal}{
\begin{tikzpicture}[x=1cm,y=1cm,baseline]
  \def\LegendStep{2.45}
  \def\BoxW{0.35}
  \def\BoxH{0.22}
  \def\TextGap{0.13}
  \def\LineW{0.50}

  \fill[AreaOne] (0*\LegendStep,0) rectangle ++(\BoxW,\BoxH);
  \node[font=\small, anchor=west] at ({0*\LegendStep+\BoxW+\TextGap},0.11)
    {$\mathcal N(1,\alpha)$};

  \fill[AreaThree] (1*\LegendStep,0) rectangle ++(\BoxW,\BoxH);
  \node[font=\small, anchor=west] at ({1*\LegendStep+\BoxW+\TextGap},0.11)
    {$\mathcal N(2,\alpha)$};

  \fill[AreaTwo] (2*\LegendStep,0) rectangle ++(\BoxW,\BoxH);
  \node[font=\small, anchor=west] at ({2*\LegendStep+\BoxW+\TextGap},0.11)
    {$\mathcal N(3,\alpha)$};

  \draw[levelset] ({3*\LegendStep},0.11) -- ++(\LineW,0);
  \node[font=\small, anchor=west] at ({3*\LegendStep+\LineW+\TextGap},0.11)
    {$c_i\alpha^{\beta_i}$};

  \draw[sublevel] ({4*\LegendStep},0.11) -- ++(\LineW,0);
  \node[font=\small, anchor=west] at ({4*\LegendStep+\LineW+\TextGap},0.11)
    {$2c_i\alpha^{\gamma_i}$};
\end{tikzpicture}
}

\begin{figure}[!ht]
\centering
\DrawNeighborhoodLegendHorizontal \par\vspace{0.8em}
\begin{tikzpicture}[scale=2]
  \DrawCommonStratification
  \DefineMainIterates

  \coordinate (p4)  at (-1.65,1.30);
  \coordinate (p5)  at (-1.35,1.20);
  \coordinate (p15) at (-0.05,-0.05);
  \coordinate (p16) at ( 0.15, 0.15);
  \coordinate (p17) at ( 0.25,-0.05);

  \draw[iterline] (p0)--(p1)--(p2)--(p3);
  \draw[iterline] (p3)--(p4)--(p5)--(p7);
  \draw[iterline] (p7)--(p8)--(p9)--(p10)--(p11)--(p12)--(p13)--(p14);
  \draw[iterline] (p14)--(p15)--(p16)--(p17)--(p18);
  \draw[iterline] (p18)--(p19)--(p20)--(p21)--(p22)--(p23);

  \foreach \pt in {p0,p1,p2,p3,p4,p5,p7,p8,p9,p10,p11,p12,p13,p14,p15,p16,p17,p18,p19,p20,p21,p22,p23}
    \fill[iterdot] (\pt) circle (1pt);

  \DrawSwitchingLabels
  \node[align=left, font=\small, anchor=west] at (-2.4,2.2) {\textcolor{blue}{$M_1$}};
\node[align=left, font=\small, anchor=west] at (-.3,-0.2) {\textcolor{blue}{$M_2$}};
\node[align=left, font=\small, anchor=west] at (3.0,-0.15) {\textcolor{blue}{$M_3$}};
\end{tikzpicture}

\par\vspace{0.8em}

\begin{tikzpicture}[scale=2]
  \DrawCommonStratification
  \DefineMainIterates

  \coordinate (p3j)  at (-1.60,1.43);
  \coordinate (p4j)  at (-1.55,1.33);
  \coordinate (p5j)  at (-1.40,1.14);
  \coordinate (p6j)  at (-1.29,1.02);
  \coordinate (p14j) at ( 0.00,0.00);

  \draw[iterlineproj] (p0)--(p1)--(p2)--(p3);
  \draw[iterlineproj] (p3)--(p3j)--(p4j)--(p5j)--(p6j)--(p7);
  \draw[iterlineproj] (p7)--(p8)--(p9)--(p10)--(p11)--(p12)--(p13);
  \draw[iterlineproj] (p13)--(p14j)--(p19);
  \draw[iterlineproj] (p19)--(p20)--(p21)--(p22)--(p23);

  \foreach \pt in {p0,p1,p2,p3,p7,p8,p9,p10,p11,p12,p13,p19,p20,p21,p22,p23}
    \fill[iterdot] (\pt) circle (1pt);

  \foreach \pt in {p3j,p4j,p5j,p6j,p14j}
    \fill[projdot] (\pt) circle (1pt);

  \DrawSwitchingLabels
    \node[align=left, font=\small, anchor=west,color = blue] at (-2.4,2.2) {\textcolor{blue}{$M_1$}};
\node[align=left, font=\small, anchor=west] at (-.3,-0.2) {\textcolor{blue}{$M_2$}};
\node[align=left, font=\small, anchor=west] at (3.0,-0.15) {\textcolor{blue}{$M_3$}};
\end{tikzpicture}

\caption{Illustration of the indices \(l_m\), \(q(l_m)\), \(s(l_m)\), and the relative length \(\mathrm{RL}\) of a subgradient sequence. The top panel displays the original sequence in magenta, whereas the bottom panel emphasizes the corresponding relative length in black. For ease of visualization, the figure is drawn for the case of a constant step size ($\alpha_k = \alpha$). When the step sizes vary, the neighborhoods \(\mathcal N(i,\alpha_k)\) evolve with the iteration index \(k\), making a  static illustration impossible.}
\label{fig:rm-illustration}
\end{figure}
\begin{remark}\label{rem:RL_basic}
The quantity \(\mathrm{RL}(x_{\llbracket k_1,k_2\rrbracket})\) should be viewed as the length of the segment \(x_{\llbracket k_1,k_2\rrbracket}\) relative to the fixed stratification \(\mathcal M\). In particular, it measures the motion of the sequence partly through projected iterates on suitable strata, rather than solely through the original iterates. 

It is also worth noticing that \(\mathrm{RL}\) is monotone with respect to interval inclusion. Namely, if
    \(
    \llbracket k_1,k_2\rrbracket\subset \llbracket k_1',k_2'\rrbracket,
    \)
    then
    \[
    \mathrm{RL}(x_{\llbracket k_1,k_2\rrbracket})
    \le
    \mathrm{RL}(x_{\llbracket k_1',k_2'\rrbracket}).
    \]
    This follows directly from the definition, since enlarging the interval only adds nonnegative contributions to the relative length.

\end{remark}

In the remainder of this subsection, we relate the relative length to the displacement of the original iterates. We begin with a single block \(\llbracket l_m,l_{m+1}\rrbracket\), where the relative length controls the displacement from the left endpoint \(l_m\) to any later index in the block, up to a projection error before \(q(l_m)\).
\begin{lemma}\label{lemma:RM_within_one}
For any $m\in \llbracket 0,\overline{m}\rrbracket$, it holds that
\begin{enumerate}
\item If $k_2 \in \llbracket l_m,q(l_{m}) - 1\rrbracket$, then $|x_{l_m} - x_{k_2}| \le \mathrm{RL}(x_{\llbracket l_m,k_2\rrbracket})+c_{G(l_m)}\alpha_{k_2}^{\beta_{G(l_m)}}$;
    \item If $k_2\in \llbracket q(l_m),l_{m+1}\rrbracket$, then $|x_{l_m} - x_{k_2}| \le \mathrm{RL}(x_{\llbracket l_m,k_2\rrbracket})$.
\end{enumerate} 
\end{lemma}
\begin{proof}
    Given $k_2 \in \llbracket l_m,q(l_{m}) - 1\rrbracket$, by the definition of $\mathrm{RL}$ above and the triangular inequality, we have
    \begin{align*}
        \mathrm{RL}(x_{\llbracket l_m,k_2\rrbracket}) &= |x_{l_m} - y_{l_m}^{G(l_m)}|+\sum_{k = l_m}^{k_2 - 1} \left|y_{k+1}^{G(l_m)} - y_{k}^{G(l_m)}\right|\\
        &\ge |x_{l_m} - y_{k_2}^{G(l_m)}|\\
        &\ge |x_{l_m} - x_{k_2}| - |x_{k_2} - y_{k_2}^{G(l_m)}|\\
        &\ge |x_{l_m} - x_{k_2}| - c_{G(l_m)}\alpha_{k_2}^{\beta_{G(l_m)}},
    \end{align*}
where the last inequality is due to $k_2 \le q(l_m) \le s(l_m)$ and the definition of $s(l_m)$ in \eqref{eq:def_l}.

If $k_2\in \llbracket q(l_m),l_{m+1}\rrbracket$, then
\begin{align*}
    \mathrm{RL}(x_{\llbracket l_m,k_2\rrbracket}) &=|x_{l_m} - y_{l_m}^{G(l_m)}|+\sum_{k = l_m}^{q(l_m) - 1}\left|y_{k+1}^{G(l_m)} - y_{k}^{G(l_m)}\right| + \left|x_{q(l_m)} - y_{q(l_m)}^{G(l_m)}\right| +  \displaystyle\sum_{k = q(l_m)}^{k_2 - 1} |x_{k+1} - x_k|\\
    &\ge |x_{l_m} - x_{k_2}|
\end{align*}
by the triangular inequality.
\end{proof}

Lemma~\ref{lemma:RM_within_one} treats intervals contained in a single block. The next lemma extends this control to segments that may cross several consecutive blocks, as long as the terminal index \(k_2\) still lies before the exit time \(s(l)\) of the original block.
\begin{lemma}\label{lemma:RM_multi}
    Let $l\in \mathcal L$, $k_2 \in \llbracket q(l), s(l)\rrbracket$, and $m_2:=\max\{m\in \llbracket 0,\overline{m}\rrbracket:l_m<k_2\}$. Then either $M_{G(l_{m_2})}\subset \partial M_{G(l)}$ or $|x_{l} - x_{k_2}| \le \mathrm{RL}(x_{\llbracket l,k_2\rrbracket})+\bar{c} \alpha_{l}^{\underline\beta_{G(l)}}$.
\end{lemma}
\begin{proof}
Let $m_1\in \llbracket 0,\overline{m}\rrbracket$ be such that $l=l_{m_1}$. If $k_2\in \llbracket q(l),l_{m_1+1}\rrbracket$, then
\[
|x_l-x_{k_2}|\le \mathrm{RL}(x_{\llbracket l,k_2\rrbracket})
\]
by Lemma~\ref{lemma:RM_within_one}, and there is nothing to prove.

Assume now that $k_2>l_{m_1+1}$.
Then $m_2\ge m_1+1$ and $k_2\in \llbracket l_{m_2},l_{m_2+1}\rrbracket$. By Case 3 of Definition \ref{def:RL} and Lemma~\ref{lemma:RM_within_one}, we have
\begin{align*}
\mathrm{RL}(x_{\llbracket l,k_2\rrbracket})
&=\sum_{m=m_1}^{m_2-1}\mathrm{RL}(x_{\llbracket l_m,l_{m+1}\rrbracket})
   +\mathrm{RL}(x_{\llbracket l_{m_2},k_2\rrbracket})\\
&\ge \sum_{m=m_1}^{m_2-1}|x_{l_m}-x_{l_{m+1}}|
   +|x_{l_{m_2}}-x_{k_2}|
   -c_{G(l_{m_2})}\alpha_{k_2}^{\beta_{G(l_{m_2})}}\\
&\ge |x_l-x_{k_2}|-c_{G(l_{m_2})}\alpha_{k_2}^{\beta_{G(l_{m_2})}},
\end{align*}
where the last inequality follows from the triangular inequality.

It remains to estimate the last error term. Since $l_{m_2}<k_2\le s(l)$ and $l_{m_2}>q(l)$ by the definition of indices in \eqref{eq:def_l}, we have
\[
l_{m_2}\in (q(l),s(l))\cap I_C.
\]
Hence, by Fact \ref{fact:no_repeat}, one has
\(
G(l_{m_2})\neq G(l)
\).
On the other hand, because $l_{m_2}\le s(l)$, we have $x_{l_{m_2}}\in \mathcal N(G(l),\alpha_{l_{m_2}})$ by the definition of $s(l)$. Also, by Fact \ref{fact:IC_close}, we have $x_{l_{m_2}}\in \mathcal N(G(l_{m_2}),\alpha_{l_{m_2}})$.

Invoking the nonintersecting condition \eqref{eq: neighborhoods no intersection}, we have either $\overline{M_i}\cap M_j \ne \emptyset$ or $M_i\cap \overline{M_j}\ne \emptyset$. By the frontier condition of the stratification, this implies that either
\[
M_{G(l_{m_2})}\subset \partial M_{G(l)}
\qquad\text{or}\qquad
M_{G(l)}\subset \partial M_{G(l_{m_2})}.
\]
Now assume that $M_{G(l)}\subset \partial M_{G(l_{m_2})}$. 
By the definition of $\underline\beta_{G(l)}$, this implies
\(
\beta_{G(l_{m_2})}\ge \underline\beta_{G(l)}
\).
Since $(\alpha_k)_{k\in \mathbb N}$ is decreasing and $k_2\ge l$, we also have $\alpha_{k_2}\le \alpha_l$. Hence
\[
c_{G(l_{m_2})}\alpha_{k_2}^{\beta_{G(l_{m_2})}}
\le \bar c\,\alpha_{k_2}^{\underline\beta_{G(l)}}
\le \bar c\,\alpha_l^{\underline\beta_{G(l)}},
\]
as $c\le \bar{c}$. Substituting this into the previous estimate yields
\[
|x_l-x_{k_2}|
\le \mathrm{RL}(x_{\llbracket l,k_2\rrbracket})+\bar c\,\alpha_l^{\underline\beta_{G(l)}},
\]
which proves the claim.
\end{proof}

We now pass from estimates based at an index \(l\in\mathcal L\) to estimates on arbitrary intervals. The next lemma shows that the diameter of the original iterates over any interval can be bounded directly in terms of the corresponding relative length. The price for this increased generality is a larger error term of order \(\alpha_{k_1}^{\underline\beta}\), compared with the preceding localized bounds. This estimate will nonetheless be essential later, when we convert the relative-length bounds from Lemma~\ref{lemma:RM_arbitrary_interval} into diameter bounds for the original sequence and thereby prove Theorem~\ref{thm:diameter}.
 \begin{lemma}\label{lemma:diameter_by_RM}
For any $0\le k_1<k_2\le K$, one has
\begin{equation}\label{eq:endpoint_by_RM}
    |x_{k_1}-x_{k_2}|
    \le
    \mathrm{RL}(x_{\llbracket k_1,k_2\rrbracket})
    +4\bar c\,\alpha_{k_1}^{\underline{\beta}}.
\end{equation}
Consequently,
\begin{equation}\label{eq:diameter_by_RM}
    \mathrm{diam}(x_{\llbracket k_1,k_2\rrbracket})
    \le
    \mathrm{RL}(x_{\llbracket k_1,k_2\rrbracket})
    +4\bar c\,\alpha_{k_1}^{\underline{\beta}}.
\end{equation}

\end{lemma}

\begin{proof}
In the case where $[k_1,k_2)\cap I_C = \emptyset$, \eqref{eq:endpoint_by_RM} holds trivially by the triangular inequality. 

We next note a local estimate within a single block. Fix $m\in \llbracket 0,\overline m-1\rrbracket$ and let
$i:=G(l_m)$. Then for any $a,b\in \llbracket l_m,l_{m+1}\rrbracket$ with $a<b$, one has
\begin{equation}\label{eq:local_RM_bound}
    |x_a-x_b|
    \le
    \mathrm{RL}(x_{\llbracket a,b\rrbracket})
    +2c_i\alpha_a^{\beta_i}
    \le
    \mathrm{RL}(x_{\llbracket a,b\rrbracket})
    +2\bar c\,\alpha_a^{\underline\beta}.
\end{equation}
Indeed, this follows by the same triangular-inequality argument as in
Lemma~\ref{lemma:RM_within_one}, distinguishing the three cases:
\[
b<q(l_m),\qquad a<q(l_m)\le b,\qquad q(l_m)\le a.
\]
In the first case, both endpoints are compared with their projections onto
$M_{G(l_m)}$, which gives two projection errors; in the second case, only the left
endpoint contributes such an error; in the third case, no projection error appears.

Now let $0\le k_1<k_2\le K$ with $(k_1,k_2)\cap \mathcal L\neq \emptyset$, and define
\[
m_1:=\min\{m\in \llbracket 0,\overline m\rrbracket:l_m>k_1\},
\qquad
m_2:=\max\{m\in \llbracket 0,\overline m\rrbracket:l_m<k_2\}.
\]
By the triangular inequality,
\begin{align*}
|x_{k_1}-x_{k_2}|
\le\;&
|x_{k_1}-x_{l_{m_1}}|
+\sum_{m=m_1}^{m_2-1}|x_{l_m}-x_{l_{m+1}}|
+|x_{l_{m_2}}-x_{k_2}|.
\end{align*}
For the two partial pieces, \eqref{eq:local_RM_bound} yields
\begin{align*}
|x_{k_1}-x_{l_{m_1}}|
&\le \mathrm{RL}(x_{\llbracket k_1,l_{m_1}\rrbracket})
   +2\bar c\,\alpha_{k_1}^{\underline\beta},\\
|x_{l_{m_2}}-x_{k_2}|
&\le \mathrm{RL}(x_{\llbracket l_{m_2},k_2\rrbracket})
   +2\bar c\,\alpha_{l_{m_2}}^{\underline\beta}
 \le \mathrm{RL}(x_{\llbracket l_{m_2},k_2\rrbracket})
   +2\bar c\,\alpha_{k_1}^{\underline\beta},
\end{align*}
where we used the fact that $(\alpha_k)$ is nonincreasing. For each complete block,
Lemma~\ref{lemma:RM_within_one} gives
\[
|x_{l_m}-x_{l_{m+1}}|
\le
\mathrm{RL}(x_{\llbracket l_m,l_{m+1}\rrbracket}),
\qquad m=m_1,\dots,m_2-1.
\]
Combining these estimates and using Case 3 of Definition \ref{def:RL}, we obtain
\[
|x_{k_1}-x_{k_2}|
\le
\mathrm{RL}(x_{\llbracket k_1,k_2\rrbracket})
+4\bar c\,\alpha_{k_1}^{\underline\beta},
\]
which proves \eqref{eq:endpoint_by_RM}.

It remains to prove \eqref{eq:diameter_by_RM}. Let $p,q\in \llbracket k_1,k_2\rrbracket$
with $p<q$. Applying \eqref{eq:endpoint_by_RM} to the interval $\llbracket p,q\rrbracket$
gives
\[
|x_p-x_q|
\le
\mathrm{RL}(x_{\llbracket p,q\rrbracket})
+4\bar c\,\alpha_p^{\underline\beta}.
\]
Since $\mathrm{RL}$ is nondecreasing under interval inclusion by Remark \ref{rem:RL_basic} and
$\alpha_p\le \alpha_{k_1}$, it follows that
\[
|x_p-x_q|
\le
\mathrm{RL}(x_{\llbracket k_1,k_2\rrbracket})
+4\bar c\,\alpha_{k_1}^{\underline\beta}.
\]
Taking the supremum over all $p,q\in \llbracket k_1,k_2\rrbracket$ proves
\eqref{eq:diameter_by_RM}.
\end{proof}

\subsection{Bounding the relative length}\label{sec:bound_rm}
The goal of this subsection is to derive quantitative upper bounds on the relative length \(\mathrm{RL}\) introduced in Section~\ref{sec:RM_def}. The strategy is to compare relative length with the drop of \(\psi(z_k^i)\) using Corollary \ref{cor:y_k^i}, and then to control the error terms produced when the sequence passes from one stratum to another. We will use repeatedly the following simple fact about power functions.

\begin{fact}\label{fact:psi_holder}
Let
\[
\psi(t):=\operatorname{sgn}(t)|t|^{1-\theta},
\qquad \theta\in(0,1).
\]
Then, for all \(t_1,t_2\in \mathbb R\),
\[
|\psi(t_1)-\psi(t_2)|
\le 2\psi(|t_1-t_2|).
\]
\end{fact}

\begin{proof}
Let \(\alpha:=1-\theta\in(0,1)\). We distinguish two cases.

\medskip
\noindent\emph{Case 1: \(t_1t_2<0\).}
Then \(t_1\) and \(t_2\) have opposite signs, so
\[
|\psi(t_1)-\psi(t_2)|
=
|t_1|^\alpha+|t_2|^\alpha \le 2(|t_1|+|t_2|)^\alpha = 2|t_1-t_2|^\alpha = 2\psi(|t_1-t_2|).
\]
\medskip
\noindent\emph{Case 2: \(t_1t_2\ge 0\).}
Then \(t_1\) and \(t_2\) have the same sign, so
\[
|\psi(t_1)-\psi(t_2)|
=
\bigl||t_1|^\alpha-|t_2|^\alpha\bigr|.
\]
Without loss of generality, assume \(|t_1|\ge |t_2|\). Since \(\alpha\in(0,1)\), the map
\(s\mapsto s^\alpha\) is subadditive on \(\mathbb R_+\), i.e.,
\[
(a+b)^\alpha\le a^\alpha+b^\alpha
\qquad \forall a,b\ge 0.
\]
Applying this with \(a=|t_1|-|t_2|\) and \(b=|t_2|\), we get
\[
|t_1|^\alpha
=
\bigl((|t_1|-|t_2|)+|t_2|\bigr)^\alpha
\le (|t_1|-|t_2|)^\alpha+|t_2|^\alpha 
\]
Therefore,
\[
|\psi(t_1)-\psi(t_2)| = |t_1|^\alpha - |t_2|^\alpha
\le (|t_1|-|t_2|)^\alpha
\le |t_1-t_2|^\alpha
\le 2\psi(|t_1-t_2|).
\]
This completes the proof.
\end{proof}

We also need the following technical lemma that passes between the projected quantity \(\psi(z_k^i)\) and the original one \(\psi(z_k)\). 
\begin{lemma}\label{lemma:endpoint_conversion}
There exists \(C_{\mathrm{end}}>0\) such that the following holds.

If \(x_k\in \mathcal N(i,\alpha_k)\) for some \(i\in \llbracket 1,\overline T\rrbracket\), then
\[
|\psi(z_k^i)-\psi(z_k)|\le C_{\mathrm{end}}\alpha_k^{\beta_i(1-\theta)}.
\]
If, in addition, $d(x_k,M_i)\le 2c_i\alpha_k^{\gamma_i}$, then
\[
|\psi(z_k^i)-\psi(z_k)|\le 2\psi(|z_k^i - z_k|)\le C_{\mathrm{end}}\alpha_k^{p_i}.
\]
\end{lemma}

\begin{proof}
Since \(f\) is \(L\)-Lipschitz on \(X\),
\[
|z_k^i-z_k|=|f(y_k^i)-f(x_k)|\le L|y_k^i-x_k|=Ld(x_k,M_i).
\]
If \(x_k\in \mathcal N(i,\alpha_k)\), then \(d(x_k,M_i)\le c_i\alpha_k^{\beta_i}\), hence $|z_k^i-z_k|\le Lc_i\alpha_k^{\beta_i}$.
Using Fact \ref{fact:psi_holder},
\[
|\psi(z_k^i)-\psi(z_k)|
\le 2(Lc_i)^{1-\theta}\alpha_k^{\beta_i(1-\theta)}.
\]
This proves the first estimate.

If, in addition, $d(x_k,M_i)\le 2c_i\alpha_k^{\gamma_i}$, by the same arguments above,
\[
|\psi(z_k^i)-\psi(z_k)|\le 2\psi(|z_k^i- z_k|)
\le 2(2Lc_i)^{1-\theta}\alpha_k^{\gamma_i(1-\theta)}
=2(2Lc_i)^{1-\theta}\alpha_k^{p_i}.
\]
After enlarging \(C_{\mathrm{end}}\), both bounds hold.
\end{proof}
The following lemma bounds the relative length on an interval that starts at a distinguished index \(l\in\mathcal L\) and either stays near the same non-open stratum up to \(q(l)\), or leaves that neighborhood and then continues inside a single open stratum. This will serve as the main local estimate from which the later bounds are derived.
\begin{lemma}\label{lemma:core_bridge}
There exist \(\hat C>0\) such that the following holds.

Let \(l\in \mathcal L\) and \(k_2\in \llbracket l,K\rrbracket\) satisfy either $k_2\in \llbracket l,q(l)\rrbracket$ or
\(
[q(l)+1,k_2)\cap I_C=\emptyset
\).
Define
\[
\widetilde z_{k_2}^{\,l}:=
\begin{cases}
z_{k_2}^{G(l)}, & k_2\le q(l),\\[0.3em]
z_{k_2}, & k_2>q(l).
\end{cases}
\]
Then
\begin{equation}\label{eq:core_bridge}
\psi(z_l^{G(l)})-\psi(\widetilde z_{k_2}^{\,l})
\ge
\frac1C\,\mathrm{RL}(x_{\llbracket l,k_2\rrbracket})
-\sum_{k=l}^{k_2-1}E_k
-\hat C\,\alpha_l^{p_{G(l)}}.
\end{equation}
\end{lemma}

\begin{proof}
Let \(i=G(l)\) and \(q=q(l)\).

\medskip
\noindent\textit{Case 1: \(k_2\le q\).}
Then \(x_k\in \mathcal N(i,\alpha_k)\) for all \(k\in \llbracket l,k_2\rrbracket\). Applying Corollary~\ref{cor:y_k^i} on this interval gives
\[
\psi(z_l^i)-\psi(z_{k_2}^i)
\ge
\frac1C\sum_{k=l}^{k_2-1}|y_{k+1}^i-y_k^i|
-\sum_{k=l}^{k_2-1}E_k
-\alpha_l.
\]
By the definition of \(\mathrm{RL}(x_{\llbracket l,k_2\rrbracket})\),
\[
\mathrm{RL}(x_{\llbracket l,k_2\rrbracket})
=
|x_l-y_l^i|+\sum_{k=l}^{k_2-1}|y_{k+1}^i-y_k^i|.
\]
Hence
\[
\psi(z_l^i)-\psi(z_{k_2}^i)
\ge
\frac1C\,\mathrm{RL}(x_{\llbracket l,k_2\rrbracket})
-\sum_{k=l}^{k_2-1}E_k
-\frac1C|x_l-y_l^i|
-\alpha_l.
\]
Since \(l\in \mathcal L\subset I_C\) and \(i=G(l)\), we have $|x_l-y_l^i|=d(x_l,M_i)\le 2c_i\alpha_l^{\gamma_i}$. By \eqref{eq:relation_beta_p}, \(p_i=\gamma_i(1-\theta)<1\) and \(\alpha_l\le 1\), both \(\alpha_l^{\gamma_i}\) and \(\alpha_l\) are bounded by a constant multiple of \(\alpha_l^{p_i}\). This proves \eqref{eq:core_bridge} in the case \(k_2\le q\).

\medskip
\noindent\textit{Case 2: \(k_2>q\).}
Since \([q+1,k_2)\cap I_C=\emptyset\), Fact~\ref{fact:IC_away} yields an open stratum \(M_{i_0}\), \(i_0\in \llbracket T+1,\overline T\rrbracket\), such that
\[
x_k\in \mathcal N(i_0,\alpha_k)\qquad \forall k\in \llbracket q+1,k_2\rrbracket.
\]
As \(M_{i_0}\) is open,
\[
y_k^{i_0}=x_k,\qquad z_k^{i_0}=z_k,\qquad \forall k\in \llbracket q+1,k_2\rrbracket.
\]

Applying Corollary~\ref{cor:y_k^i} on \(\llbracket l,q\rrbracket\) along \(M_i\), and on \(\llbracket q+1,k_2\rrbracket\) along \(M_{i_0}\), we obtain
\[
\psi(z_l^i)-\psi(z_q^i)
\ge
\frac1C\sum_{k=l}^{q-1}|y_{k+1}^i-y_k^i|
-\sum_{k=l}^{q-1}E_k
-\alpha_l,
\]
and
\[
\psi(z_{q+1})-\psi(z_{k_2})
\ge
\frac1C\sum_{k=q+1}^{k_2-1}|x_{k+1}-x_k|
-\sum_{k=q+1}^{k_2-1}E_k
-\alpha_{q+1}.
\]
Since \(\alpha_{q+1}\le \alpha_l\), summing the two inequalities and inserting the bridge terms at \(q\), we get
\begin{align*}
\psi(z_l^i)-\psi(z_{k_2})
\ge{}&
\frac1C\sum_{k=l}^{q-1}|y_{k+1}^i-y_k^i|
+\frac1C\sum_{k=q+1}^{k_2-1}|x_{k+1}-x_k|
-\sum_{k=l}^{k_2-1}E_k
-2\alpha_l\\
&+\bigl(\psi(z_q^i)-\psi(z_q)\bigr)
+\bigl(\psi(z_q)-\psi(z_{q+1})\bigr)\\
\ge{}&
\frac1C\sum_{k=l}^{q-1}|y_{k+1}^i-y_k^i|
+\frac1C\sum_{k=q+1}^{k_2-1}|x_{k+1}-x_k|
-\sum_{k=l}^{k_2-1}E_k
-2\alpha_l\\
&-2\psi(|z_q^i-z_q|)
-2\psi(|z_q-z_{q+1}|),
\end{align*}
where we used Fact \ref{fact:psi_holder}.
Now
\[
\mathrm{RL}(x_{\llbracket l,k_2\rrbracket})
\le
|x_l-y_l^i|
+\sum_{k=l}^{q-1}|y_{k+1}^i-y_k^i|
+|x_q-y_q^i|
+\sum_{k=q}^{k_2-1}|x_{k+1}-x_k|,
\]
Substituting this inequality yields
\begin{align}
\psi(z_l^i)-\psi(z_{k_2})
\ge{}&
\frac1C\,\mathrm{RL}(x_{\llbracket l,k_2\rrbracket})
-\sum_{k=l}^{k_2-1}E_k \notag\\
&-\frac1C\Bigl(|x_l-y_l^i|+|x_q-y_q^i|+|x_q-x_{q+1}|\Bigr)
-2\alpha_l \label{eq:error_1}\\
&-2\psi(|z_q^i-z_q|)-2\psi(|z_q-z_{q+1}|).\label{eq:error_2}
\end{align}
We now bound the error terms \eqref{eq:error_1}-\eqref{eq:error_2} on the right-hand side. By the definition of $l$ and $q$ in \eqref{eq:def_l}, we have
\[
d(x_l,M_i)\le 2c_i\alpha_l^{\gamma_i},\qquad d(x_q,M_i)\le 2c_i\alpha_q^{\gamma_i}\le 2c_i\alpha_l^{\gamma_i},
\]
because \((\alpha_k)\) is decreasing. Therefore,
\[
|x_l-y_l^i|\le 2c_i\alpha_l^{\gamma_i},
\qquad
|x_q-y_q^i|\le 2c_i\alpha_l^{\gamma_i}.
\]
Moreover, as $f$ is $L$-Lipschitz in $X$,
\[
|x_q-x_{q+1}|\le L\alpha_q\le L\alpha_l.
\]
By \eqref{eq:relation_beta_p}, we have \(p_i:= \gamma_i(1-\theta)<\gamma_i\) and \(p_i<1\). As \(\alpha_l\le 1\), it follows that
\[
\frac1C\Bigl(|x_l-y_l^i|+|x_q-y_q^i|+|x_q-x_{q+1}|\Bigr)+2\alpha_l
\le \frac1C\Bigl(2c_i\alpha_l^{\gamma_i}+2c_i\alpha_l^{\gamma_i}+L\alpha_l\Bigr)+2\alpha_l \le C_1\alpha_l^{p_i}
\]
for some \(C_1>0\).

To bound \eqref{eq:error_2}, by Lemma \ref{lemma:endpoint_conversion} and the definition of $q$ in \eqref{eq:def_l}, there exists \(C_{\mathrm{end}}>0\) such that
\[
\psi(|z_q^i-z_q|)
\le  C_{\mathrm{end}}\alpha_l^{p_i}/2.
\]

Finally, by the expression of $z$ in \eqref{eq:expr_g_E_z},
\[
|z_q-z_{q+1}|
\le |f(x_q)-f(x_{q+1})|+|g_q-g_{q+1}|
\le L|x_q-x_{q+1}|+c\alpha_q^{1+\hat\beta}
\le C_3\alpha_l
\]
for some \(C_3>0\), since \(\alpha_q\le \alpha_l\le 1\). Hence
\[
\psi(|z_q-z_{q+1}|)
\le C_4\alpha_l^{1-\theta}
\le C_4\alpha_l^{p_i},
\]
because \(1-\theta>p_i\) and \(\alpha_l\le 1\).

Substituting the above bounds into the previous inequality, there exists $\hat C$ such that
\[
\psi(z_l^i)-\psi(z_{k_2})
\ge
\frac1C\,\mathrm{RL}(x_{\llbracket l,k_2\rrbracket})
-\sum_{k=l}^{k_2-1}E_k
-\hat C\,\alpha_l^{p_i},
\]
which proves \eqref{eq:core_bridge} in the case \(k_2>q\). By enlarging $\hat C$ if necessary, the same estimate holds for Case 1.
\end{proof}

Combining Lemma \ref{lemma:core_bridge} with Lemma \ref{lemma:endpoint_conversion} yields a clean comparison across two consecutive indices in \(\mathcal L\). In particular, it controls the relative length accumulated between \(l_m\) and \(l_{m+1}\) by the corresponding drop in \(\psi\), up to summable errors and endpoint terms.
\begin{corollary}\label{cor:lm->lm+1}
There exists \(\overline C>0\) such that for any \(m=0,\ldots,\overline m-1\),
\[
\psi\left(z^{G(l_m)}_{l_m}\right)-\psi\left(z^{G(l_{m+1})}_{l_{m+1}}\right)
\ge \frac{1}{C}\,\mathrm{RL}(x_{\llbracket l_m,l_{m+1}\rrbracket})
-\sum_{k = l_m}^{l_{m+1}-1}E_k
-\overline{C} \left(\alpha_{l_m}^{p_{G(l_m)}} +\alpha_{l_{m+1}}^{p_{G(l_{m+1})}}\right).
\]
\end{corollary}

\begin{proof}
Apply Lemma~\ref{lemma:core_bridge} with \(l=l_m\) and \(k_2=l_{m+1}\). Since \(l_{m+1}>q(l_m)\) and \([q(l_m)+1,l_{m+1})\cap I_C=\emptyset\) by the definition of $l_{m+1}$ in \eqref{eq:def_l}, we obtain
\[
\psi(z_{l_m}^{G(l_m)})-\psi(z_{l_{m+1}})
\ge
\frac1C\,\mathrm{RL}(x_{\llbracket l_m,l_{m+1}\rrbracket})
-\sum_{k=l_m}^{l_{m+1}-1}E_k
-\overline C\,\alpha_{l_m}^{p_{G(l_m)}}.
\]
Since \(l_{m+1}\in I_C\), combining Fact \ref{fact:IC_close} and Lemma~\ref{lemma:endpoint_conversion} gives
\[
|\psi(z_{l_{m+1}}^{G(l_{m+1})})-\psi(z_{l_{m+1}})|
\le C_{\mathrm{end}}\alpha_{l_{m+1}}^{p_{G(l_{m+1})}}.
\]
Combining the two inequalities if necessary yields the claim for a $\overline{C}$ there is large enough.
\end{proof}
The next step is to understand what happens on a larger interval of the form \(\llbracket l,s(l)\rrbracket\). The following lemma shows that such an interval has a dichotomic structure: either the relative length is already large enough to be useful, or the step size must decrease by a factor of two before the sequence can leave the corresponding neighborhood. This provides the combinatorial mechanism that will later control the cumulative endpoint errors.

\begin{lemma}\label{lemma:RM_or_half}
For any $Q>0$, there exists $\bar\alpha_Q\in (0,\hat\alpha]$ such that, for every $\alpha\in (0,\bar\alpha_Q]$,
\begin{equation}\label{eq:alphaQ_case1_new}
c_i2^{-\beta_i}\alpha^{\beta_i-p_i}-2c_i\alpha^{\gamma_i-p_i}-L\alpha^{1-p_i}\ge Q+\bar c,
\qquad \forall i\in \llbracket 1,T\rrbracket,
\end{equation}
and
\begin{equation}\label{eq:alphaQ_case2_new}
c_j\alpha^{\gamma_j-p_i}-L\alpha^{1-p_i}\ge Q+\bar c,
\qquad \forall i,j\in \llbracket 1,T\rrbracket \text{ with } M_j\subset \partial M_i.
\end{equation}
If $\bar\alpha\le \bar\alpha_Q$, then for any $l\in \mathcal L$ such that $s(l)<K$, it holds that either
\[
\mathrm{RL}(x_{\llbracket l,s(l)\rrbracket})\ge Q\alpha_l^{p_{G(l)}}
\qquad\text{or}\qquad
\alpha_{s(l)+1}\le \alpha_l/2.
\]
\end{lemma}

\begin{proof}
The existence of $\bar\alpha_Q$ satisfying \eqref{eq:alphaQ_case1_new}--\eqref{eq:alphaQ_case2_new}
follows from \eqref{eq:relation_beta_p}. Fix \(l\in \mathcal L\) such that \(s(l)<K\), and write
\[
i:=G(l),\qquad p:=p_i,\qquad s:=s(l).
\]

\medskip
\noindent\underline{Step 1: The case \(s=l\).}
Assume \(s=l\), and use the convention \(\mathrm{RL}(x_{\llbracket l,l\rrbracket})=0\).
It is enough to prove \(\alpha_{l+1}\le \alpha_l/2\). Suppose on the contrary that
\(\alpha_{l+1}>\alpha_l/2\). Since \(x_l\in \mathcal N(i,\alpha_l)\) and
\(x_{l+1}\notin \mathcal N(i,\alpha_{l+1})\), either $d(x_{l+1},M_i)>c_i\alpha_{l+1}^{\beta_i}$, or there exists \(j\in \llbracket 1,T\rrbracket\) with \(M_j\subset \partial M_i\) such that $d(x_{l+1},M_j)\le c_j\alpha_{l+1}^{\gamma_j}$.

In the first case, Fact \ref{fact:IC_close} gives \(d(x_l,M_i)\le 2c_i\alpha_l^{\gamma_i}\), hence
\[
L\alpha_l\ge |x_{l+1}-x_l|
\ge d(x_{l+1},M_i)-d(x_l,M_i)
> c_i\alpha_{l+1}^{\beta_i}-2c_i\alpha_l^{\gamma_i}
\ge c_i2^{-\beta_i}\alpha_l^{\beta_i}-2c_i\alpha_l^{\gamma_i},
\]
contradicting \eqref{eq:alphaQ_case1_new}.

In the second case, the definition of $G$ in \eqref{eq:def_G} yields $d(x_l,M_j)\ge 2c_j\alpha_l^{\gamma_j}$, and therefore
\[
L\alpha_l\ge |x_{l+1}-x_l|
\ge d(x_l,M_j)-d(x_{l+1},M_j)
\ge 2c_j\alpha_l^{\gamma_j}-c_j\alpha_{l+1}^{\gamma_j}
\ge c_j\alpha_l^{\gamma_j},
\]
as $\alpha_{l+1}\le \alpha_l$, contradicting \eqref{eq:alphaQ_case2_new}. Hence \(\alpha_{l+1}\le \alpha_l/2\).

\medskip
\noindent\underline{Step 2: Reduction to a lower bound on \(|x_l-x_s|\).}
From now on, assume that \(s\ge l+1\) and \(\alpha_{s+1}>\alpha_l/2\). We aim to prove
\[
\mathrm{RL}(x_{\llbracket l,s\rrbracket})\ge Q\alpha_l^p.
\]
In this step, we show that it suffices to prove
\begin{equation}\label{eq:displacement_bound}
    |x_l-x_s|\ge (Q+\bar c)\alpha_l^p.
\end{equation}

Let
\[
m_2:=\max\{m\in \llbracket 0,\overline m\rrbracket:l_m<s\},
\qquad
i_2:=G(l_{m_2}).
\]
By Lemma~\ref{lemma:RM_multi}, one of the following two alternatives holds.

\smallskip
\noindent\emph{Case 1: \(M_{i_2}\not\subset \partial M_i\).}
In this case,
\begin{equation}\label{eq:x_to_RM_sorted}
|x_l-x_s|
\le \mathrm{RL}(x_{\llbracket l,s\rrbracket})+\bar c\,\alpha_l^{\underline\beta_i}
\le \mathrm{RL}(x_{\llbracket l,s\rrbracket})+\bar c\,\alpha_l^p,
\end{equation}
where the second inequality uses \(\underline\beta_i>p\) and \(\alpha_l\le 1\). Hence,
if \eqref{eq:displacement_bound} holds, then
\[
\mathrm{RL}(x_{\llbracket l,s\rrbracket})
\ge |x_l-x_s|-\bar c\,\alpha_l^p
\ge Q\alpha_l^p.
\]
Thus the desired lower bound follows in this case from \eqref{eq:displacement_bound}.

\smallskip
\noindent\emph{Case 2: \(M_{i_2}\subset \partial M_i\).}
Write \(l=l_{m_1}\). By Definition~\ref{def:RL} and Lemma~\ref{lemma:RM_within_one},
\begin{align}
\mathrm{RL}(x_{\llbracket l,s\rrbracket})
&=
\mathrm{RL}(x_{\llbracket l_{m_1},l_{m_2}\rrbracket})
+
\mathrm{RL}(x_{\llbracket l_{m_2},s\rrbracket}) \notag\\
&=
\sum_{m=m_1}^{m_2-1}
\mathrm{RL}(x_{\llbracket l_m,l_{m+1}\rrbracket})
+
\mathrm{RL}(x_{\llbracket l_{m_2},s\rrbracket}) \notag\\
&\ge
|x_l-x_{l_{m_2}}|
+
\mathrm{RL}(x_{\llbracket l_{m_2},s\rrbracket}).
\label{eq:RL_split_last_block}
\end{align}

We now distinguish two subcases.

\smallskip
\noindent\emph{Subcase 2a: \(l_{m_2}=q(l_{m_2})\).}
By the definition of \(m_2\), we have $s\in \llbracket l_{m_2},l_{m_2+1}\rrbracket
=
\llbracket q(l_{m_2}),l_{m_2+1}\rrbracket$. Therefore, by Lemma~\ref{lemma:RM_within_one} and the triangular inequality,
\[
\mathrm{RL}(x_{\llbracket l,s\rrbracket})
\ge |x_l-x_{l_{m_2}}|
   + \mathrm{RL}(x_{\llbracket l_{m_2},s\rrbracket})
\ge |x_l-x_{l_{m_2}}|+|x_{l_{m_2}}-x_s|
\ge |x_l-x_s|.
\]
Together with \eqref{eq:displacement_bound}, this gives $\mathrm{RL}(x_{\llbracket l,s\rrbracket})\ge Q\alpha_l^p$.

\smallskip
\noindent\emph{Subcase 2b: \(l_{m_2}<q(l_{m_2})\).}
In this case, \(l_{m_2}+1\le s\). Hence, by Definition~\ref{def:RL} and
Remark~\ref{rem:RL_basic},
\[
\mathrm{RL}(x_{\llbracket l_{m_2},s\rrbracket})
\ge
\mathrm{RL}(x_{\llbracket l_{m_2},l_{m_2}+1\rrbracket})
=
|x_{l_{m_2}}-y^{i_2}_{l_{m_2}}|
+
|y^{i_2}_{l_{m_2}}-y^{i_2}_{l_{m_2}+1}|
\ge
|x_{l_{m_2}}-y^{i_2}_{l_{m_2}+1}|.
\]
Combining this with \eqref{eq:RL_split_last_block}, we get
\[
\mathrm{RL}(x_{\llbracket l,s\rrbracket})
\ge
|x_l-x_{l_{m_2}}|+|x_{l_{m_2}}-y^{i_2}_{l_{m_2}+1}|
\ge
|x_l-y^{i_2}_{l_{m_2}+1}|.
\]
Since \(y^{i_2}_{l_{m_2}+1}\in M_{i_2}\), it follows that $|x_l-y^{i_2}_{l_{m_2}+1}|\ge d(x_l,M_{i_2})$. Moreover, because \(M_{i_2}\subset \partial M_i\), the definition of \(i=G(l)\) in
\eqref{eq:def_G} gives $d(x_l,M_{i_2})>2c_{i_2}\alpha_l^{\gamma_{i_2}}$. Consequently,
\[
\mathrm{RL}(x_{\llbracket l,s\rrbracket})
\ge d(x_l,M_{i_2})
>2c_{i_2}\alpha_l^{\gamma_{i_2}}
\ge Q\alpha_l^p,
\]
where the last inequality follows from \eqref{eq:alphaQ_case2_new}. Thus this subcase is
settled directly, without using \eqref{eq:displacement_bound}.

\smallskip
Combining the two cases, we have shown that either $\mathrm{RL}(x_{\llbracket l,s\rrbracket})\ge Q\alpha_l^p$ already holds, or it follows from \eqref{eq:displacement_bound}. Therefore, in the remainder
of the proof it remains only to establish \eqref{eq:displacement_bound}.

\medskip
\noindent\underline{Step 3: Exit from \(\mathcal N(i,\alpha)\).}
By the definition of \(s(l)\), we have \(x_k\in \mathcal N(i,\alpha_k)\) for all
\(k=l,\dots,s\), while \(x_{s+1}\notin \mathcal N(i,\alpha_{s+1})\). Since
\(x_{\llbracket 0,K\rrbracket}\subset X\cap [|f|\le \epsilon]\), either
\[
d(x_{s+1},M_i)>c_i\alpha_{s+1}^{\beta_i},
\]
or there exists \(j\in \llbracket 1,T\rrbracket\) such that \(M_j\subset \partial M_i\) and
\[
d(x_{s+1},M_j)\le c_j\alpha_{s+1}^{\gamma_j}.
\]

We first consider the case \(d(x_{s+1},M_i)>c_i\alpha_{s+1}^{\beta_i}\). Recall that Fact \ref{fact:IC_close} gives
\(d(x_l,M_i)\le 2c_i\alpha_l^{\gamma_i}\). So
\[
|x_s-x_l|
\ge |x_{s+1}-x_l|-|x_{s+1}-x_s|
\ge d(x_{s+1},M_i)-d(x_l,M_i)-L\alpha_s
\ge c_i\alpha_{s+1}^{\beta_i}-2c_i\alpha_l^{\gamma_i}-L\alpha_s.
\]
Using \(\alpha_{s+1}>\alpha_l/2\) and \(\alpha_s\le \alpha_l\), we obtain
\[
|x_s-x_l|
\ge c_i2^{-\beta_i}\alpha_l^{\beta_i}-2c_i\alpha_l^{\gamma_i}-L\alpha_l
\ge (Q+\bar c)\alpha_l^p
\]
by \eqref{eq:alphaQ_case1_new}.

It remains to consider the case \(d(x_{s+1},M_j)\le c_j\alpha_{s+1}^{\gamma_j}\) for some
\(M_j\subset \partial M_i\). The definition of $G$ in \eqref{eq:def_G} gives $d(x_l,M_j)\ge 2c_j\alpha_l^{\gamma_j}$. Hence
\[
|x_s-x_l|
\ge |x_{s+1}-x_l|-|x_{s+1}-x_s|
\ge d(x_l,M_j)-d(x_{s+1},M_j)-L\alpha_s
\ge 2c_j\alpha_l^{\gamma_j}-c_j\alpha_{s+1}^{\gamma_j}-L\alpha_s.
\]
Since \(\alpha_{s+1}\le \alpha_s\le \alpha_l\), it follows that
\[
|x_s-x_l|
\ge c_j\alpha_l^{\gamma_j}-L\alpha_l
\ge (Q+\bar c)\alpha_l^p
\]
by \eqref{eq:alphaQ_case2_new}.

Thus, in all cases,
\[
|x_s-x_l|\ge (Q+\bar c)\alpha_l^p.
\]
With the claim proved in Step 2, we have
\[
\mathrm{RL}(x_{\llbracket l,s\rrbracket})\ge Q\alpha_l^p.
\]
This completes the proof.
\end{proof}
We now invoke Lemma~\ref{lemma:RM_or_half} to estimate the sum of the quantities
\(\alpha_\ell^{p_{G(\ell)}}\) over \(\ell\in \mathcal L\cap [] k_1,k_2]\).
These are precisely the error terms produced when the local estimates of
Corollary~\ref{cor:lm->lm+1} are summed over consecutive blocks. The next lemma shows
that their total contribution is bounded by an arbitrarily small multiple of the relative
length, up to an additional remainder of order \(\alpha_{k_1}^{\underline p}\). For any finite
set \(A\), we denote its cardinality by \(|A|\).
\begin{lemma}\label{lemma:sum_alpha_p}
Given \(0\le k_1<k_2\le K\), let
\(
u:=\bigl|G([k_1,k_2]\cap I_C)\bigr|
\).
For any \(Q>0\), let \(\bar{\alpha}_Q\) be given by Lemma~\ref{lemma:RM_or_half}. If
\(\bar\alpha\le \bar{\alpha}_Q\), then
\begin{equation}\label{eq:sum_alpha_p_final}
\sum_{\ell\in \mathcal L\cap [k_1,k_2]}\alpha_{\ell}^{p_{G(\ell)}}
\le
\frac{u}{Q}\,\mathrm{RL}(x_{\llbracket k_1,k_2\rrbracket})
+
u\alpha_{k_1}^{\underline p}
\left(
1+\sum_{j=0}^{\left\lfloor\log_2(\alpha_{k_1}/\alpha_{k_2})\right\rfloor-1}2^{-j\underline p}
\right).
\end{equation}
\end{lemma}

\begin{proof}
We split the sum into two parts according to whether the corresponding block has already ended before \(k_2\):
\begin{equation}\label{eq:alpha_split}
 \sum_{\ell\in \mathcal L\cap [k_1,k_2]}\alpha_{\ell}^{p_{G(\ell)}}
=
\sum_{\substack{\ell\in \mathcal L\\ k_1 \le \ell \le k_2,\ s(\ell)<k_2}}\alpha_{\ell}^{p_{G(\ell)}}
+
\sum_{\substack{\ell\in \mathcal L\\ k_1 \le \ell \le k_2\le s(\ell)}}\alpha_{\ell}^{p_{G(\ell)}}.   
\end{equation}

We start by estimating the first sum. For each \(i\in \llbracket 1,T\rrbracket\), define
\[
\mathcal L_i:= \{\ell\in \mathcal L \cap [k_1,k_2]: \ s(\ell)<k_2,\ G(\ell) = i\}.
\]
Then \(\{\mathcal L_i\}_{i\in \llbracket 1,T\rrbracket}\) forms a partition of
\(\{\ell\in \mathcal L\cap [k_1,k_2]: s(\ell)<k_2\}\). Moreover, if \(\mathcal L_i\neq \emptyset\), then
\(
G(\mathcal L_i)=\{i\}\subset G([k_1,k_2]\cap I_C),
\)
so this can happen for at most \(u\) different values of \(i\).

Fix \(i\in \llbracket 1,T\rrbracket\) such that \(\mathcal L_i\neq \emptyset\), and write
\[
\mathcal L_i=\{\ell_1<\ell_2<\cdots<\ell_r\}.
\]
By Fact~\ref{fact:no_repeat},
\begin{equation}\label{eq:ell_sep_revised}
\ell_{t+1}>s(\ell_t)\qquad \forall t\in \llbracket 1,r-1\rrbracket.
\end{equation}
Since \(s(\ell_t)<k_2\le K\), Lemma~\ref{lemma:RM_or_half} applies to each \(\ell_t\), and yields
\[
\mathrm{RL}(x_{\llbracket \ell_t,s(\ell_t)\rrbracket})
\ge Q\bigl(1-I_t\bigr)\alpha_{\ell_t}^{p_i},
\]
where
\[
I_t:= \left\{\begin{array}{ll}
1 &\text{if }\alpha_{s(\ell_t)+1}\le \alpha_{\ell_t}/2,\\
0 &\text{if }\alpha_{s(\ell_t)+1}> \alpha_{\ell_t}/2.
\end{array}\right.
\]
Hence
\[
\alpha_{\ell_t}^{p_i}
\le \frac1Q\,\mathrm{RL}(x_{\llbracket \ell_t,s(\ell_t)\rrbracket})
+I_t\alpha_{\ell_t}^{p_i}.
\]
Summing over \(t=1,\dots,r\), we obtain
\begin{equation}\label{eq:sum_fixed_i_revised}
\sum_{\ell\in \mathcal L_i}\alpha_\ell^{p_i}
\le \frac1Q\sum_{t=1}^r\mathrm{RL}(x_{\llbracket \ell_t,s(\ell_t)\rrbracket})
+\sum_{t=1}^r I_t\alpha_{\ell_t}^{p_i}.
\end{equation}

By \eqref{eq:ell_sep_revised}, the intervals
\[
\llbracket \ell_t,s(\ell_t)\rrbracket,\qquad t=1,\dots,r,
\]
are pairwise disjoint and all lie in \(\llbracket k_1,k_2\rrbracket\). By Definition~\ref{def:RL}, each
\(\mathrm{RL}(x_{\llbracket \ell_t,s(\ell_t)\rrbracket})\) is a sum of some of the nonnegative terms appearing in
\(\mathrm{RL}(x_{\llbracket k_1,k_2\rrbracket})\). Since the corresponding intervals are disjoint, these groups of terms are disjoint as well. Therefore,
\[
\sum_{t=1}^r\mathrm{RL}(x_{\llbracket \ell_t,s(\ell_t)\rrbracket})
\le \mathrm{RL}(x_{\llbracket k_1,k_2\rrbracket}).
\]

It remains to estimate \(\sum_{t=1}^r I_t\alpha_{\ell_t}^{p_i}\). Let 
\[
J:=\{t\in \llbracket 1,r\rrbracket:I_t=1\}=\{t_1<\cdots<t_B\}.
\]
If \(J=\emptyset\), then this sum is zero. Otherwise, since \(I_{t_b}=1\), we have
\[
\alpha_{s(\ell_{t_b})+1}\le \alpha_{\ell_{t_b}}/2.
\]
If \(b<B\), then by \eqref{eq:ell_sep_revised}, we have
\[
\alpha_{\ell_{t_b+1}}\le \alpha_{s(\ell_{t_b})+1}\le \alpha_{\ell_{t_b}}/2.
\]
It follows inductively that
\[
\alpha_{\ell_{t_b}}\le 2^{-(b-1)}\alpha_{k_1},
\qquad b=1,\dots,B.
\]
Therefore,
\[
\sum_{t=1}^r I_t\alpha_{\ell_t}^{p_i}
=\sum_{b=1}^{B}\alpha_{\ell_{t_b}}^{p_i}
\le \alpha_{k_1}^{p_i}\sum_{b=0}^{B-1}2^{-bp_i}.
\]

We next bound \(B\). Since \(s(\ell_{t_B})<k_2\), one has \(s(\ell_{t_B})+1\le k_2\), so
\[
\alpha_{k_2}\le \alpha_{s(\ell_{t_B})+1}\le \alpha_{\ell_{t_B}}/2\le 2^{-B}\alpha_{k_1}.
\]
Thus
\[
B\le \left\lfloor \log_2\frac{\alpha_{k_1}}{\alpha_{k_2}}\right\rfloor,
\]
and consequently
\[
\sum_{t=1}^r I_t\alpha_{\ell_t}^{p_i}
\le \alpha_{k_1}^{p_i}\sum_{j=0}^{\left\lfloor\log_2(\alpha_{k_1}/\alpha_{k_2})\right\rfloor-1}2^{-jp_i}.
\]

Substituting the last two bounds into \eqref{eq:sum_fixed_i_revised}, we get
\[
\sum_{\ell\in \mathcal L_i}\alpha_\ell^{p_i}
\le \frac1Q\,\mathrm{RL}(x_{\llbracket k_1,k_2\rrbracket})
+\alpha_{k_1}^{p_i}\sum_{j=0}^{\left\lfloor\log_2(\alpha_{k_1}/\alpha_{k_2})\right\rfloor-1}2^{-jp_i}.
\]
Summing this inequality over all \(i\in \llbracket 1,T\rrbracket\) such that \(\mathcal L_i\neq \emptyset\), and using that there are at most \(u\) such indices, yields
\begin{align}
\sum_{\substack{\ell\in \mathcal L\\ k_1 \le \ell \le k_2,\ s(\ell)<k_2}}\alpha_{\ell}^{p_{G(\ell)}}
&= \sum_{i=1}^T\sum_{\ell\in \mathcal L_i}\alpha_\ell^{p_i}\notag\\
&\le \frac{u}{Q}\,\mathrm{RL}(x_{\llbracket k_1,k_2\rrbracket})
+\sum_{i\in \llbracket 1,T\rrbracket: \mathcal L_i \ne \emptyset} \alpha_{k_1}^{p_i}
\sum_{j=0}^{\left\lfloor\log_2(\alpha_{k_1}/\alpha_{k_2})\right\rfloor-1}2^{-jp_i}\notag\\
&\le \frac{u}{Q}\,\mathrm{RL}(x_{\llbracket k_1,k_2\rrbracket})
+u\alpha_{k_1}^{\underline p}\sum_{j=0}^{\left\lfloor\log_2(\alpha_{k_1}/\alpha_{k_2})\right\rfloor-1}2^{-j\underline p}. \label{eq:first_sum_bound_final}
\end{align}

We now estimate the second sum in \eqref{eq:alpha_split}. For each \(i\in \llbracket 1,T\rrbracket\), define
\[
\mathcal J_i:=\{\ell\in \mathcal L:k_1\le \ell\le k_2\le s(\ell),\ G(\ell)=i\}.
\]
Again, \(\{\mathcal J_i\}_{i\in \llbracket 1,T\rrbracket}\) forms a partition of
\(\{\ell\in \mathcal L:k_1\le \ell\le k_2\le s(\ell)\}\), and \(\mathcal J_i\neq \emptyset\) for at most \(u\) different values of \(i\). Also, \(|\mathcal J_i|\le 1\) for every \(i\). Indeed, if \(\bar\ell<\tilde\ell\) both belong to \(\mathcal J_i\), then Fact~\ref{fact:no_repeat} gives
\[
\tilde\ell>s(\bar\ell)\ge k_2,
\]
contradicting \(\tilde\ell\le k_2\). Hence
\begin{equation}\label{eq:second_sum_bound_final}
    \sum_{\substack{\ell\in \mathcal L\\ k_1 \le \ell \le k_2\le s(\ell)}}\alpha_{\ell}^{p_{G(\ell)}}
= \sum_{i=1}^T\sum_{\ell\in \mathcal J_i}\alpha_\ell^{p_i} \le \sum_{i\in \llbracket 1,T\rrbracket: \mathcal J_i \ne \emptyset}\alpha_{k_1}^{p_i}
\le u\alpha_{k_1}^{\underline p}.
\end{equation}

Combining \eqref{eq:first_sum_bound_final} and \eqref{eq:second_sum_bound_final} proves \eqref{eq:sum_alpha_p_final}.
\end{proof}
We are now ready to pass from local blockwise estimates to a bound on the relative length over an arbitrary interval whose endpoints lie in \(\mathcal L\). The next lemma is obtained by summing Corollary~\ref{cor:lm->lm+1} over consecutive blocks and then absorbing the resulting endpoint errors by Lemma~\ref{lemma:sum_alpha_p}. It is the main estimate of this subsection.

\begin{lemma}\label{lemma:RM_arbitrary_interval}
Let \(k_1,k_2\in \mathcal L\) satisfy \(k_1<k_2\), and let $u:=\bigl|G([k_1,k_2]\cap I_C)\bigr|$. Let $\overline C>0$ be given by Corollary \ref{cor:lm->lm+1}.  Fix \(Q>2\overline C Cu\) and let $\bar{\alpha}_Q$ be given by Lemma \ref{lemma:RM_or_half}.  If  $\bar{\alpha}\le \bar{\alpha}_Q$, then
\begin{align*}
\mathrm{RL}(x_{\llbracket k_1,k_2\rrbracket})
\le \frac{C}{1-2\overline C Cu/Q}\Bigg(
&\psi\bigl(z_{k_1}^{G(k_1)}\bigr)-\psi\bigl(z_{k_2}^{G(k_2)}\bigr)
+\sum_{k=k_1}^{k_2-1}E_k\\
&+2\overline C\,u\,\alpha_{k_1}^{\underline p}
\Bigg(1+\sum_{j=0}^{\left\lfloor\log_2(\alpha_{k_1}/\alpha_{k_2})\right\rfloor-1}2^{-j\underline p}\Bigg)
\Bigg).
\end{align*}
\end{lemma}

\begin{proof}
Let \(m_1,m_2\in \llbracket 0,\overline m\rrbracket\) be such that \(k_1=l_{m_1}\) and \(k_2=l_{m_2}\). Since \(k_1,k_2\in \mathcal L\) and \(k_1<k_2\), we have \(m_1<m_2\). Applying Corollary~\ref{cor:lm->lm+1} for \(m=m_1,\ldots,m_2-1\), and summing the resulting inequalities, we obtain
\begin{align*}
\psi\bigl(z_{k_1}^{G(k_1)}\bigr)-\psi\bigl(z_{k_2}^{G(k_2)}\bigr)={}&\sum_{m=m_1}^{m_2-1}\Bigl(\psi\bigl(z_{l_m}^{G(l_m)}\bigr)-\psi\bigl(z_{l_{m+1}}^{G(l_{m+1})}\bigr)\Bigr)\\
\ge{}&
\frac1C\sum_{m=m_1}^{m_2-1}\mathrm{RL}(x_{\llbracket l_m,l_{m+1}\rrbracket})
-\sum_{m=m_1}^{m_2-1}\sum_{k=l_m}^{l_{m+1}-1}E_k\\
&-\overline C\sum_{m=m_1}^{m_2-1}\Bigl(\alpha_{l_m}^{p_{G(l_m)}}+\alpha_{l_{m+1}}^{p_{G(l_{m+1})}}\Bigr).
\end{align*}
By the definition of \(\mathrm{RL}\) on a union of consecutive intervals in Definition \ref{def:RL}, the first sum equals
\[
\sum_{m=m_1}^{m_2-1}\mathrm{RL}(x_{\llbracket l_m,l_{m+1}\rrbracket})
=
\mathrm{RL}(x_{\llbracket k_1,k_2\rrbracket}),
\]
and the double sum telescopes to
\[
\sum_{m=m_1}^{m_2-1}\sum_{k=l_m}^{l_{m+1}-1}E_k
=
\sum_{k=k_1}^{k_2-1}E_k.
\]
Moreover, by Lemma~\ref{lemma:sum_alpha_p},
\begin{align*}
 \sum_{m=m_1}^{m_2-1}\Bigl(\alpha_{l_m}^{p_{G(l_m)}}+\alpha_{l_{m+1}}^{p_{G(l_{m+1})}}\Bigr)
&\le 2\sum_{\ell\in \mathcal L\cap [k_1,k_2]}\alpha_\ell^{p_{G(\ell)}}\\ &\le\frac{2u}{Q}\,\mathrm{RL}(x_{\llbracket k_1,k_2\rrbracket})
+
2u\alpha_{k_1}^{\underline p}
\left(
1+\sum_{j=0}^{\left\lfloor\log_2(\alpha_{k_1}/\alpha_{k_2})\right\rfloor-1}2^{-j\underline p}
\right).
\end{align*}
Thus,
\begin{align*}
\psi\bigl(z_{k_1}^{G(k_1)}\bigr)-\psi\bigl(z_{k_2}^{G(k_2)}\bigr)
\ge{}&
\Bigl(\frac1C-\frac{2\overline Cu}{Q}\Bigr)\mathrm{RL}(x_{\llbracket k_1,k_2\rrbracket})
-\sum_{k=k_1}^{k_2-1}E_k\\
&-2\overline C\,u\,\alpha_{k_1}^{\underline p}
\Bigg(1+\sum_{j=0}^{\left\lfloor\log_2(\alpha_{k_1}/\alpha_{k_2})\right\rfloor-1}2^{-j\underline p}\Bigg).
\end{align*}
Since \(Q>2\overline C Cu\), the coefficient \(\frac1C-\frac{2\overline Cu}{Q}\) is positive. Rearranging the above inequality gives the desired estimate.
\end{proof}
Finally, to bound the relative length over an arbitrary interval, we also need a one-sided version of the previous bound when the interval ends at the terminal index \(K\). Since there is no further distinguished index after the last element of \(\mathcal L\), this case must be handled separately. The following lemma provides the required estimate on the terminal piece.
\begin{lemma}\label{lemma:RM_terminal_piece}
There exist  \(C_{\mathrm{tail},1},C_{\mathrm{tail},2}>0\) such that the following holds.

Assume that $\mathcal L \ne \emptyset$ and let \(l:=\max_{\ell \in \mathcal L} \ell\). Then
\begin{equation}\label{eq:terminal_piece_RM}
\mathrm{RL}(x_{\llbracket l,K\rrbracket})
\le
C_{\mathrm{tail},1}\left(
\psi(z_l^{G(l)})-\psi(z_K)\right)
+C_{\mathrm{tail},2}\left(\sum_{k=l}^{K-1}E_k
+\alpha_l^{\underline\beta(1-\theta)}
\right).
\end{equation}
\end{lemma}

\begin{proof}
Write
\(
i:=G(l)\) and \(q:=q(l)\). 
It suffices to prove that
\begin{equation}\label{eq:terminal_RM_lower}
\psi(z_l^i)-\psi(z_K)
\ge
\frac1C\,\mathrm{RL}(x_{\llbracket l,K\rrbracket})
-\sum_{k=l}^{K-1}E_k
-C_1\alpha_l^{\underline\beta(1-\theta)}
\end{equation}
for some constant \(C_1>0\).

We distinguish two cases.

\medskip
\noindent\textit{Case 1: \(K>q\).}
Since \([q+1,K)\cap I_C=\emptyset\), Lemma~\ref{lemma:core_bridge} applies with \(k_2=K\), and yields
\[
\psi(z_l^i)-\psi(z_K)
\ge
\frac1C\,\mathrm{RL}(x_{\llbracket l,K\rrbracket})
-\sum_{k=l}^{K-1}E_k
-\hat C\,\alpha_l^{p_i}.
\]
By \eqref{eq:relation_beta_p}, $\alpha_l^{p_i}\le \alpha_l^{\underline\beta(1-\theta)}$ and thus \eqref{eq:terminal_RM_lower} follows in this case.

\medskip
\noindent\textit{Case 2: \(K\le q\).}
Lemma~\ref{lemma:core_bridge} gives
\[
\psi(z_l^i)-\psi(z_K^i)
\ge
\frac1C\,\mathrm{RL}(x_{\llbracket l,K\rrbracket})
-\sum_{k=l}^{K-1}E_k
-\hat C\,\alpha_l^{p_i}.
\]
Since \(x_K\in \mathcal N(i,\alpha_K)\), Lemma~\ref{lemma:endpoint_conversion} yields
\[
|\psi(z_K^i)-\psi(z_K)|
\le C_{\mathrm{end}}\alpha_K^{\beta_i(1-\theta)}
\le C_{\mathrm{end}}\alpha_l^{\beta_i(1-\theta)}
\le C_{\mathrm{end}}\alpha_l^{\underline\beta(1-\theta)}.
\]
Combining the last two inequalities and using again \(\alpha_l^{p_i}\le \alpha_l^{\underline\beta(1-\theta)}\), we obtain \eqref{eq:terminal_RM_lower} also in this case.
\end{proof}
\subsection{Proof of Theorem \ref{thm:diameter}}

\label{sec:final_proof}
We are now ready to estimate the diameter of subgradient sequences. Let $f:\mathbb R^n\to \mathbb R$ be locally Lipschitz definable and $X\subset \mathbb R^n$ be compact definable. Applying Corollary \ref{cor:y_k^i} gives a stratification \(\mathcal{M}:= \{M_1,\dots,M_{\overline{T}}\}\) of \(X\), and constants including \( c_i,\beta_i,\gamma_i>0\) and $\hat\alpha \in (0,1]$; we use the same notation and constants introduced at the beginning of Section \ref{sec:RM_def}.

Let $\overline C>0$ be the constant that appears in Corollary \ref{cor:lm->lm+1} and Lemma \ref{lemma:RM_arbitrary_interval}, and $Q:= 4\overline C CT>0$. Also, we let $\bar{\alpha}_Q>0$ be given by Lemma \ref{lemma:RM_or_half}. Fix a subgradient sequence \((x_k)_{k\in\mathbb N}\) with step sizes
\[
0<\alpha_{K-1}\le \cdots \le \alpha_0\le \bar\alpha,
\qquad
x_{\llbracket 0,K\rrbracket}\subset X\cap [|f|\le \epsilon],
\]
where $\bar{\alpha} \le \min\{\hat \alpha, \bar{\alpha}_Q\}$ is an upper bound on step sizes such that \eqref{eq:standing_alpha_small} holds.  For any such sequence, we may define its length relative to the stratification $\mathcal M$ following the construction in Definition \ref{def:RL}, and all the results in Sections \ref{sec:RM_def} and \ref{sec:bound_rm} hold true.

We write 
\begin{equation}\label{eq:final_beta}
    \beta:=\min\{\hat\beta,\underline\beta(1-\theta)\}>0.
\end{equation}

We distinguish two cases according to whether the sequence ever approaches a non-open stratum. If \(I_C\neq \emptyset\), we set
\[
l^-:=l_0=\min I_C,
\qquad
l^+:=l_{\overline m}=\max \mathcal L.
\]
If \(I_C=\emptyset\), we simply set
\[
l^-=l^+=K.
\]
With this convention, we will estimate the relative length \(\mathrm{RL}(x_{\llbracket 0,K\rrbracket})\) by splitting the interval at \(l^-\) and \(l^+\), and only afterward convert the resulting bound into a diameter estimate. The case \(I_C=\emptyset\) is therefore included as a degenerate case of the same decomposition. In the argument below, we focus on the nontrivial case \(I_C\neq \emptyset\); at the end of Step~4 we verify that the same final bound also covers the case \(I_C=\emptyset\).

\noindent\underline{Step 1: the initial piece \(\llbracket 0,l^-\rrbracket\).}
Since \([0,l^-)\cap I_C=\emptyset\), Fact~\ref{fact:IC_away} yields an open stratum \(M_{i_-}\), with
\(i_-\in \llbracket T+1,\overline T\rrbracket\), such that
\[
x_k\in \mathcal N(i_-,\alpha_k)
\qquad \forall k\in \llbracket 0,l^-\rrbracket.
\]
Hence
\begin{equation} \label{eq:RL_initial}
    \mathrm{RL}(x_{\llbracket 0,l^-\rrbracket})
=
\sum_{k=0}^{l^--1}|x_{k+1}-x_k| \le C\Bigl(\psi(z_0)-\psi(z_{l^-})+\sum_{k=0}^{l^--1}E_k+\alpha_0\Bigr)
\end{equation}
by applying Corollary~\ref{cor:y_k^i} on \(\llbracket 0,l^-\rrbracket\).

\smallskip
\noindent\underline{Step 2: the middle piece \(\llbracket l^-,l^+\rrbracket\).}
If \(l^-=l^+\), then
\(
\mathrm{RL}(x_{\llbracket l^-,l^+\rrbracket})=0.
\)
Assume now that \(l^-<l^+\). Since
\[
u:=\bigl|G([l^-,l^+]\cap I_C)\bigr|\le |\llbracket 1,T\rrbracket| = T,
\]
the choice \(Q=4\overline C CT\) yield $(2\overline C Cu)/Q \le 1/2$. Therefore, by Lemma~\ref{lemma:RM_arbitrary_interval},
\begin{equation}\label{eq:RL_middle}
\mathrm{RL}(x_{\llbracket l^-,l^+\rrbracket})
\le
2C\Biggl(
\psi\bigl(z_{l^-}^{G(l^-)}\bigr)-\psi\bigl(z_{l^+}^{G(l^+)}\bigr)
+\sum_{k=l^-}^{l^+-1}E_k
+2\overline C T\,\alpha_{l^-}^{\underline p}
\sum_{j=0}^{\left\lfloor\log_2(\alpha_{l^-}/\alpha_{l^+})\right\rfloor}
2^{-j\underline p}
\Biggr).
\end{equation}
Note that \eqref{eq:RL_middle} holds true regardless of whether $l^-=l^+$ or $l^-<l^+$.

\smallskip
\noindent\underline{Step 3: the terminal piece \(\llbracket l^+,K\rrbracket\).} Lemma~\ref{lemma:RM_terminal_piece} yields
\begin{equation}\label{eq:RL_terminal}
\mathrm{RL}(x_{\llbracket l^+,K\rrbracket})
\le
C_{\mathrm{tail},1}\Bigl(
\psi\bigl(z_{l^+}^{G(l^+)}\bigr)-\psi(z_K)\Bigr)
+C_{\mathrm{tail},2}\left(\sum_{k=l^+}^{K-1}E_k
+\alpha_{l^+}^{\underline\beta(1-\theta)}
\right),
\end{equation}
for some $C_{\mathrm{tail},1},C_{\mathrm{tail},2}>0$. Again, the above bound remains valid when $l^+ = K$.

\smallskip
\noindent\underline{Step 4: combine the three relative-length bounds.}
By the definition of \(\mathrm{RL}\) on general intervals given in the third case of Definition \ref{def:RL},
\[
\mathrm{RL}(x_{\llbracket 0,K\rrbracket})
=
\mathrm{RL}(x_{\llbracket 0,l^-\rrbracket})
+
\mathrm{RL}(x_{\llbracket l^-,l^+\rrbracket})
+
\mathrm{RL}(x_{\llbracket l^+,K\rrbracket}).
\]
We multiply the right hand sides of \eqref{eq:RL_initial}, \eqref{eq:RL_middle}, and \eqref{eq:RL_terminal} so that the coefficients in front of the $\psi$ terms become $C_1 := \max\{2C,C_{\mathrm{tail},1}\}>0$. We then sum them up and obtain
\begin{align*}
\mathrm{RL}(x_{\llbracket 0,K\rrbracket})
\le{}\;&
C_1\Bigl(
\psi(z_0)-\psi(z_{l^-})
+\psi\bigl(z_{l^-}^{G(l^-)}\bigr)-\psi\bigl(z_{l^+}^{G(l^+)}\bigr)
+\psi\bigl(z_{l^+}^{G(l^+)}\bigr)-\psi(z_K)
\Bigr)\\
&+C_2\sum_{k=0}^{K-1}E_k
+C_2\Bigl(
\alpha_0
+\alpha_{l^+}^{\underline\beta(1-\theta)}
+\alpha_{l^-}^{\underline p}
\sum_{j=0}^{\left\lfloor\log_2(\alpha_{l^-}/\alpha_{l^+})\right\rfloor}
2^{-j\underline p}
\Bigr),
\end{align*}
some \(C_2 \ge C_1\). Since the \(\psi\)-terms telescope except at the index \(l^-\),
\[
\psi(z_0)-\psi(z_{l^-})
+\psi\bigl(z_{l^-}^{G(l^-)}\bigr)-\psi\bigl(z_{l^+}^{G(l^+)}\bigr)
+\psi\bigl(z_{l^+}^{G(l^+)}\bigr)-\psi(z_K)
=
\psi(z_0)-\psi(z_K)
+\psi\bigl(z_{l^-}^{G(l^-)}\bigr)-\psi(z_{l^-}).
\]
Since \(l^-\in I_C\), combining Fact \ref{fact:IC_close} and Lemma~\ref{lemma:endpoint_conversion} gives
\[
\bigl|\psi\bigl(z_{l^-}^{G(l^-)}\bigr)-\psi(z_{l^-})\bigr|
\le
C_{\mathrm{end}}\alpha_{l^-}^{p_{G(l^-)}} \le C_{\mathrm{end}}\alpha_{l^-}^{\underline p}.
\]
Therefore,
\begin{align}
\mathrm{RL}(x_{\llbracket 0,K\rrbracket})
\le{}\;&
C_1\bigl(\psi(z_0)-\psi(z_K)\bigr)
+C_3\sum_{k=0}^{K-1}E_k \notag\\
&+C_3\Bigl(
\alpha_0
+\alpha_{l^-}^{\underline p}
+\alpha_{l^+}^{\underline\beta(1-\theta)}
+\alpha_{l^-}^{\underline p}
\sum_{j=0}^{\left\lfloor\log_2(\alpha_{l^-}/\alpha_{l^+})\right\rfloor}
2^{-j\underline p}
\Bigr), \label{eq:RL_total_pre}
\end{align}
where $C_3 := C_2 + C_{\mathrm{end}}>0$.

We next show that the above estimate holds for the case of $I_C = \emptyset$ as well. Indeed, in this case, by \eqref{eq:RL_initial}, we have
\begin{align*}
    \mathrm{RL}(x_{\llbracket 0,K\rrbracket})
&=
\sum_{k=0}^{K-1}|x_{k+1}-x_k|\\
&\le C\Bigl(\psi(z_0)-\psi(z_{K})+\sum_{k=0}^{K-1}E_k+\alpha_0\Bigr)\\
&\le C_1\Bigl(\psi(z_0)-\psi(z_{K})+\sum_{k=0}^{K-1}E_k+\alpha_0\Bigr)\\
&\le C_1\bigl(\psi(z_0)-\psi(z_K)\bigr)
+C_3\sum_{k=0}^{K-1}E_k + C_3 \alpha_0
\end{align*}
as $C_1\ge 2c >c$ and $C_3 \ge C_1$.

\smallskip
\noindent\underline{Step 5: convert relative length to diameter.}
Applying Lemma~\ref{lemma:diameter_by_RM} with \(k_1=0\) and \(k_2=K\) to \eqref{eq:RL_total_pre} and using \(\alpha_{l^+} \le \alpha_{l^-}\le \alpha_0\), we obtain
\begin{align}
\mathrm{diam}(x_{\llbracket 0,K\rrbracket})
\le{}\;&
C_1\bigl(\psi(z_0)-\psi(z_K)\bigr)
+C_3\sum_{k=0}^{K-1}E_k \notag\\
&+C_3\Bigl(
\alpha_0
+\alpha_0^{\underline p}
+\alpha_0^{\underline\beta}
+\alpha_0^{\underline\beta(1-\theta)}
+\alpha_0^{\underline p}
\sum_{j=0}^{\left\lfloor\log_2(\alpha_{l^-}/\alpha_{l^+})\right\rfloor}
2^{-j\underline p}
\Bigr) \label{eq:diam_pre_z_star}
\end{align}
after replacing $C_3$ by $\max\{C_3,4\bar{c}\}$. Since \(\beta \le \underline\beta(1-\theta) <\underline p<1 \) by \eqref{eq:relation_beta_p} and \eqref{eq:final_beta}, while \(\alpha_0\le 1\), each of the algebraic error terms on the right-hand side is bounded by a constant multiple of \(\alpha_0^{\beta}\). Moreover,
\[
\sum_{j=0}^{\left\lfloor\log_2(\alpha_{l^-}/\alpha_{l^+})\right\rfloor}
2^{-j\underline p}
\le
\sum_{j=0}^{\infty}2^{-j\underline p}
=
\frac{1}{1-2^{-\underline p}},
\]
so the last term in \eqref{eq:diam_pre_z_star} is also bounded by a constant multiple of \(\alpha_0^{\beta}\). Hence there exists \(C_4>0\) such that
\begin{equation}\label{eq:diam_pre_final_star}
\mathrm{diam}(x_{\llbracket 0,K\rrbracket})
\le
C_1\bigl(\psi(z_0)-\psi(z_K)\bigr)
+
C_4\Bigl(\alpha_0^{\beta}+\sum_{k=0}^{K-1}E_k\Bigr).
\end{equation}

It remains to rewrite the right-hand side in terms of \(f(x_0)\) and \(f(x_K)\). Since
\[
z_0=f(x_0)+g_0,
\qquad
z_K=f(x_K),
\qquad
g_0=C\sum_{k=0}^{K-1}\alpha_k^{1+\hat\beta},
\]
and \(\psi(t)=\operatorname{sgn}(t)|t|^{1-\theta}\), we have
\begin{align*}
    \psi(z_0)-\psi(z_K)
&\le
\psi(f(x_0))-\psi(f(x_K))
+2\psi(g_0)\\
&\le \operatorname{sgn}(f(x_0))|f(x_0)|^{1-\theta}
-
\operatorname{sgn}(f(x_K))|f(x_K)|^{1-\theta}
+
2c^{1-\theta}\Bigl(\sum_{k=0}^{K-1}\alpha_k^{1+\hat\beta}\Bigr)^{1-\theta},
\end{align*}
where the first inequality follows from Fact \ref{fact:psi_holder}. Also,
\[
\sum_{k=0}^{K-1}E_k
=
\sum_{k=0}^{K-1}\alpha_k^{1+\hat\beta}
+
c^\theta\sum_{k=0}^{K-1}\alpha_k
\Bigl(\sum_{j=k}^{K-1}\alpha_j^{1+\hat\beta}\Bigr)^\theta.
\]

Finally, since \(\beta\le \hat\beta\) and \(\alpha_k\le 1\),
\[
\alpha_k^{1+\hat\beta}\le \alpha_k^{1+\beta},
\qquad
\Bigl(\sum_{k=0}^{K-1}\alpha_k^{1+\hat\beta}\Bigr)^{1-\theta}
\le
\Bigl(\sum_{k=0}^{K-1}\alpha_k^{1+\beta}\Bigr)^{1-\theta},
\]
and
\[
\sum_{k=0}^{K-1}\alpha_k
\Bigl(\sum_{j=k}^{K-1}\alpha_j^{1+\hat\beta}\Bigr)^\theta
\le
\sum_{k=0}^{K-1}\alpha_k
\Bigl(\sum_{j=k}^{K-1}\alpha_j^{1+\beta}\Bigr)^\theta.
\]
Substituting these bounds into \eqref{eq:diam_pre_final_star}, we obtain
\begin{align*}
\mathrm{diam}(x_{\llbracket 0,K\rrbracket})
\le{}\;&
\varsigma_1\Bigl(
\operatorname{sgn}(f(x_0))|f(x_0)|^{1-\theta}
-
\operatorname{sgn}(f(x_K))|f(x_K)|^{1-\theta}
\Bigr)\\
&+\varsigma_2\Biggl(
\alpha_0^{\beta}
+\sum_{k=0}^{K-1}\alpha_k^{1+\beta}
+\Bigl(\sum_{k=0}^{K-1}\alpha_k^{1+\beta}\Bigr)^{1-\theta}
+\sum_{k=0}^{K-1}\alpha_k
\Bigl(\sum_{j=k}^{K-1}\alpha_j^{1+\beta}\Bigr)^\theta
\Biggr).
\end{align*}
with $\varsigma_1:= C_1$ and some $\varsigma_2>0$. \qed

\section*{Acknowledgments}
We are grateful to the reviewers and the associate editor for their very helpful and insightful comments, which have been extremely useful in improving the paper.
\appendix
\section{Lipschitz stratification}\label{sec:Lipschitz}

In this appendix, we recall the definition of Lipschitz stratification. We adhere the notations in \cite{nguyen2016lipschitz}. Given a definable stratification $\Sigma$ of $X$, denote by $X^i$ the $i$-th skeleton of the stratification $\Sigma$, which is the union of all the strata of dimension less than or equal to $i$. We thus get a sequence
\[
X  =  X^d  \supset  X^{d-1}  \supset  \cdots  \supset  X^l
\]
and such that each difference
$\mathring{X}^i :=  X^i \setminus X^{i-1}$
is an $i$-dimensional definable submanifold of $\mathbb{R}^n$ or empty. The strata coincide with the connected components of $\mathring{X}^i$. 

Let $c>1$ be a fixed constant. A \emph{$c$-chain} of $q \in \mathring{X}^j$ is a strictly decreasing sequence of indices
\[
j = j_1 > j_2 > \cdots > j_r = \ell
\]
and a corresponding sequence of points $q_{j_s} \in \mathring{X}^{j_s}$ such that $q_{j_1} = q$ and each $j_s$ is the greatest integer for which
\[
d(q,X^k) \ge 2 c^2  d(q,X^{j_s})
\quad \text{for all } k < j_s,
\quad \text{and}
\quad
|q - q_{j_s}| \le c   d(q,X^{j_s}).
\]
For each point $q \in \mathring{X}^j$, let $P_q : \mathbb{R}^n \to T_{\mathring{X}^j}(q) $ and $P_q^\perp = \mathrm{Id} - P_q : \mathbb{R}^n \to N_{\mathring{X}^j}(q)$ respectively denote the orthogonal projections from $\mathbb{R}^n$ onto the tangent and normal spaces to $\mathring{X}^j$.

\begin{definition}[Lipschitz stratification \cite{mostowski1985lipschitz}]\label{def:LipschitzStrat}
A stratification 
$\mathcal{X}  =  \{X_i\}_{i=\ell}^d $
of a definable set $X$ is said to be a Lipschitz stratification if for every $c>1$ there is some $C>0$ such that for every $c$-chain $\{q=q_{j_1}, \dots, q_{j_r}\}$ we have for every $1\le k\le r$,
\[
|P_{q_{j_1}}^\perp   P_{q_{j_2}} \cdots P_{q_{j_k}}|
 \le  C  \frac{\lvert q - q_{j_2}\rvert}{d\bigl(q,X^{j_{k-1}}\bigr)},
\]
and
\[
|(P_{q} - P_{q'}) P_{q_{j_2}} \cdots P_{q_{j_k}}|
 \le  C  \frac{\lvert q - q'\rvert}{d\bigl(q,X^{j_{k-1}}\bigr)},
\]
in particular,
\[
| P_q - P_{q'} |  \le  C \frac{\lvert q - q'\rvert}{d\bigl(q,X^{j-1}\bigr)}
\quad(\text{set $d\bigl(q,X^{\ell-1}\bigr)=1$ for $q\in X$}).
\]
\end{definition}

\bibliographystyle{alpha}
\bibliography{mybib}
\end{document}